\documentclass[a4paper,USenglish,cleveref, autoref, numberwithinsect, thm-restate]{lipics-v2021}

\usepackage{algorithm}
\usepackage{algpseudocode}
\usepackage{amsmath,amsthm,amssymb}
\usepackage{tikz}
\usepackage{subcaption}
\usepackage{graphicx}
\usetikzlibrary{calc}
\usetikzlibrary{decorations.pathmorphing, decorations.pathreplacing, decorations.shapes} 
\usetikzlibrary{shapes, fit, positioning}
\usetikzlibrary{backgrounds}
\usepackage{paralist} 
\usepackage{amsmath}  
\usepackage{amsthm}
\usepackage{amsfonts}
\usepackage{complexity}
\usepackage{float}
\usepackage{todonotes}

 \usepackage{apxproof}
\ifx\proof\inlineproof
  \def\apxmark#1{}
  \def\apxcmark#1{}
  \long\def\onlyapx#1{}
\else
  \def\apxxxmark#1#2{\hyperref[#1]{{\sf#2$\!\!$(*)\,}}}
  \def\apxmark#1{\apxxxmark{#1}{\bfseries}}
  \def\apxcmark#1{\apxxxmark{#1}{$\!\!$}}
  \long\def\onlyapx#1{#1}
\fi
\newtheoremrep{theorem2}[theorem]{Theorem}
\newtheoremrep{proposition2}[theorem]{Proposition}
\newtheoremrep{corollary2}[theorem]{Corollary}
\newtheoremrep{lemma2}[theorem]{Lemma}
\newtheoremrep{observation2}[theorem]{Observation}
\newtheoremrep{definition2}[theorem]{Definition}

\pdfoutput=1 
\hideLIPIcs  
\nolinenumbers

\title{Conflict-Free Coloring Planar Graphs with 4 Colors}

\author{Petr Hlin{\v e}n\'y}{Masaryk University, Brno, Czech Republic}{hlineny@fi.muni.cz}{https://orcid.org/0000-0003-2125-1514}{}
\author{Luk\'a\v{s} M\'alik}{Masaryk University, Brno, Czech Republic}{514189@mail.muni.cz}{https://orcid.org/0009-0001-7961-849X}{}

\authorrunning{P.~Hlin{\v e}n\'y and L.~M\'alik} 

\Copyright{Petr Hlin{\v e}n\'y and Luk\'a\v{s} M\'alik} 



\funding{Both authors supported by the project 26-21334S of the Czech Science Foundation.}

\EventEditors{John Q. Open and Joan R. Access}
\EventNoEds{2}
\EventLongTitle{42nd Conference on Very Important Topics (CVIT 2016)}
\EventShortTitle{CVIT 2016}
\EventAcronym{CVIT}
\EventYear{2016}
\EventDate{December 24--27, 2016}
\EventLocation{Little Whinging, United Kingdom}
\EventLogo{}
\SeriesVolume{42}
\ArticleNo{23}

\begin{document}
\maketitle

\begin{abstract}
We efficiently conflict-free color every planar graph with $4$ colors.
An (open-neighborhood) \emph{conflict-free coloring} assigns colors to vertices in a way that every vertex $v$ has a neighbor $w$ such that the color of $w$ is distinct from the colors of the other neighbors of~$v$ (i.e., the color of $w$ is unique in the open neighborhood of~$v$).
A previous best upper bound on the conflict-free chromatic number of planar graphs was~$5$, and it is known that $4$ colors are sometimes necessary. Deciding whether, e.g., a planar graph admits a conflict-free coloring with $3$ colors is \NP-complete.
Our approach uses a refined variant of the classical Gallai--Edmonds decomposition and the Four Color Theorem. In fact, our result is equivalent to the Four Color~Theorem.
\end{abstract}

\section{Introduction}

Conflict-free coloring was introduced in the early 2000s in geometric and combinatorial terms, motivated by applications such as frequency assignment in wireless networks~\cite{Even2003CF, smorodinsky2012conflictfreecoloringapplications}.
In this setting, one seeks a coloring such that every relevant region contains a uniquely colored element.
This notion has since evolved into a rich area of research spanning combinatorics, computational geometry, and graph theory.

A natural graph-theoretic variant of the problem considers the neighborhoods of vertices as the relevant regions. 
An \emph{open-neighborhood conflict-free coloring} (shortly a CFON coloring, as used e.g.~in~\cite{Bhyravarapu2020WG}) of a graph $G$ is an assignment of colors to vertices of $G$ such that 
for every vertex $v$ having at least one neighbor, there exists an assigned color that is unique (appears exactly once) in its open neighborhood $N(v)$. 
Note that we always assume $v\not\in N(v)$, that is, we ignore possible loops in~$G$, and parallel edges do not matter.
Several variants of this notion have been studied, including the distinction between \emph{partial} colorings (where only a subset of vertices is colored) and \emph{full} colorings (where each vertex of $G$ is assigned a color), 
and the more restricted \emph{proper} version in which adjacent vertices additionally must receive distinct colors.
In this paper, we study full CFON colorings which are not necessarily~proper.

\medskip
For a historical overview, we briefly review this and other variants of conflict-free colorings on graphs in the literature.
An early work of Abel et al.~\cite{AbelCF} addressed combinatorial bounds for general graphs. Later algorithmic and parameterized aspects of conflict-free colorings with respect to neighborhoods were investigated e.g.\ by Bodlaender, Kolay, and Pieterse~\cite{Bodlaender2021SIDMA}. 

Specifically for planar graphs, the line of research of combinatorial bounds for CFON colorings begins with the aforementioned Abel et al.~\cite{AbelCF}, who established the first general bounds for conflict-free colorings with respect to neighborhoods.
Considerable effort has been devoted to improving these bounds for planar graphs. 
Next, Bhyravarapu and Kalyanasundaram~\cite{Bhyravarapu2020WG} showed that every planar graph admits a partial CFON coloring with at most five colors, and that every outerplanar graph admits such a coloring with at most four colors. Their work also developed general bounds for CFON colorings in terms of several structural graph parameters; these results were later extended in joint work with Mathew~\cite{Sriram}.

A breakthrough for the planar partial CFON case was obtained by Huang, Guo, and Yuan~\cite{Huang}, who proved that every connected graph that is minor-$k$-colorable (that is, every minor of it is properly $k$-colorable) admits a partial CFON coloring with at most $k$ colors. 
As a consequence, every planar graph admits a partial CFON coloring with at most four colors, and so a full CFON coloring with at most five colors by assigning a dummy fifth color to every uncolored vertex.

Note also that a lower bound of $4$ colors in the worst case on planar graphs \cite{AbelCF} is nearly trivial; 
take the $1$-subdivision $G_1$ of a graph $G$ (such as of planar $G=K_4$), then for every edge $uv$ of $G$ the vertices $u$ and $v$ must receive distinct colors in a CFON coloring of $G_1$, or the vertex subdividing $uv$ does not see a unique color, and so at least $\chi(G)$ colors are needed.
The question of whether the maximum CFON chromatic number in the class of planar graphs is $4$ or $5$ remained a central open problem in the area (e.g., \cite[Conjecture~9.1]{Fabrici}).

\medskip
In parallel, several related variants of conflict-free coloring have been studied. 
In particular, \emph{proper conflict-free colorings}, in which adjacent vertices must receive distinct colors, have received attention in recent years. 
Fabrici et al.~\cite{Fabrici} showed, among other results, that every planar graph admits a proper CFON coloring with at most eight colors, and that there exist planar graphs that do not admit such a coloring with five colors. 
Further refinements for planar graphs and related classes include bounds under additional structural restrictions, such as girth or forbidden cycles~\cite{Eko,wang2025proper}. 
Recent works have also investigated asymptotic and structural bounds for proper conflict-free colorings in graphs of large maximum degree~\cite{cranston2024proper,liu2025asymptotically}.

Beyond planar graphs, conflict-free coloring has been studied in a variety of directions, including list variants, algorithmic aspects, and complexity questions~\cite{Gargano,GuptaMathew2026Choosability}. 
In addition, variants motivated by combinatorial games and dynamic settings have recently been explored~\cite{pantoja2025conflict}. 
Altogether, this highlights the richness of the area and the diversity of techniques~involved. 

\subparagraph*{New contribution. }
While leaving the definitions and technical details for later, we state:
\begin{theorem}\label{thm:main}
Every planar graph admits a full open-neighborhood conflict-free (CFON) coloring with at most four colors, and such a coloring can be computed in polynomial time.
\end{theorem}

Closing the gap between $4$ and $5$ for the maximum CFON chromatic number in the class of planar graphs has required the development of significant new methods and tools.
Our approach is based on a localized analysis of the Gallai--Edmonds decomposition for maximum matchings in a graph. 
The Gallai--Edmonds structure~\cite{Gallai1964,edmonds1965paths} is a classical tool in matching theory~\cite{Plummer}, and we develop a refinement that focuses on individual factor-critical components and their interaction with unmatched vertices. 
While the connection between CFON colorings and matchings in a graph (\Cref{obs:perfect}, informally saying that the matched vertices are easily CFON colorable) has been known before~\cite{Huang}, 
to the best of our knowledge, such a localized use of the Gallai--Edmonds decomposition has not been previously exploited in the context of conflict-free coloring.

The new perspective allows us to derive some structural properties of planar graphs, in particular, that there exists a maximum matching in which every unmatched vertex has degree at most five. 
This structural insight serves as the starting point for our analysis and enables a systematic treatment of local configurations (Figures~\ref{fig:matching_one}--\ref{fig:neighbors_five}), ultimately leading to a four-color bound for the CFON colorings of planar graphs.
Our arguments also use the classical Four Color Theorem and, in turn, \Cref{thm:main} implies the Four Color Theorem via the lower-bound reduction with a $1$-subdivision sketched above.

\subparagraph*{Paper organization. }
We introduce the necessary graph coloring and matching tools in \Cref{sec:prelimin}.
In \Cref{sec:switchable} we design our tools for a localized analysis of the Gallai--Edmonds decomposition.
In \Cref{sec:classification} we prove our first cornerstone result -- \Cref{thm:deg5} about degrees of unmatched vertices, and use it to classify the local cases which have to be handled in the proof of our main result.
\Cref{sec:controlled} solves the majority of the local cases with one universal argument -- \Cref{lem:controlled}, based on proper coloring of an augmented matching-contracted graph and Kempe chains.
\Cref{sec:density} brings an additional planarity-related tool -- \Cref{lemma_gadgets}, which is then used in \Cref{sec:difficult2} to finish off the remaining two difficult local cases.
\Cref{sec:completing} then summarizes the proof of the main result -- \Cref{thm:main}, and \Cref{sec:algorithm} subsequently deals with the mostly straightforward algorithmic details.

\onlyapx{%
To improve overall readability and focus on the core ideas, many lengthy technical parts of the arguments and proofs are left for the Appendix.
}

\section{Preliminaries}\label{sec:prelimin}

We work with finite and simple graphs.
For a vertex $v$ of $G$, we denote by $N_G(v)$ (shortly~$N(v)$) its open neighborhood -- the set of all vertices adjacent to $v$ (and hence excluding $v$ itself since we have simple graphs), and by $\deg_G(v)=|N_G(v)|$ its degree. 
For a set $X\subseteq V(G)$, we denote by $N_G(X):=\bigcup_{v\in X}N_G(v)$.

\subparagraph{Colorings.}
A \emph{coloring} of a graph $G$ by/with $k$ colors (a \emph{$k$-coloring}) is an arbitrary mapping $c:V(G)\to \{1,2,\dots,k\}$.
The parameter $k$ is often implicit and we may skip it.

A \emph{proper coloring} of a graph $G$ is a coloring $c$ such that for every edge $uv\in E(G)$ we have~$c(u)\not=c(v)$.
(This is usually called just a `coloring', but we want to distinguish the term from conflict-free colorings which are not required to be proper.)
One of the most fundamental results of graph theory is the Four Color Theorem:

\begin{theorem}[Appel and Haken~\cite{appel1977every1,appel1977every2}, Robertson, Sanders, Seymour and Thomas~\cite{DBLP:conf/stoc/RobertsonSST96,DBLP:journals/jct/RobertsonSST97}]
\label{thm:FCT}
Every loopless planar graph can be properly colored with $4$ colors, and this coloring can be found in quadratic time.
\end{theorem}

We will also use the standard tool of Kempe chains. For two colors $i$ and $j$ in a properly colored graph~$G$, a \emph{$(i,j)$-Kempe chain} is a connected component of the subgraph induced by the vertices colored $i$ and~$j$.

In our paper, we study the following variant of colorings:
\begin{definition}[Open-neighborhood conflict-free coloring]\label{def:CFON}
A coloring $c$ is an \emph{open-neighbor\-hood conflict-free coloring} (shortly CFON coloring) of a graph $G$ if for every vertex $v\in V(G)$ with $N(v)\neq\varnothing$, there exists a color $i$ such that
$|c^{-1}(i)\cap N(v)|=1$.

Equivalently, for every vertex $v$ with nonempty neighborhood, there is a neighbor with a color that appears exactly once among the neighbors (is \emph{unique}).

The minimum number of colors in a CFON coloring of $G$ is called the \emph{conflict-free chromatic number} of $G$ and is denoted by $\chi_{\mathrm{CF}}(G)$.
\end{definition}

\subparagraph{Matching-theoretic tools.}
We use standard notions of matching theory~\cite{Plummer}. 
Given a \emph{matching} $M$ in a graph, that is a set of mutually non-adjacent edges, a vertex is \emph{matched} (with respect to $M$) if it is incident to an edge of $M$, and \emph{unmatched} otherwise.
A matching is \emph{perfect} if there is no unmatched vertex.
A vertex $v\in V(G)$ is called \emph{essential} if it is matched in every maximum matching of $G$, and \emph{inessential} otherwise (i.e., it is unmatched in at least one maximum matching).

\begin{definition}[Gallai--Edmonds decomposition~\cite{Gallai1964,edmonds1965paths}]\label{def:GallaiEdmonds}
For a graph $G$, let
\begin{align*}
D(G) &= \{v\in V(G) : v \text{ is inessential}\},\\
A(G) &= N_G(D(G)) \setminus D(G),\\
C(G) &= V(G) \setminus (A(G) \cup D(G)).
\end{align*}
So, the sets $D(G)$, $A(G)$, and $C(G)$ form a partition of $V(G)$.
\end{definition}

\begin{claim}[e.g., \cite{Plummer}]\label{clm:GallaiEdmonds}
The following properties hold for a Gallai--Edmonds decomposition.
\begin{enumerate}[a)]
    \item Each connected component $H$ of the induced subgraph $G[D(G)]$ is \emph{factor-critical} (that is, for every vertex $v \in V(H)$, the graph $H-v$ has a perfect matching).
    \item The induced subgraph $G[C(G)]$ has a perfect matching.
    \item There are no edges between $C(G)$ and $D(G)$ in~$G$.
    \item Every vertex of $A(G)$ has a neighbor in $D(G)$. Moreover, in every maximum matching $M$ of $G$, each vertex of $A(G)$ is incident to an edge of $M$ whose other endpoint lies in $D(G)$, and these endpoints lie in pairwise distinct connected components of $G[D(G)]$.
\label{it:GE-A-pairwise_disjoint}
    \item If $G[D(G)]$ has $k$ connected components, then for every maximum matching $M$ of $G$ there are $k-|A(G)|$ connected components of $G[D(G)]$ which are not adjacent by an edge of $M$ to~$A(G)$,
	and so there are exactly $k-|A(G)|$ unmatched vertices in~$M$.	
\label{it:k-A-unmatched}
    \item For every nonempty set $X \subseteq A(G)$, the set $N_G(X)\cap D(G)$ intersects at least $|X|+1$ distinct connected components of $G[D(G)]$.
\end{enumerate}
\end{claim}

\begin{theorem}[Edmonds' blossom algorithm \cite{edmonds1965paths}]\label{thm:blossom}
A matching of maximum size in a graph can be found in polynomial time.
\end{theorem}

\subparagraph{Perfect matching implies a conflict-free coloring via contraction.}

Let $F \subseteq E(G)$ be an arbitrary set of edges (however, we will usually have $F=M$ for a matching~$M$). We define the \emph{contracted graph} $H = H(G,F)$ of $G$ as follows.
Every connected component $C$ of the spanning subgraph $(V(G),F)$ is contracted down into one \emph{supervertex} of~$H$, which we denote by $[v_1\ldots v_m]$ where $V(C)=\{v_1,\ldots,v_m\}$,
while suppressing loops and multiple edges so that $H$ is simple.
So, two supervertices are adjacent in $H$ if and only if there exists an edge between their corresponding vertex sets in~$G$. 
In particular, a vertex $x$ not incident to any edge of $F$ is kept as its own supervertex $[x]$ in~$H$.

A coloring $c$ of $H$ \emph{lifts} to a coloring $c'$ of $G$ if every vertex $v$ of $G$ which belongs to a supervertex $w$ of $H$ receives the color $c'(v):=c(w)$.

\begin{observation}
\label{obs:perfect}
If $M$ is a perfect matching of $G$, then any proper coloring of the contracted graph $H=H(G,M)$ lifts to a CFON coloring~of~$G$.
\end{observation}

\begin{proof}
Let $c$ be a proper coloring of $H$ that lifts to the coloring $c'$ of~$G$. 
Pick any vertex $v \in V(G)$, and let $u$ be such that $uv\in M$ (since $M$ is a perfect matching).
Since $u \in N(v)$, the color $c'(u) = c'(v) = c([uv])$ appears at least once in $N(v)$. We claim it appears exactly once. 
Let $w \in N(v)\setminus\{u\}$ and $w'\in V(G)$ be such that~$ww'\in M$. Note that $uv\not=ww'$.
Since the vertices $v$ and $w$ are adjacent in $G$, their supervertices $[uv]$ and $[ww']$ are adjacent in $H$. 
Consequently, $c([uv])\not=c([ww'])$ in a proper coloring, and hence $c'(w)=c([ww']) \neq c'(u)$.
\end{proof}

\Cref{obs:perfect} shows that when $G$ has a perfect matching, conflict-free colorings can be obtained directly from proper colorings of the contracted graph $H$. 
Since contractions of edges preserve planarity, we can use \Cref{thm:FCT}.
Consequently, the main difficulty arises from vertices that are unmatched by a maximum matching. The remainder of the paper focuses on controlling these unmatched vertices by enforcing additional color constraints in the contracted graph~$H$ and/or by locally modifying the matching.

Similar matching-based ideas have already appeared in the study of conflict-free colorings. In particular, Huang et al.~\cite{Huang} use maximum matchings and the Gallai--Edmonds decomposition in their analysis of partial conflict-free colorings of planar graphs.
We develop this perspective further; we use a maximum matching as a global structural framework and introduce additional tools, such as alternating-path reachability and the auxiliary graphs~$G_x$ (see \Cref{sec:switchable}), which allow us to localize the analysis to individual factor-critical components.

\section{Switchable vertices and the subgraph $G_x$}\label{sec:switchable}

\begin{toappendix}
\label{apx:start}
\end{toappendix}
\onlyapx{%
In this section, we develop the tools which will be used to analyze the structure of unmatched vertices in planar graphs. 
We move most of the technical parts of the section to Appendix and mark the respective statements with~~~\apxcmark{apx:start}.
}

\begin{toappendix}
\label{apx:switchable}
In this section, we develop the tools which will be used to analyze the structure of unmatched vertices in planar graphs. 
In particular, we show that every unmatched vertex either has degree at most~5 or is connected by an even-length alternating path to a vertex of degree at most~5. 
We also introduce an auxiliary subgraph $G_x$, which will be used extensively in the sequel to capture local structural constraints around an unmatched vertex $x$.
\end{toappendix}

Let $M$ be a fixed maximum matching of $G$. An \emph{$M$-alternating path} is a path whose edges alternate between edges in $M$ and edges not in $M$. 
We are now going to define several notions which depend on the choice of a matching $M$ in our graph $G$,
but we will for simplicity skip an explicit reference to $M$ whenever the considered matching $M$ is clear from the context.

\begin{definition}[Vertex switchable from $x$; set $S(x)$]
Let $G$ be a graph, $M$ a matching and $x$ be a vertex unmatched by $M$. A vertex $v$ is \emph{switchable from $x$} if there exists an $M$-alternating path of even length starting in~$x$ with an edge not in $M$ and ending in~$v$ with an edge in~$M$.

We denote by $S(x)$ the set consisting of $x$ together with all vertices switchable from $x$.
\end{definition}

\begin{toappendix}

\begin{definition}[Switching along an alternating path]
Let $M$ be a matching in $G$ and let $P$ be an $M$-alternating path. The \emph{switching of $M$ along $P$} is the matching
\[
M \triangle E(P),
\]
obtained by taking the symmetric difference of $M$ with the edge set of $P$, that is, by removing all edges of $P$ that lie in $M$ and adding all edges of $P$ that do not lie in $M$.
\end{definition}

Switching along such a path produces another matching of the same size, since the path has even length and contains equally many edges in $M$ and outside $M$. In particular, it yields another maximum matching in which $x$ becomes matched and $v$ becomes unmatched.

\begin{observation}
\label{obs:switchable_inessential}
Let $x$ be unmatched in a maximum matching $M$. Then every vertex $v \in S(x)$ is inessential.
\end{observation}

\begin{proof}
Let $P$ be an even $M$-alternating path from $x$ to $v$ whose last edge lies in $M$. Replacing $M$ by $M \triangle E(P)$ yields another maximum matching in which $v$ is unmatched. Hence $v$ is inessential.
\end{proof}

\begin{lemma}
\label{lem::sym_def_matchings}
Let \(G\) be a graph and let \(M\) and \(M'\) be two maximum matchings in~\(G\).  Define
$ R \;=\; M\triangle M' \;=\;(M\setminus M')\;\cup\;(M'\setminus M) $\,.
Then every connected component of the subgraph $H\subseteq G$ formed by the edges of $R$ and their end vertices is either
\begin{enumerate}
  \item an even cycle whose edges alternate between \(M\) and \(M'\), or
  \item a path whose edges alternate between \(M\) and \(M'\), and whose endpoints are unmatched in exactly one of the matchings.
\end{enumerate}
Moreover, since \(M\) and \(M'\) are both maximum, each alternating path component has one endpoint unmatched in \(M\) and the other unmatched in \(M'\).
\end{lemma}

\begin{proof}
Every vertex \(v\in V(G)\) is incident to at most one edge of \(M\setminus M'\) and at most one of \(M'\setminus M\).  Hence \(\deg_H(v)\le2\), so each connected component of \(H\) is either a cycle or a path, and edges along any component alternate between membership in \(M\) and \(M'\).

Let \(P=(x_0x_1\cdots x_k)\) be a path component of \(H\) with edges \(e_i=\{x_{i-1},x_i\}\).  Its endpoints satisfy \(\deg_H(x_0)=\deg_H(x_k)=1\), and so
\[
\begin{cases}
e_1\in M\setminus M' &\implies x_0\text{ is matched in }M,\text{ free in }M',\\
e_1\in M'\setminus M   &\implies x_0\text{ is matched in }M',\text{ free in }M,
\end{cases}
\]
and similarly for \(x_k\) with respect to \(e_k\).  Interior vertices have degree 2 in \(H\) and are thus matched in both.

If a path \(P\) had both endpoints free in the same matching (say \(M\)), then \(e_1,e_k\in M'\), and flipping \(M\) along \(P\) (removing its \(M\)-edges and inserting its \(M'\)-edges) would produce a larger matching, contradicting maximality of \(M\).  The same argument applies swapping \(M\) and \(M'\).  Therefore every alternating path in \(R\) must join an \(M\)-free vertex to an \(M'\)-free vertex.
\end{proof}

\begin{corollary}
\label{cor::alternative_def}
Let \(G\) be a graph and \(M\) a fixed maximum matching in \(G\).  A vertex \(v\in V\) is inessential if and only if there exists an even-length \(M\)\nobreakdash-alternating path
\[
P = (w = x_0, x_1, \dots, x_{2\ell} = v)
\]
starting at some \(w\) free in \(M\) and ending at \(v\), with its last edge \(\{x_{2\ell-1},x_{2\ell}\}\in M\).
\end{corollary}

\begin{proof}
(\(\Rightarrow\))  
If $v$ is unmatched in some maximum matching $M'$, then by Lemma~\ref{lem::sym_def_matchings}, the component of $M \triangle M'$ containing $v$ is an alternating path whose other endpoint $w$ is unmatched in $M$. Since $v$ is unmatched in $M'$, the edge of this path incident to $v$ must belong to $M$. Hence the path has even length and ends at $v$ with an edge of $M$.

(\(\Leftarrow\))  
Conversely, if there is an even-length $M$-alternating path $P$ from a free vertex $w$ to $v$, flipping $M$ along $P$ produces a maximum matching $M'$ that leaves $v$ unmatched.
\end{proof}

\end{toappendix}

\begin{definition}[The subgraph $G_x$]\label{def:Gx}
Let $G$ be a graph, $M$ a matching and $x$ be a vertex unmatched by $M$. Define $G_x$ to be the subgraph of $G$ induced by the vertex set
\[
V(G_x) := S(x) \cup N_G(S(x)).
\]
\end{definition}

The graph $G_x$ is the basic local object in our analysis. It contains all vertices that may become unmatched by switching from $x$, together with all their neighbors, and hence $G_x$ is in particular connected and $x$ is the only unmatched vertex in~$G_x$ if $M$ is maximum. 
In our context, $G_x$ records the part of the graph that will be relevant for our local modifications around $x$.
Our goal in the remainder of this section is to show that the Gallai--Edmonds decomposition of each such $G_x$ is particularly simple and that local changes of the matching within $G_x$ do not significantly affect the neighborhoods of other unmatched vertices.

We summarize our core findings as follows:

\begin{claim}\apxcmark{lem:mono}
\label{clm:Gx_structure}
Let $G$ be a graph, let $M$ be a maximum matching of $G$, and let $x$ be unmatched in $M$. Then the connected auxiliary graph
$ G_x = G\bigl[S(x)\cup N_G(S(x))\bigr] $ satisfies
\begin{align*}
D(G_x) &= S(x), &&A(G_x) = N_G(S(x))\setminus S(x), &&C(G_x) = \varnothing, \quad\mbox{and}\\
D(G_x) &= V(G_x)\cap D(G), &&A(G_x) \subseteq A(G). &&
\end{align*}
In particular, the Gallai--Edmonds decomposition of $G_x$ consists only of $D$- and $A$-vertices, with no $C$-part, and $x$ is the only vertex of $G_x$ unmatched by the restriction of $M$ to~$G_x$.
Consequently, if $|A(G_x)| = k$, then $G[D(G_x)]$ has $k+1$ connected components.

Moreover, every connected component of the graph $G[D(G_x)]$ is a connected component of $G[D(G)]$, too.
If $u\not=x$ is another vertex unmatched in~$M$ and $D_0$ is the connected component of $G[D(G_u)]$ containing~$u$, then $V(D_0)\cap V(G_x)=\varnothing$.
If $P$ is an even $M$-alternating path from $x$ to some $y\in S(x)$ and $M'=M\triangle E(P)$, then for every edge $f$ with at least one endpoint in~$D_0$ we have $f\in M$ $\iff$~$f\in M'$.
\end{claim}

Claim~\ref{clm:Gx_structure} is proven in several steps. A key technical ingredient is a monotonicity property of reachability under switching: if we switch from $x$ to a vertex $u\in S(x)$, then no new switchable vertices are created outside the original set $S(x)$. This will ensure that all subsequent local modifications remain inside $G_x$.

We begin with this monotonicity statement and then use it to show that $x$ is the unique vertex of $G_x$ unmatched by the restricted matching $M\cap E(G_x)$. From there we determine the Gallai--Edmonds decomposition of $G_x$ itself, compare it with the global Gallai--Edmonds decomposition of $G$ and derive the desired conclusions.
\onlyapx{See the Appendix.}

\begin{toappendix}
\onlyapx{%
We continue with detailed arguments leading to a full proof of \Cref{clm:Gx_structure}.
}

\begin{lemma}
\label{lem:mono}
Let $x$ be unmatched in $M$, and let $u\in S_M(x)$. Let $M'$ be the maximum matching obtained by switching along an even $M$-alternating path from $x$ to $u$. Then
\[
S_{M'}(u) \subseteq S_M(x).
\]
\end{lemma}

\begin{proof}
Let $P$ be the even $M$-alternating path from $x$ to $u$ used to obtain $M'$, so that
\[
M' = M\triangle E(P).
\]

Let $v\in S_{M'}(u)$. If $v=u$, then $v\in S_M(x)$ by assumption, so suppose $v\neq u$. By definition, there exists an even $M'$-alternating path $Q$ from $u$ to $v$ whose last edge lies in $M'$. Switching $M'$ along $Q$ gives another maximum matching
\[
M'' := M'\triangle E(Q)
\]
in which $v$ is unmatched.

We compare the two maximum matchings $M$ and $M''$. Since
\[
M'' = M\triangle E(P)\triangle E(Q),
\]
we have
\[
M\triangle M'' = E(P)\triangle E(Q).
\]

This is now the symmetric difference of two matchings, namely $M$ and $M''$. Therefore every connected component of the graph with edge set $M\triangle M''$ is an alternating path or an alternating cycle with respect to $M$ and $M''$. Indeed, each vertex is incident to at most one edge of $M$ and at most one edge of $M''$, so two consecutive edges in such a component cannot both belong to the same matching.

The matching $M$ leaves $x$ unmatched, while $M''$ leaves $v$ unmatched. Moreover, no unmatched vertex of $M$ other than $x$ lies on $P$ or $Q$: such a vertex is also unmatched in $M'$, and it cannot occur as an internal vertex of an $M'$-alternating path. Hence, in the component of $M\triangle M''$ containing $x$, the other endpoint must be $v$.

Thus this component is an alternating path from $x$ to $v$ with respect to $M$ and $M''$. Since $x$ is unmatched in $M$, the first edge of this path is not in $M$. Since $v$ is unmatched in $M''$, the last edge of this path lies in $M$. Hence it is an even $M$-alternating path from $x$ to $v$ whose last edge lies in $M$.

Therefore $v\in S_M(x)$, and so
\[
S_{M'}(u)\subseteq S_M(x).
\]
\end{proof}

Lemma~\ref{lem:mono} shows that the set $S(x)$ is stable under switching: once a vertex becomes reachable from $x$ by an even alternating path, this reachability is preserved when switching to another vertex of $S(x)$. In particular, all subsequent switching operations starting from $x$ remain confined to the vertex set $S(x)$, and hence to the auxiliary graph $G_x$.

The following corollary makes this explicit at the level of the auxiliary graphs.

\begin{corollary}
\label{cor:Gu_in_Gx}
Let $x$ be unmatched in $M$, let $u\in S(x)$, and let $M'$ be obtained by switching along an even $M$-alternating path from $x$ to $u$. Then
\[
G_u \subseteq G_x,
\]
where $G_u$ is defined with respect to $M'$.
\end{corollary}

\begin{proof}
By Lemma~\ref{lem:mono}, we have $S_{M'}(u)\subseteq S_M(x)$. It follows that
\[
N\bigl(S_{M'}(u)\bigr)\subseteq N\bigl(S_M(x)\bigr).
\]
Therefore,
\[
V(G_u)=S_{M'}(u)\cup N\bigl(S_{M'}(u)\bigr)
\subseteq
S_M(x)\cup N\bigl(S_M(x)\bigr)
= V(G_x),
\]
and hence $G_u\subseteq G_x$.
\end{proof}

We now begin the proof of Claim~\ref{clm:Gx_structure}. The first step is to understand the matching structure inside $G_x$. We show that the restriction of $M$ to $G_x$ leaves exactly one vertex unmatched, namely $x$ itself.

\begin{lemma}
\label{lem:restriction_one_free}
Let $G$ be a graph, let $M$ be a maximum matching of $G$, and let $x$ be unmatched in $M$. Define
\[
G_x := G\bigl[S(x)\cup N(S(x))\bigr],
\]
and let
\[
M_x := M\cap E(G_x).
\]
Then every vertex of $G_x$ except $x$ is matched by $M_x$. In particular, $x$ is the unique vertex of $G_x$ unmatched by $M_x$, and $M_x$ is a maximum matching of $G_x$.
\end{lemma}

\begin{proof}
By definition, $x$ is unmatched by $M$, and hence also by $M_x$.

We show that every vertex $y\in V(G_x)\setminus\{x\}$ is matched by $M_x$.

First suppose that $y\in S(x)\setminus\{x\}$. By definition of $S(x)$, there exists an even $M$-alternating path
\[
P=(x=x_0,x_1,\dots,x_{2\ell}=y)
\]
from $x$ to $y$ whose last edge $x_{2\ell-1}x_{2\ell}$ lies in $M$. Thus $y$ is matched in $M$ to the vertex $x_{2\ell-1}$. Since both endpoints of this edge lie on $P$, they belong to $V(G_x)$, so this matching edge belongs to $M_x$.

Now suppose that $y\in N(S(x))\setminus S(x)$.
Choose a vertex $w\in S(x)$ adjacent to $y$. Since $w\in S(x)$, there exists an even $M$-alternating path
\[
P=(x=x_0,x_1,\dots,x_{2\ell}=w)
\]
from $x$ to $w$ whose last edge lies in $M$.

If $y\notin V(P)$ and $y$ is not matched by $M$, then appending the edge $wy$ to $P$ yields an $M$-augmenting path from $x$ to $y$ which contradicts maximality of~$M$.
So, let $z$ be the vertex matched with $y$ by $M$.
Then, appending the edges $wy$ and $yz$ to $P$ yields an even $M$-alternating path from $x$ to $z$. Hence $z\in S(x)$, and so $yz\in M_x$.

If $y\in V(P)$, then, since $P$ is an even $M$-alternating path that starts with an edge not in~$M$, $y\notin S(x)$ implies that $y$ occurs on $P$ at an odd index, say $y=x_{2i-1}$. 
Then, by definition, $z=x_{2i}\in S(x)$ and the edge $yz$ belongs to $M_x$. Therefore, $y$ is matched by $M_x$.

We have proved that $x$ is the only vertex of $G_x$ unmatched by $M_x$. Therefore $M_x$ is a maximum matching of $G_x$.
\end{proof}

\Cref{lem:restriction_one_free} establishes the first key property of $G_x$: it behaves like a graph with a single exposed vertex. This will allow us to determine its Gallai--Edmonds decomposition explicitly.

We now determine the Gallai--Edmonds decomposition of $G_x$. The next lemma identifies its $D$-, $A$-, and $C$-parts, thereby proving the first three identities in \Cref{clm:Gx_structure}.

\begin{lemma}
\label{lem:DGx_equals_Splus}
Let $G$ be a graph, let $M$ be a maximum matching of $G$, and let $x$ be unmatched in $M$. Let
\[
G_x := G\bigl[S(x)\cup N(S(x))\bigr].
\]
Then
\[
D(G_x)=S(x),
\qquad
A(G_x)=N(S(x))\setminus S(x),
\qquad
C(G_x)=\varnothing.
\]
\end{lemma}

\begin{proof}
Let
\[
M_x:=M\cap E(G_x).
\]
By Lemma~\ref{lem:restriction_one_free}, $M_x$ is a maximum matching of $G_x$, and $x$ is the unique vertex of $G_x$ unmatched by $M_x$.

We now prove the equality $D(G_x)=S(x)$.

\smallskip
\noindent
\emph{First inclusion: $S(x)\subseteq D(G_x)$.}
Since $x$ is unmatched by the maximum matching $M_x$, we have $x\in D(G_x)$.

Now let $u\in S(x)$. By definition, there exists an even $M$-alternating path from $x$ to $u$ whose last edge lies in $M$. All vertices of this path belong to $S(x)$ or $N(S(x))$, and hence the path is contained in $G_x$. Switching along this path yields a matching of $G_x$ of the same size as $M_x$, and thus a maximum matching of $G_x$, in which $u$ is unmatched. Therefore $u\in D(G_x)$.

\smallskip
\noindent
\emph{Second inclusion: $D(G_x)\subseteq S(x)$.}
Let $u\in D(G_x)$. Then there exists a maximum matching $M_x'$ of $G_x$ that leaves $u$ unmatched.

Applying Corollary~\ref{cor::alternative_def} to the graph $G_x$ and the maximum matching $M_x$, there exists an even $M_x$-alternating path from the unique $M_x$-unmatched vertex $x$ to $u$ whose last edge lies in $M_x$.

Since $M_x$ is the restriction of $M$, this is also an even $M$-alternating path in $G$ from $x$ to $u$ whose last edge lies in $M$. Therefore $u\in S(x)$.

So $D(G_x)=S(x)$.

Finally, by definition of $G_x$,
\[
V(G_x)=S(x)\cup N(S(x)).
\]
Thus every vertex in $V(G_x)\setminus S(x)$ has a neighbor in $S(x)=D(G_x)$, and hence belongs to $A(G_x)$. Therefore
\[
A(G_x)=N(S(x))\setminus S(x),
\]
and consequently $C(G_x)=\varnothing$.
\end{proof}

Thus we have shown that
\[
D(G_x)=S(x), \qquad A(G_x)=N(S(x))\setminus S(x), \qquad C(G_x)=\varnothing,
\]
establishing the first part of Claim~\ref{clm:Gx_structure}.

We next relate the local structure of $G_x$ to the global Gallai--Edmonds decomposition of $G$. The following lemma identifies precisely which vertices of $G_x$ belong to the global $D(G)$.

\begin{lemma}
\label{lem:V_Gx_cap_DG}
Let $G$ be a graph, let $M$ be a maximum matching of $G$, and let $x$ be unmatched in $M$. Then
\[
V(G_x)\cap D(G)=S(x).
\]
\end{lemma}

\begin{proof}
We first show that
\[
S(x)\subseteq V(G_x)\cap D(G).
\]
By definition, $S(x)\subseteq V(G_x)$. Moreover, $x$ is unmatched in the maximum matching $M$, so $x\in D(G)$, and every vertex of $S(x)$ is inessential by Observation~\ref{obs:switchable_inessential}. Hence $S(x)\subseteq D(G)$, and therefore
\[
S(x)\subseteq V(G_x)\cap D(G).
\]

For the reverse inclusion, let $a\in V(G_x)\cap D(G)$. Suppose for contradiction that $a\notin S(x)$. Since
\[
V(G_x)=S(x)\cup N(S(x)),
\]
it follows that $a\in N(S(x))\setminus S(x)$. In particular, there exists a vertex $u\in S(x)$ adjacent to~$a$.

Let $M'$ be the maximum matching in $G$ obtained by switching along an even $M$-alternating path from $x$ to $u$. Then $u$ is unmatched in $M'$.

Since $a\in D(G)$, the vertex $a$ is inessential in $G$. By Corollary~\ref{cor::alternative_def}, applied with respect to the maximum matching $M'$, there exists an even $M'$-alternating path from some vertex unmatched in $M'$ to $a$, whose last edge lies in $M'$.

We claim that this path must start at $u$. Indeed, if it started at some unmatched vertex $z\neq u$, then appending the edge $au$ would produce an $M'$-augmenting path from $z$ to $u$, contradicting the maximality of $M'$.

Thus there exists an even $M'$-alternating path from $u$ to $a$ whose last edge lies in $M'$, so $a\in S_{M'}(u)$. By Lemma~\ref{lem:mono},
\[
S_{M'}(u)\subseteq S_M(x),
\]
and hence $a\in S(x)$, a contradiction.

Therefore
\[
V(G_x)\cap D(G)\subseteq S(x).
\]
Combining the two inclusions gives
\[
V(G_x)\cap D(G)=S(x).
\qedhere \]
\end{proof}

This proves the equality
\[
V(G_x)\cap D(G)=S(x)=D(G_x),
\]
which is the fourth statement of Claim~\ref{clm:Gx_structure}.

Finally, we compare the $A$-part of $G_x$ with the global set $A(G)$.

\begin{lemma}
\label{lem:AGx_in_AG}
Let $G$ be a graph, let $M$ be a maximum matching of $G$, and let $x$ be unmatched in $M$. Then
\[
A(G_x)\subseteq A(G).
\]
\end{lemma}

\begin{proof}
Let $a\in A(G_x)$. By Lemma~\ref{lem:DGx_equals_Splus},
\[
A(G_x)=N(S(x))\setminus S(x)
\quad\text{and}\quad
D(G_x)=S(x).
\]
Hence $a\notin S(x)$ and $a$ has a neighbor in $S(x)$.

By Observation~\ref{obs:switchable_inessential}, every vertex of $S(x)$ is inessential in $G$, so
\[
S(x)\subseteq D(G).
\]
Therefore $a$ has a neighbor in $D(G)$.

Since there are no edges between $C(G)$ and $D(G)$ in the Gallai--Edmonds decomposition of $G$, it follows that $a\notin C(G)$.

On the other hand, by Lemma~\ref{lem:V_Gx_cap_DG},
\[
V(G_x)\cap D(G)=S(x).
\]
Because $a\in V(G_x)\setminus S(x)$, we conclude that $a\notin D(G)$.

Thus $a$ lies in neither $C(G)$ nor $D(G)$, and therefore
\[
a\in V(G)\setminus (C(G)\cup D(G))=A(G).
\]
Hence
\[
A(G_x)\subseteq A(G).
\qedhere \]
\end{proof}

This establishes the inclusion
\[
A(G_x)\subseteq A(G),
\]
completing the first part of Claim~\ref{clm:Gx_structure}.

Observation~\ref{obs:A_size} will play a key role in our later degree argument. 
It provides a precise relation between the number of connected components of $G[D(G_x)]$ and the size of the set $A(G_x)$. In particular, it allows us to control the number of edges between $D(G_x)$ and $A(G_x)$ in terms of the number of components of $D(G_x)$, which will be crucial in deriving a contradiction in planar graphs.

\begin{observation}
\label{obs:A_size}
Let $G$ be a graph, $M$ a maximum matching of $G$, and let $x$ be unmatched in $M$. Consider the subgraph $G_x$ and its Gallai--Edmonds decomposition. 
If $|A(G_x)|=k$, then $G[D(G_x)]$ has $k+1$ connected components.
\end{observation}

\begin{proof}
By Lemma~\ref{lem:restriction_one_free}, the matching $M_x:=M\cap E(G_x)$ is a maximum matching of $G_x$ that leaves exactly one vertex unmatched.

Let $G[D(G_x)]$ have $\ell$ connected components.
By \Cref{clm:GallaiEdmonds}\,\ref{it:k-A-unmatched}),
the number of vertices left unmatched by a maximum matching equals
\[
\ell - |A(G_x)| = \ell-k .
\]
Since this number equals $1$, we obtain
\[
\ell = k + 1.
\qedhere \]
\end{proof}

To analyze the internal structure of the components of $G[D(G_x)]$, we recall a standard property of factor-critical graphs (\Cref{clm:GallaiEdmonds}). This will allow us to control alternating paths inside individual components.

\begin{lemma}
\label{lem:fc_alt_path}
Let $H$ be a factor-critical graph, let $r \in V(H)$, and let $M$ be a maximum matching of $H$ that leaves exactly $r$ unmatched. Then for every vertex $u \in V(H)$, there exists an $M$-alternating path of even length from $r$ to $u$ whose last edge lies in $M$.
\end{lemma}

\begin{proof}
Since $H$ is factor-critical, the graph $H \setminus u$ has a perfect matching $M_u$. Consider the symmetric difference
\[
M \triangle M_u.
\]
Every vertex of $V(H)\setminus\{r,u\}$ has degree $0$ or $2$ in this graph, while $r$ is unmatched in $M$ and matched in $M_u$, and $u$ is matched in $M$ and unmatched in $M_u$. Hence the component of $M \triangle M_u$ containing $r$ is an alternating path from $r$ to $u$.

Its first edge lies in $M_u \setminus M$, so it is not in $M$, and its last edge lies in $M \setminus M_u$, so it is in $M$. Therefore the path has even length and is $M$-alternating.
\end{proof}

We next recall a structural property of the Gallai--Edmonds decomposition that describes how maximum matchings interact with individual factor-critical components.

\begin{lemma}
\label{lem:GE_component_matching_structure}
Let $H$ be a graph, let $M$ be a maximum matching of $H$, and let $D$ be a connected component of $H[D(H)]$. Then exactly one of the following holds:
\begin{enumerate}
    \item no edge of $M$ joins $A(H)$ to $D$, and $M|_{D}$ leaves exactly one vertex of $D$ unmatched;
    \item exactly one edge of $M$ joins $A(H)$ to $D$, and if $v\in V(D)$ is its endpoint in $D$, then $M|_{D \setminus \{v\}}$ is a perfect matching of $D \setminus \{v\}$.
\end{enumerate}
\end{lemma}

\begin{proof}
Since $D$ is a connected component of $H[D(H)]$, it is factor-critical; in particular, $|V(D)|$ is odd.

First note that any edge with exactly one endpoint in $D$ has its other endpoint in $A(H)$. Indeed, there are no edges between $D(H)$ and $C(H)$, and there are no edges between distinct components of $H[D(H)]$.

By the Gallai--Edmonds structure theorem \Cref{it:GE-A-pairwise_disjoint}, in every maximum matching each vertex of $A(H)$ is matched to a vertex of $D(H)$, and these matched vertices lie in pairwise distinct components of $H[D(H)]$. Hence the fixed component $D$ is incident with at most one matching edge from $A(H)$.

We now distinguish the two possible cases.

First suppose that no edge of $M$ joins $A(H)$ to $D$. Then every edge of $M$ incident with a vertex of $D$ lies inside $D$. Thus $M|_D$ is a matching of $D$. Since $|V(D)|$ is odd, it leaves at least one vertex of $D$ unmatched.

We claim that $M|_D$ is a maximum matching of $D$. Otherwise, replacing $M|_D$ by a larger matching of $D$ would produce a matching of $H$ larger than $M$, contradicting the maximality of $M$. Since $D$ is factor-critical, every maximum matching of $D$ leaves exactly one vertex unmatched. Hence $M|_D$ leaves exactly one vertex of $D$ unmatched.

Now suppose that an edge of $M$ joins $A(H)$ to $D$, and let $v\in V(D)$ be its endpoint in $D$. By the preceding paragraph, this is the only matching edge joining $A(H)$ to $D$. Hence every edge of $M$ incident with a vertex of $D\setminus\{v\}$ lies inside $D\setminus\{v\}$, so $M|_{D\setminus\{v\}}$ is a matching of $D\setminus\{v\}$.

We claim that $M|_{D\setminus\{v\}}$ is a maximum matching of $D\setminus\{v\}$. Otherwise, replacing it by a larger matching of $D\setminus\{v\}$ would again produce a matching of $H$ larger than $M$. Since $D$ is factor-critical, $D\setminus\{v\}$ has a perfect matching. Therefore $M|_{D\setminus\{v\}}$ is a perfect matching of $D\setminus\{v\}$.

These two cases are mutually exclusive and exhaustive, since $D$ is incident with at most one matching edge from $A(H)$.
\end{proof}

We now relate the structure of the components of $G[D(G_x)]$ to alternating paths. 
The following lemma shows that within each such component, alternating paths can be chosen to remain entirely inside the component. In this sense, alternating reachability is local to the factor-critical component: if $u$ lies in the same component as $x$, then there exists an even $M$-alternating path from $x$ to $u$ that does not leave this component.

\begin{lemma}
\label{lem:alternating_components}
Let $G$ be a graph, let $M$ be a maximum matching of $G$, and let $x$ be unmatched in $M$. For a vertex $u\in S(x)$, the following are equivalent:
\begin{enumerate}
    \item $u$ and $x$ belong to the same connected component of $G[D(G_x)]$;
    \item there exists an even $M$-alternating path from $x$ to $u$ whose vertices all lie in $S(x)$.
\end{enumerate}
\end{lemma}

\begin{proof}
By Lemma~\ref{lem:DGx_equals_Splus}, we have $D(G_x)=S(x)$. Let $D_0$ be the connected component of $G[D(G_x)]$ containing $x$.

\smallskip
\noindent
\emph{$(2)\Rightarrow(1)$.}
If there exists such an alternating path, it is contained in $G[D(G_x)]$, so $x$ and $u$ lie in the same component.

\smallskip
\noindent
\emph{$(1)\Rightarrow(2)$.}
Assume $u\in D_0$. Then $D_0$ is factor-critical.

Let $M_x := M\cap E(G_x)$. By \Cref{lem:restriction_one_free}, $M_x$ is a maximum matching of $G_x$ leaving exactly $x$ unmatched.

Applying \Cref{lem:GE_component_matching_structure} to $G_x$ and $M_x$, the component $D_0$ is not incident with any edge of $M_x$ joining $A(G_x)$ to $D_0$, since it contains the unmatched vertex $x$. Hence $M_x|_{D_0}$ is a near-perfect matching of $D_0$ leaving exactly $x$ unmatched.

By \Cref{lem:fc_alt_path}, there exists an even $M_x$-alternating path from $x$ to $u$ whose vertices lie in $D_0$. Since $M_x$ is the restriction of $M$, this is also an even $M$-alternating path contained in $S(x)$.
\end{proof}

We now compare the local structure of $G_x$ with the global Gallai--Edmonds decomposition of $G$. The next lemma shows that the components of $G[D(G_x)]$ correspond exactly to components of $G[D(G)]$.

\begin{lemma}
\label{lem:D_components_coincide}
Let $G$ be a graph, let $M$ be a maximum matching of $G$, and let $x$ be unmatched in $M$. Let $G_x := G[S(x)\cup N(S(x))]$.

Then for every connected component $D'$ of $G[D(G_x)]$, there exists a unique connected component $D$ of $G[D(G)]$ such that
\[
V(D') = V(D).
\]
In particular, $D'$ and $D$ coincide as induced subgraphs of $G$.
\end{lemma}

\begin{proof}
By \Cref{lem:DGx_equals_Splus}, we have
\[
D(G_x)=S(x).
\]
By \Cref{lem:V_Gx_cap_DG},
\[
V(G_x)\cap D(G)=S(x).
\]
Hence
\[
G[D(G_x)] = G[S(x)],
\]
and this is an induced subgraph of $G[D(G)]$. Therefore every connected component of $G[S(x)]$ is contained in a connected component of $G[D(G)]$.

It remains to show that if a connected component $D$ of $G[D(G)]$ intersects $S(x)$, then in fact
\[
V(D)\subseteq S(x).
\]

So let $D$ be a connected component of $G[D(G)]$, and let $a\in V(D)\cap S(x)$. Since $a\in S(x)$, there exists an even $M$-alternating path from $x$ to $a$ whose last edge lies in $M$. Switch along this path, and let $M'$ be the resulting maximum matching. Then $a$ is unmatched in $M'$.

Now apply \Cref{lem:GE_component_matching_structure} to the graph $G$, the maximum matching $M'$, and the component $D$. Since $a\in D$ is unmatched in $M'$, no edge of $M'$ joins $A(G)$ to $D$, and the restriction $M'|_D$ leaves exactly $a$ unmatched.

Because $D$ is factor-critical, \Cref{lem:fc_alt_path} implies that for every vertex $u\in V(D)$, there exists an even $M'$-alternating path from $a$ to $u$ whose last edge lies in $M'$. Thus
\[
u\in S_{M'}(a).
\]
By \Cref{lem:mono},
\[
S_{M'}(a)\subseteq S_M(x)=S(x).
\]
Hence every vertex of $D$ belongs to $S(x)$, proving that
\[
V(D)\subseteq S(x).
\]

Therefore each connected component of $G[D(G)]$ that intersects $S(x)$ is entirely contained in $S(x)$, and so the connected components of $G[S(x)]$ are exactly the connected components of $G[D(G)]$ that intersect $S(x)$.
\end{proof}

Finally, we show in two lemmas that the effect of switching operations is localized. 
More precisely, modifying the matching along alternating paths inside the auxiliary graph $G_x$ does not affect the components associated with other unmatched vertices. In this sense, the switching operations used in our arguments are local and do not influence other unmatched vertices and their local situation.

\begin{lemma}
\label{lem:Dx_disjoint_Gu}
Let $G$ be a graph, let $M$ be a maximum matching of $G$, and let $x$ and $u$ be distinct vertices unmatched in $M$. Let $D_0$ be the connected component of $G[D(G)]$ containing~$u$. Then
\[
V(D_0)\cap V(G_x)=\varnothing.
\]
\end{lemma}

\begin{proof}
By \Cref{lem:V_Gx_cap_DG} we have
\[
V(G_x)\cap D(G)=S(x).
\]
Suppose for contradiction that
\[
V(D_0)\cap V(G_x)\neq\varnothing.
\]
Then, because $D_0\subseteq D(G)$, we obtain
\[
V(D_0)\cap S(x)\neq\varnothing.
\]
By \Cref{lem:D_components_coincide}, every connected component of $G[D(G_x)]$ is a connected component of $G[D(G)]$. Since
\[
D(G_x)=S(x),
\]
it follows that the whole component $D_0$ is contained in $S(x)$. In particular, $u\in S(x)$.

But $x\neq u$ and both $x$ and $u$ are unmatched in $M$, whereas every vertex of $S(x)\setminus\{x\}$ is matched in $M$ (being the endpoint of an even $M$-alternating path whose last edge lies in~$M$). This is impossible. Therefore
\[
V(D_0)\cap V(G_x)=\varnothing.
\qedhere \]
\end{proof}

\begin{lemma}
\label{lem:switch_preserves_other_local_configs}
Let $G$ be a planar graph, let $M$ be a maximum matching of $G$, and let $x$ and $u$ be distinct vertices unmatched in $M$. Let $P$ be an even $M$-alternating path from $x$ to some vertex of $S(x)$, and let
\[
M' := M \triangle E(P).
\]
Then the local matching configuration around $u$ is the same with respect to $M$ and $M'$.

More precisely, let $D_u$ be the connected component of $G[D(G)]$ containing $u$. Then:
\begin{enumerate}
    \item every edge of $M$ with at least one endpoint in $D_u$ also belongs to $M'$, and vice versa;
    \item if two neighbors of $u$ are matched together with respect to $M$, then they are also matched together with respect to $M'$;
    \item if a neighbor of $u$ is matched outside $N(u)$ with respect to $M$, then it is also matched outside $N(u)$ with respect to $M'$;
    \item for every neighbor $a$ of $u$, the membership of $a$ in $A(G)$ or in $D_u$ is the same with respect to $M$ and $M'$.
\end{enumerate}
In particular, switching from $x$ does not change the case of $u$.
\end{lemma}

\begin{proof}
Since $P\subseteq G_x$ and
\[
V(G_x)\cap D(G)=S(x)
\]
by \Cref{lem:V_Gx_cap_DG}, every vertex of $P$ that lies in $D(G)$ belongs to $S(x)$. On the other hand, the component $D_u$ of $u$ in $G[D(G)]$ is disjoint from $S(x)$, since otherwise \(u\) would belong to \(S(x)\),
which is impossible because \(u\neq x\) is also unmatched. Hence
\[
V(P)\cap V(D_u)=\varnothing.
\]
Therefore no edge of the switching path $P$ is incident with a vertex of $D_u$, and so the matching on edges incident to $D_u$ is unchanged. This proves~(1).

Now let $a,b\in N(u)$ be matched together with respect to $M$. Then the paths
\[
u,a,b
\qquad\text{and}\qquad
u,b,a
\]
are even $M$-alternating paths ending with an edge of $M$. Hence $a,b\in S(u)\subseteq D(G)$. Since both are adjacent to $u\in D(G)$, they lie in the same connected component $D_u$ of $G[D(G)]$. By~(1), the edge $ab$ is unchanged by the switch, so $a$ and $b$ are still matched together with respect to $M'$. This proves~(2).

Next let $a\in N(u)$ be matched outside $N(u)$ with respect to $M$. If $a\in D(G)$, then since $a$ is adjacent to $u\in D(G)$, it lies in $D_u$, and by~(1) its matching edge is unchanged, so in particular it is still matched outside $N(u)$ with respect to $M'$. If instead $a\notin D(G)$, then $a$ cannot lie in $C(G)$ because it is adjacent to $u\in D(G)$ and there are no edges between $C(G)$ and $D(G)$. Thus $a\in A(G)$. The sets $A(G)$ and $D(G)$ are independent of the choice of maximum matching, so $a$ remains in $A(G)$ with respect to $M'$. In particular, $a$ cannot become matched together with another neighbor of $u$, because by \Cref{obs:neighbors_of_x} any two neighbors of an unmatched vertex that are matched together lie in $D(G)$, whereas $a\in A(G)$. Hence $a$ is still matched outside $N(u)$ with respect to $M'$. This proves~(3).

Statement~(4) follows immediately from the preceding argument: if a neighbor $a$ of $u$ lies in $D_u$, then its matching is unchanged by~(1); if $a\notin D_u$, then $a\in A(G)$, and membership in $A(G)$ is independent of the choice of maximum matching.

Thus the local matching pattern around $u$ is the same for $M$ and $M'$, and so the local configuration of $u$ does not change.
\end{proof}

This establishes the second and final part of \Cref{clm:Gx_structure}.

\end{toappendix}

\section{Degree bound and a classification of configurations}
\label{sec:classification}
\begin{toappendix}
\label{apx:classification}
\end{toappendix}

The main structural message is the following new degree bound in planar graphs: from any unmatched vertex, one can switch to a low-degree unmatched vertex.

\begin{theorem}[Low-degree switchable vertex in planar graphs]
\label{thm:deg5}
Let $G$ be a planar graph and $M$ a maximum matching of $G$. If $x$ is unmatched in $M$, then there exists a vertex $v$ switchable from $x$, $v \in S(x)$, with $\deg_G(v) \le 5$ (possibly $v=x$).
\end{theorem}

\begin{proof}
We focus on the planar induced subgraph $G_x = G\bigl[S(x)\cup N(S(x))\bigr]$ of~$G$ within the proof.
By \Cref{clm:Gx_structure}\onlyapx{ (Appendix: see \Cref{lem:restriction_one_free} and \Cref{lem:DGx_equals_Splus})}, the matching $M$ leaves exactly one vertex unmatched in $G_x$, namely $x$. 
In the Gallai--Edmonds decomposition of $G_x$ we have $V(G_x) = A \cup D$ where $D = D(G_x)=S(x)$ and $A = A(G_x)$. 
Let $D_1, \dots, D_k$ be the connected components of $G_x[D]$.

Suppose for a contradiction that every vertex in $D$ has degree at least $6$ in $G$. Since $G_x$ contains all neighbors of $S(x)=D$, we have $\deg_{G_x}(v)=\deg_G(v)\ge 6$ for all $v\in D$.

Pick a component $D_i$ and let $n_i = |D_i|$. Then, denoting by $e(D_i)$ the number of edges inside $D_i$ and $e(D_i,A)$ the number of edges between $D_i$ and $A$, our assumption yields
\begin{equation}\label{eq:tag1}
2e(D_i) + e(D_i,A) \;=\; \sum_{v\in D_i} \deg_{G_x}(v) \;\ge\; 6|D_i|= 6n_i.
\end{equation}
Since $D_i$ is planar, we have $e(D_i) \le 3n_i - 3$, and substituting into \eqref{eq:tag1} gives $ e(D_i,A) \;\ge\; 6n_i - 2(3n_i - 3) = 6 $.
Summing over all $k$ components of $G_x[D]$,
\begin{equation}\label{eq:tag2}
e(D,A) \;\ge\; 6k.
\end{equation}
 
Again by planarity of $G_x$, $|E(G_x)| < 3|V(G_x)|$.
Thus
\[
\sum_{v\in V(G_x)} \deg_{G_x}(v) = 2|E(G_x)| < 6|V(G_x)|.
\]
Focusing now on the degrees in $A$ and using $\deg_{G_x}(v)\ge 6$ for $v\in D$, we continue
\[
\sum_{a\in A} \deg_{G_x}(a) \;=\; \sum_{v\in V(G_x)} \deg_{G_x}(v) - \sum_{v\in D} \deg_{G_x}(v) \;<\; 6|V(G_x)| - 6|D| \;=\; 6|A|.
\]
Since every edge between $D$ and $A$ contributes to the degree of some vertex of $A$,
\begin{equation}\label{eq:tag5}
e(D,A) \;\le\; \sum_{a\in A} \deg_{G_x}(a) \;<\; 6|A|.
\end{equation}

Finally, by \Cref{clm:GallaiEdmonds}\,\ref{it:k-A-unmatched}) we have $|A|<k$ (actually, $|A|=k-1$ in this case). 
Substituting into \eqref{eq:tag5}, we get $ e(D,A)<6k$. Combining \eqref{eq:tag2} and \eqref{eq:tag5} we hence arrive at a contradiction $6k\le e(D,A)<6k$.
Therefore, some vertex in $D=S(x)$ has degree at most $5$ in $G$.
\end{proof}

The last step before our case analysis is to show that switching the unmatched vertices to low degree ones, as in \Cref{thm:deg5}, can be simultaneously done in the whole graph.
This can be derived from \Cref{thm:deg5} using also \Cref{clm:Gx_structure}\onlyapx{ (Appendix: see \Cref{lem:V_Gx_cap_DG})}, and we skip the straightforward proof in here:

\begin{corollary2rep}\apxmark{apx:classification}\label{cor:lowdeg_free}
Every planar graph $G$ admits a maximum matching $M$ such that either $M$ is perfect, or every vertex unmatched by $M$ has degree at most $5$.
\end{corollary2rep}

\begin{proof}
Let $M$ be an arbitrary maximum matching of $G$. If $M$ is perfect we are done, so assume it is not.  
Among all maximum matchings of $G$, choose one (still denoted $M$) that minimizes the quantity
\[
\Phi(M) := \sum_{x \in U(M)} \max\{0,\, \deg_G(x)-5\},
\]
where $U(M)$ denotes the set of vertices unmatched by $M$.

Suppose for contradiction that $\Phi(M)>0$. Then there exists an unmatched vertex $x\in U(M)$ with $\deg_G(x)\ge 6$. Consider the subgraph $G_x$ defined with respect to $M$. By Theorem~\ref{thm:deg5}, there exists a vertex $v\in S(x)$ with $\deg_G(v)\le 5$. Let $P$ be an even $M$-alternating path from $x$ to $v$ whose last edge lies in $M$, and let $M' := M \triangle E(P)$ be the matching obtained by switching along $P$. Then $M'$ is a maximum matching in which $v$ is unmatched and $x$ is matched.

We claim that every vertex unmatched by $M$ other than $x$ remains unmatched in $M'$. Since $P\subseteq G_x$, it suffices to show that no vertex of $U(M)\setminus\{x\}$ lies in $G_x$. But every unmatched vertex belongs to $D(G)$, and by Lemma~\ref{lem:V_Gx_cap_DG} we have
\[
V(G_x)\cap D(G)=S(x).
\]
Thus the only unmatched vertex contained in $G_x$ is $x$, and hence no other unmatched vertex lies on $P$.

Therefore, $U(M') = (U(M)\setminus\{x\}) \cup \{v\}$. Since $\deg_G(x)\ge 6$ and $\deg_G(v)\le 5$, we have
\[
\max\{0,\deg_G(v)-5\}=0
\quad\text{and}\quad
\max\{0,\deg_G(x)-5\}\ge 1,
\]
and hence $\Phi(M') < \Phi(M)$, contradicting the choice of $M$ minimizing $\Phi$. Thus $\Phi(M)=0$, i.e., every unmatched vertex has degree at most $5$.
\end{proof}

\subparagraph{The classification of configurations.}

By Observation~\ref{obs:perfect}, the only obstruction to turning a proper coloring of the contracted graph $H(G,M)$ of a matching $M$ in $G$ into a conflict-free coloring of $G$ comes from vertices that are unmatched by~$M$.
Moreover, by \Cref{cor:lowdeg_free}, we may assume that every vertex unmatched by $M$ has degree at most $5$ in~$G$.

We therefore consider a maximum matching $M$ in $G$ and an unmatched vertex $x$ of degree at most~$5$.
Since $M$ is a maximum matching, each vertex of the neighborhood $N(x)$ is matched, and so either matched to another vertex of $N(x)$, or matched to a vertex outside of $N(x)$. 
We classify all possible configurations at $x$, up to \emph{isomorphism}, with respect to \emph{the degree of $x$} and the way \emph{its neighbors are matched} (either inside $N(x)$, or outwards), in Figures~\ref{fig:matching_one}--\ref{fig:neighbors_five}.

\begin{figure}[ht]
    \centering
    \begin{subfigure}[b]{0.45\textwidth}
        \centering
        \begin{tikzpicture}[every node/.style={circle, draw, fill=black, inner sep=1pt}, node distance=1cm]
            \node (t1) at (0,0) {};
            \node (m1) at ($(t1) + (0.0,-1)$) {};
            \node (m12) at ($(m1) + (0,-1)$) {};
            \draw (t1) -- (m1);
            \draw (m12) -- (m1);
            \draw[dashed] (m12) -- ++(0, -1);

            \node[draw=blue, ellipse, fit=(m1) (m12), inner sep=4pt, fill=none, transform shape, scale=0.80] {};

            \node[draw=none, fill=none, anchor=west] at ($(t1) + (0.1, 0)$) {$x$};
            \node[draw=none, fill=none, anchor=west] at ($(m1) + (0.15, 0)$) {$a_1$};

            \node[draw=none, fill=none, shape=rectangle] at (0,-3.5) {\footnotesize Case 1A: Edges matched outwards};
        \end{tikzpicture}
    \end{subfigure}
    \caption{Degree one vertex; note that other edges could be incident to depicted vertices except~$x$}
    \label{fig:matching_one}
\end{figure}

\begin{figure}[ht]
    \centering
    \begin{subfigure}[b]{0.45\textwidth}
        \centering
        \begin{tikzpicture}[every node/.style={circle, draw, fill=black, inner sep=1pt}, node distance=1cm]
            \node (t1) at (0,0) {};
            \node (l1) at ($(t1) + (-0.7,-1)$) {};
            \node (r1) at ($(t1) + (0.7,-1)$) {};
            \node (l12) at ($(l1) + (0,-1)$) {};
            \node (r12) at ($(r1) + (0,-1)$) {};
            \draw (t1) -- (l1);
            \draw (t1) -- (r1);
            \draw (l12) -- (l1);
            \draw (r12) -- (r1);
            \draw[dashed] (l1) -- (r1);
            \draw[dashed] (l12) -- ++(0, -1);
            \draw[dashed] (r12) -- ++(0, -1);

            \node[draw=blue, ellipse, fit=(r1) (r12), inner sep=4pt, fill=none, transform shape, scale=0.80] {};
            \node[draw=blue, ellipse, fit=(l1) (l12), inner sep=4pt, fill=none, transform shape, scale=0.80] {};

            \node[draw=none, fill=none, anchor=west] at ($(t1) + (0.1, 0)$) {$x$};
            \node[draw=none, fill=none, anchor=east] at ($(l1) + (-0.1, 0)$) {$a_1$};
            \node[draw=none, fill=none, anchor=west] at ($(r1) + (0.15, 0)$) {$a_2$};
            
            \node[draw=none, fill=none, shape=rectangle] at (0,-3.5) {\footnotesize Case 2A: Edges matched outwards};
        \end{tikzpicture}
    \end{subfigure}
    \hfill
    \begin{subfigure}[b]{0.45\textwidth}
        \centering
        \begin{tikzpicture}[every node/.style={circle, draw, fill=black, inner sep=1pt}, node distance=1cm]
            \node (t2) at (0,0) {};
            \node (l2) at ($(t2) + (-0.7,-1)$) {};
            \node (r2) at ($(t2) + (0.7,-1)$) {};
            \draw (t2) -- (l2);
            \draw (t2) -- (r2);
            \draw (l2) -- (r2);
            \draw[dashed] (l2) -- ++(0, -1);
            \draw[dashed] (r2) -- ++(0, -1);

            \node[draw=blue, ellipse, fit=(l2) (r2), inner sep=4pt, fill=none, transform shape, scale=0.80] {};

            \node[draw=none, fill=none, anchor=west] at ($(t1) + (0.1, 0)$) {$x$};
            \node[draw=none, fill=none, anchor=east] at ($(l1) + (-0.25, 0)$) {$a_1$};
            \node[draw=none, fill=none, anchor=west] at ($(r1) + (0.25, 0)$) {$a_2$};
            
            \node[draw=none, fill=none, shape=rectangle] at (0,-3.5) {\footnotesize\it Case 2B: Matched neighbor pair};
        \end{tikzpicture}
    \end{subfigure}
    \caption{Degree two vertices; again, other edges could exist in the picture except those incident~to~$x$}
    \label{fig:branching-pairs}
\end{figure}

\begin{figure}[ht]
    \centering
    \begin{subfigure}[b]{0.45\textwidth}
        \centering
        \begin{tikzpicture}[every node/.style={circle, draw, fill=black, inner sep=1pt}, node distance=1cm]
            \node (t1) at (0,0) {};
            \node (l1) at ($(t1) + (-1.0,-1)$) {};
            \node (r1) at ($(t1) + (1.0,-1)$) {};
            \node (m1) at ($(t1) + (0,-1)$) {};
            \node (l12) at ($(l1) + (0,-1)$) {};
            \node (r12) at ($(r1) + (0,-1)$) {};
            \node (m12) at ($(m1) + (0,-1)$) {};
            \draw (t1) -- (l1);
            \draw (t1) -- (r1);
            \draw (t1) -- (m1);
            \draw (l12) -- (l1);
            \draw (m12) -- (m1);
            \draw (r12) -- (r1);
            \draw[dashed] (l12) -- ++(0, -1);
            \draw[dashed] (r12) -- ++(0, -1);
            \draw[dashed] (m12) -- ++(0, -1);

            \node[draw=blue, ellipse, fit=(r1) (r12), inner sep=4pt, fill=none, transform shape, scale=0.80] {};
            \node[draw=blue, ellipse, fit=(l1) (l12), inner sep=4pt, fill=none, transform shape, scale=0.80] {};
            \node[draw=blue, ellipse, fit=(m1) (m12), inner sep=4pt, fill=none, transform shape, scale=0.80] {};

            \node[draw=none, fill=none, anchor=west] at ($(t1) + (0.1, 0)$) {$x$};
            \node[draw=none, fill=none, anchor=west] at ($(l1) + (0.15, 0)$) {$a_1$};
            \node[draw=none, fill=none, anchor=west] at ($(r1) + (0.15, 0)$) {$a_3$};
            \node[draw=none, fill=none, anchor=west] at ($(m1) + (0.15, 0)$) {$a_2$};

            \node[draw=none, fill=none, shape=rectangle] at (0,-3.5) {\footnotesize Case 3A: All edges matched outwards};
        \end{tikzpicture}
    \end{subfigure}
    \hfill
    \begin{subfigure}[b]{0.45\textwidth}
        \centering
        \begin{tikzpicture}[every node/.style={circle, draw, fill=black, inner sep=1pt}, node distance=1cm]
            \node (t1) at (0,0) {};
            \node (l1) at ($(t1) + (-1.0,-1)$) {};
            \node (r1) at ($(t1) + (1.0,-1)$) {};
            \node (m1) at ($(t1) + (0,-1)$) {};
            \node (r12) at ($(r1) + (0,-1)$) {};
            \draw (t1) -- (l1);
            \draw (t1) -- (r1);
            \draw (t1) -- (m1);
            \draw (l1) -- (m1);
            \draw (r12) -- (r1);
            \draw[dashed] (l1) -- ++(0, -1);
            \draw[dashed] (m1) -- ++(0, -1);
            \draw[dashed] (r12) -- ++(0, -1);

            \node[draw=blue, ellipse, fit=(l1) (m1), inner sep=4pt, fill=none, transform shape, scale=0.80] {};
            \node[draw=blue, ellipse, fit=(r1) (r12), inner sep=4pt, fill=none, transform shape, scale=0.80] {};

            \node[draw=none, fill=none, anchor=west] at ($(t1) + (0.1, 0)$) {$x$};
            \node[draw=none, fill=none, anchor=east] at ($(l1) + (-0.2, 0)$) {$a_1$};
            \node[draw=none, fill=none, anchor=west] at ($(r1) + (0.15, 0)$) {$a_3$};
            \node[draw=none, fill=none, anchor=west] at ($(m1) + (0.25, 0)$) {$a_2$};

            \node[draw=none, fill=none, shape=rectangle] at (0,-3.5) {\footnotesize Case 3B: Matched one neighbor pair};
        \end{tikzpicture}
    \end{subfigure}
    \caption{Degree three vertices; again, other edges could exist in the picture except those incident to~$x$, in particular edges between neighbors of $x$ that are not matched together}
    \label{fig:neighbors-triple}
\end{figure}

\begin{figure}[ht]
    \centering
    \begin{subfigure}[b]{0.45\textwidth}
        \centering
        \begin{tikzpicture}[every node/.style={circle, draw, fill=black, inner sep=1pt}, node distance=1cm]
            \node (t1) at (0,0) {};
            \node (l1) at ($(t1) + (-1.5,-1)$) {};
            \node (r1) at ($(t1) + (1.5,-1)$) {};
            \node (rm1) at ($(t1) + (0.5,-1)$) {};
            \node (lm1) at ($(t1) + (-0.5,-1)$) {};
            \node (r12) at ($(r1) + (0,-1)$) {};
            \node (rm12) at ($(rm1) + (0,-1)$) {};
            \node (l12) at ($(l1) + (0,-1)$) {};
            \node (lm12) at ($(lm1) + (0,-1)$) {};
            \draw (t1) -- (l1);
            \draw (t1) -- (r1);
            \draw (t1) -- (lm1);
            \draw (t1) -- (rm1);
            \draw (l1) -- (l12);
            \draw (r12) -- (r1);
            \draw (rm12) -- (rm1);
            \draw (lm12) -- (lm1);
            \draw[dashed] (l12) -- ++(0, -1);
            \draw[dashed] (lm12) -- ++(0, -1);
            \draw[dashed] (r12) -- ++(0, -1);
            \draw[dashed] (rm12) -- ++(0, -1);

            \node[draw=blue, ellipse, fit=(l1) (l12), inner sep=4pt, fill=none, transform shape, scale=0.80] {};
            \node[draw=blue, ellipse, fit=(rm1) (rm12), inner sep=4pt, fill=none, transform shape, scale=0.80] {};
            \node[draw=blue, ellipse, fit=(lm1) (lm12), inner sep=4pt, fill=none, transform shape, scale=0.80] {};
            \node[draw=blue, ellipse, fit=(r1) (r12), inner sep=4pt, fill=none, transform shape, scale=0.80] {};

            \node[draw=none, fill=none, anchor=west] at ($(t1) + (0.1, 0)$) {$x$};
            \node[draw=none, fill=none, anchor=west] at ($(l1) + (0.15, 0)$) {$a_1$};
            \node[draw=none, fill=none, anchor=west] at ($(r1) + (0.15, 0)$) {$a_4$};
            \node[draw=none, fill=none, anchor=west] at ($(lm1) + (0.15, 0)$) {$a_2$};
            \node[draw=none, fill=none, anchor=west] at ($(rm1) + (0.15, 0)$) {$a_3$};

            \node[draw=none, fill=none, shape=rectangle] at (0,-3.5) {\footnotesize Case 4A: All edges matched outwards};
        \end{tikzpicture}
    \end{subfigure}
    \hfill
    \begin{subfigure}[b]{0.45\textwidth}
        \centering
        \begin{tikzpicture}[every node/.style={circle, draw, fill=black, inner sep=1pt}, node distance=1cm]
            \node (t1) at (0,0) {};
            \node (l1) at ($(t1) + (-1.5,-1)$) {};
            \node (r1) at ($(t1) + (1.5,-1)$) {};
            \node (rm1) at ($(t1) + (0.5,-1)$) {};
            \node (lm1) at ($(t1) + (-0.5,-1)$) {};
            \node (r12) at ($(r1) + (0,-1)$) {};
            \node (rm12) at ($(rm1) + (0,-1)$) {};
            \draw (t1) -- (l1);
            \draw (t1) -- (r1);
            \draw (t1) -- (lm1);
            \draw (t1) -- (rm1);
            \draw (l1) -- (lm1);
            \draw (r12) -- (r1);
            \draw (rm12) -- (rm1);
            \draw[dashed] (l1) -- ++(0, -1);
            \draw[dashed] (lm1) -- ++(0, -1);
            \draw[dashed] (r12) -- ++(0, -1);
            \draw[dashed] (rm12) -- ++(0, -1);

            \node[draw=blue, ellipse, fit=(l1) (lm1), inner sep=4pt, fill=none, transform shape, scale=0.80] {};
            \node[draw=blue, ellipse, fit=(rm1) (rm12), inner sep=4pt, fill=none, transform shape, scale=0.80] {};
            \node[draw=blue, ellipse, fit=(r1) (r12), inner sep=4pt, fill=none, transform shape, scale=0.80] {};

            \node[draw=none, fill=none, anchor=west] at ($(t1) + (0.1, 0)$) {$x$};
            \node[draw=none, fill=none, anchor=east] at ($(l1) + (-0.2, 0)$) {$a_1$};
            \node[draw=none, fill=none, anchor=west] at ($(r1) + (0.15, 0)$) {$a_4$};
            \node[draw=none, fill=none, anchor=west] at ($(lm1) + (0.25, 0)$) {$a_2$};
            \node[draw=none, fill=none, anchor=west] at ($(rm1) + (0.15, 0)$) {$a_3$};
            
            \node[draw=none, fill=none, shape=rectangle] at (0,-3.5) {\footnotesize Case 4B: Matched one neighbor pair};
        \end{tikzpicture}
    \end{subfigure}
    \hfill
    \begin{subfigure}[b]{0.45\textwidth}
        \centering\medskip
        \begin{tikzpicture}[every node/.style={circle, draw, fill=black, inner sep=1pt}, node distance=1cm]
            \node (t1) at (0,0) {};
            \node (l1) at ($(t1) + (-2.2,-1)$) {};
            \node (r1) at ($(t1) + (2.2,-1)$) {};
            \node (rm1) at ($(t1) + (1.0,-1)$) {};
            \node (lm1) at ($(t1) + (-1.0,-1)$) {};
            \draw (t1) -- (l1);
            \draw (t1) -- (r1);
            \draw (t1) -- (lm1);
            \draw (t1) -- (rm1);
            \draw (l1) -- (lm1);
            \draw (rm1) -- (r1);
            \draw (rm1) -- (rm1);
            \draw[dashed] (l1) -- ++(0, -1);
            \draw[dashed] (lm1) -- ++(0, -1);
            \draw[dashed] (r1) -- ++(0, -1);
            \draw[dashed] (rm1) -- ++(0, -1);

            \node[draw=blue, ellipse, fit=(l1) (lm1), inner sep=4pt, fill=none, transform shape, scale=0.80] {};
            \node[draw=blue, ellipse, fit=(r1) (rm1), inner sep=4pt, fill=none, transform shape, scale=0.80] {};

            \node[draw=none, fill=none, anchor=west] at ($(t1) + (0.15, 0)$) {$x$};
            \node[draw=none, fill=none, anchor=east] at ($(l1) + (-0.25, 0)$) {$a_1$};
            \node[draw=none, fill=none, anchor=west] at ($(r1) + (0.25, 0)$) {$a_4$};
            \node[draw=none, fill=none, anchor=west] at ($(lm1) + (0.25, 0)$) {$a_2$};
            \node[draw=none, fill=none, anchor=east] at ($(rm1) + (-0.25, 0)$) {$a_3$};
            
            \node[draw=none, fill=none, shape=rectangle] at (0,-3) {\footnotesize\it Case 4C: Matched two neighbor pairs};
        \end{tikzpicture}
    \end{subfigure}
    \caption{Degree four vertices; other edges could exist in the picture except those incident to~$x$, and the planar cyclic ordering of the edges at $x$ is not necessarily the same as depicted here}
    \label{fig:neighbors_four}
\end{figure}

\begin{figure}[ht]
    \centering
    \begin{subfigure}[b]{0.45\textwidth}
        \centering
        \begin{tikzpicture}[every node/.style={circle, draw, fill=black, inner sep=1pt}, node distance=1cm]
            \node (t1) at (0,0) {};
            \node (l1) at ($(t1) + (-2.0,-1)$) {};
            \node (r1) at ($(t1) + (2.0,-1)$) {};
            \node (rm1) at ($(t1) + (1.,-1)$) {};
            \node (m1) at ($(t1) + (0.0,-1)$) {};
            \node (lm1) at ($(t1) + (-1.0,-1)$) {};
            \node (r12) at ($(r1) + (0,-1)$) {};
            \node (rm12) at ($(rm1) + (0,-1)$) {};
            \node (l12) at ($(l1) + (0,-1)$) {};
            \node (lm12) at ($(lm1) + (0,-1)$) {};
            \node (m12) at ($(m1) + (0,-1)$) {};
            \draw (t1) -- (l1);
            \draw (t1) -- (r1);
            \draw (t1) -- (lm1);
            \draw (t1) -- (rm1);
            \draw (t1) -- (m1);
            \draw (l1) -- (l12);
            \draw (r12) -- (r1);
            \draw (rm12) -- (rm1);
            \draw (m12) -- (m1);
            \draw (lm12) -- (lm1);
            \draw[dashed] (l12) -- ++(0, -1);
            \draw[dashed] (lm12) -- ++(0, -1);
            \draw[dashed] (r12) -- ++(0, -1);
            \draw[dashed] (rm12) -- ++(0, -1);
            \draw[dashed] (m12) -- ++(0, -1);

            \node[draw=blue, ellipse, fit=(l1) (l12), inner sep=4pt, fill=none, transform shape, scale=0.80] {};
            \node[draw=blue, ellipse, fit=(rm1) (rm12), inner sep=4pt, fill=none, transform shape, scale=0.80] {};
            \node[draw=blue, ellipse, fit=(r1) (r12), inner sep=4pt, fill=none, transform shape, scale=0.80] {};
            \node[draw=blue, ellipse, fit=(m1) (m12), inner sep=4pt, fill=none, transform shape, scale=0.80] {};
            \node[draw=blue, ellipse, fit=(lm1) (lm12), inner sep=4pt, fill=none, transform shape, scale=0.80] {};

            \node[draw=none, fill=none, anchor=west] at ($(t1) + (0.15, 0)$) {$x$};
            \node[draw=none, fill=none, anchor=west] at ($(l1) + (0.15, 0)$) {$a_1$};
            \node[draw=none, fill=none, anchor=west] at ($(r1) + (0.15, 0)$) {$a_5$};
            \node[draw=none, fill=none, anchor=west] at ($(lm1) + (0.15, 0)$) {$a_2$};
            \node[draw=none, fill=none, anchor=west] at ($(rm1) + (0.15, 0)$) {$a_4$};
            \node[draw=none, fill=none, anchor=west] at ($(m1) + (0.15, 0)$) {$a_3$};
            
            \node[draw=none, fill=none, shape=rectangle] at (0,-3.5) {\footnotesize Case 5A: All edges matched outwards};
        \end{tikzpicture}
    \end{subfigure}
    \hfill
    \begin{subfigure}[b]{0.45\textwidth}
        \centering
        \begin{tikzpicture}[every node/.style={circle, draw, fill=black, inner sep=1pt}, node distance=1cm]
            \node (t1) at (0,0) {};
            \node (l1) at ($(t1) + (-2.0,-1)$) {};
            \node (r1) at ($(t1) + (2.0,-1)$) {};
            \node (rm1) at ($(t1) + (1.0,-1)$) {};
            \node (m1) at ($(t1) + (0.0,-1)$) {};
            \node (lm1) at ($(t1) + (-1.0,-1)$) {};
            \node (r12) at ($(r1) + (0,-1)$) {};
            \node (rm12) at ($(rm1) + (0,-1)$) {};
            \node (m12) at ($(m1) + (0,-1)$) {};
            \draw (t1) -- (l1);
            \draw (t1) -- (r1);
            \draw (t1) -- (lm1);
            \draw (t1) -- (rm1);
            \draw (t1) -- (m1);
            \draw (l1) -- (lm1);
            \draw (r12) -- (r1);
            \draw (rm12) -- (rm1);
            \draw (m12) -- (m1);
            \draw[dashed] (l1) -- ++(0, -1);
            \draw[dashed] (lm1) -- ++(0, -1);
            \draw[dashed] (r12) -- ++(0, -1);
            \draw[dashed] (rm12) -- ++(0, -1);
            \draw[dashed] (m12) -- ++(0, -1);

            \node[draw=blue, ellipse, fit=(l1) (lm1), inner sep=4pt, fill=none, transform shape, scale=0.80] {};
            \node[draw=blue, ellipse, fit=(rm1) (rm12), inner sep=4pt, fill=none, transform shape, scale=0.80] {};
            \node[draw=blue, ellipse, fit=(r1) (r12), inner sep=4pt, fill=none, transform shape, scale=0.80] {};
            \node[draw=blue, ellipse, fit=(m1) (m12), inner sep=4pt, fill=none, transform shape, scale=0.80] {};

            \node[draw=none, fill=none, anchor=west] at ($(t1) + (0.1, 0)$) {$x$};
            \node[draw=none, fill=none, anchor=east] at ($(l1) + (-0.2, 0)$) {$a_1$};
            \node[draw=none, fill=none, anchor=west] at ($(r1) + (0.15, 0)$) {$a_5$};
            \node[draw=none, fill=none, anchor=west] at ($(lm1) + (0.25, 0)$) {$a_2$};
            \node[draw=none, fill=none, anchor=west] at ($(rm1) + (0.15, 0)$) {$a_4$};
            \node[draw=none, fill=none, anchor=west] at ($(m1) + (0.15, 0)$) {$a_3$};
            
            \node[draw=none, fill=none, shape=rectangle] at (0,-3.5) {\footnotesize Case 5B: Matched one neighbor pair};
        \end{tikzpicture}
    \end{subfigure}
    \hfill
    \begin{subfigure}[b]{0.45\textwidth}
        \centering\medskip
        \begin{tikzpicture}[every node/.style={circle, draw, fill=black, inner sep=1pt}, node distance=1cm]
            \node (t1) at (0,0) {};
            \node (l1) at ($(t1) + (-2.5,-1)$) {};
            \node (r1) at ($(t1) + (2,-1)$) {};
            \node (rm1) at ($(t1) + (1.0,-1)$) {};
            \node (m1) at ($(t1) + (0.0,-1)$) {};
            \node (lm1) at ($(t1) + (-1.5,-1)$) {};
            \node (r12) at ($(r1) + (0,-1)$) {};
            \draw (t1) -- (l1);
            \draw (t1) -- (r1);
            \draw (t1) -- (lm1);
            \draw (t1) -- (rm1);
            \draw (t1) -- (m1);
            \draw (l1) -- (lm1);
            \draw (r12) -- (r1);
            \draw (rm1) -- (m1);
            \draw[dashed] (l1) -- ++(0, -1);
            \draw[dashed] (lm1) -- ++(0, -1);
            \draw[dashed] (r12) -- ++(0, -1);
            \draw[dashed] (rm1) -- ++(0, -1);
            \draw[dashed] (m1) -- ++(0, -1);

            \node[draw=blue, ellipse, fit=(l1) (lm1), inner sep=4pt, fill=none, transform shape, scale=0.80] {};
            \node[draw=blue, ellipse, fit=(r1) (r12), inner sep=4pt, fill=none, transform shape, scale=0.80] {};
            \node[draw=blue, ellipse, fit=(m1) (rm1), inner sep=4pt, fill=none, transform shape, scale=0.80] {};

            \node[draw=none, fill=none, anchor=west] at ($(t1) + (0.1, 0)$) {$x$};
            \node[draw=none, fill=none, anchor=east] at ($(l1) + (-0.2, 0)$) {$a_1$};
            \node[draw=none, fill=none, anchor=west] at ($(r1) + (0.15, 0)$) {$a_5$};
            \node[draw=none, fill=none, anchor=west] at ($(lm1) + (0.25, 0)$) {$a_2$};
            \node[draw=none, fill=none, anchor=west] at ($(rm1) + (0.25, 0)$) {$a_4$};
            \node[draw=none, fill=none, anchor=east] at ($(m1) + (-0.25, 0)$) {$a_3$};
            
            \node[draw=none, fill=none, shape=rectangle] at (0,-3.5) {\footnotesize Case 5C: Matched two neighbor pairs};
        \end{tikzpicture}
    \end{subfigure}
    \caption{Degree five vertices; other edges could exist in the picture except those incident to~$x$, and the planar cyclic ordering of the edges at $x$ is not necessarily the same as depicted here}
    \label{fig:neighbors_five}
\end{figure}
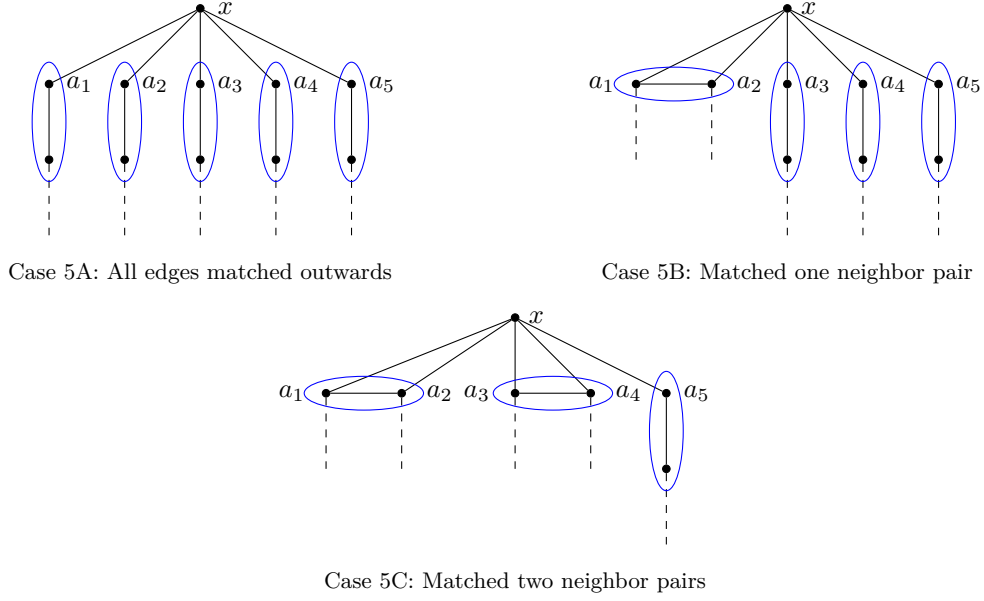

As depicted, there are $11$ such configurations without isolated vertices, and we further divide them into three groups based on how they will be approached in the coming proofs:
\begin{itemize}
    \item \emph{Contractible cases:} Case~1A, Case~2A, Case~3A, Case~3B, Case~4B, Case~5A, Case~5C.
    \item \emph{Contractible with Kempe chaining, or Kempe-contractible cases:} Case~4A, Case~5B.
    \item \emph{Difficult cases (marked in italic in the pictures):} Case~2B, Case~4C.
\end{itemize}

\section{Resolving the contractible cases}\label{sec:controlled}

The task of this section is to provide the full proof of an ``easier part'' of \Cref{thm:main} which simply builds on the ideas of \Cref{obs:perfect} and on \Cref{cor:lowdeg_free}.
However, since this part does not resolve all of the configurations enumerated in \Cref{sec:classification} (the two difficult cases remain to be resolved later),
the formulation of this partial result is quite technical in order to transparently allow for future local modifications concerning the difficult cases.

For any planar graph $H$ with a fixed planar embedding and a vertex $x\in V(H)$, we define the following operation.
The \emph{wheel augmentation at $x$} adds to $H$ a new edge between every non-adjacent pair of consecutive neighbors of~$x$ in this embedding of~$H$.
Informally, this construction ``{makes a wheel (graph)}'' around $x$.
If, for some independent set of vertices $X\subseteq V(H)$, $\>H'$ is the graph obtained from $H$ by a sequence of wheel augmentations at every vertex $x\in X$ (the order of these augmentations does not matter since $X$ is independent), then $H'$ is called the \emph{wheel augmented graph} of~$H$ at~$X$.
Independent wheel augmentations trivially preserve planarity of $H$\onlyapx{ (Appendix:~\Cref{lemma_wheel_can})}.

\begin{toappendix}
\begin{observation}
\label{lemma_wheel_can}
Let $H$ be a plane embedded graph and $X\subseteq V(H)$ an independent set.
Then the wheel augmented graph $H'$ of $H$ at $X$ is again a plane embedded graph, and $H'$ is uniquely determined up to isomorphism.
\end{observation}

\begin{proof}
Consider any $x\in X$. Then, in the process of wheel augmentation of $H$ at $X$, every two consecutive neighbors $y$ and $z$ of $x$ in $H$ share a face together and with $x$.
So, the edge $yz$ can be drawn uncrossed within this face.
Furthermore, as both $y,z$ are neighbors of $x$ and $X$ is independent, neither of $y,z$ is a subject of subsequent wheel augmentation steps leading to~$H'$.
Therefore, no conflict arises between wheel augmentation steps at distinct vertices of~$X$, the process preserves planarity and defines $H'$ uniquely up to isomorphism.
\end{proof}

\end{toappendix}

Let $M$ be a maximum matching of $G$ and $H=H(G,F)$ be the contracted graph of $G$ for some~$F\subseteq E(G)$.
Note that, unlike in \Cref{obs:perfect}, the edge set $F$ which we contract in~$G$ while constructing our coloring,
will not always be equal to the matching~$M$ -- in specific cases we may choose $F$ to be ``slightly different'' from~$M$, but the difference will occur only in the closed neighborhood of vertices unmatched by~$M$.
While the proof of \Cref{lem:controlled} itself only needs to contract the matching edges for the purpose of constructing our coloring,
proofs of the difficult cases (see \Cref{sec:difficult2}) will require to contract, e.g., three vertices on a path of length~$2$ into one supervertex.
In the first reading of the statements and proofs below, it is safe to assume ``$F=M$'' and consider the fine details regarding $F$ later on.

In the next definition we aim to capture changes of colorings of $H$ which ``leave intact'' (up to a permutation of the colors)
vertices of $H$ coming by contractions from chosen components of the graph $G[D(G)]$ defined with respect to~$M$.
For any two colorings $c$ and $\bar c$ of $H$ and a subset $U$ of vertices of $G$ unmatched by~$M$, 
we say that $\bar c$ is a \emph{$U$-protecting recoloring} of~$c$ if the following holds for every vertex $x\in U$ and the connected component $D_x$ of $G[D(G)]$ containing $x$:
if $V_x\subseteq V(H)$ are the vertices into which some vertex of $D_x$ is contracted via~$F$, then the coloring $\bar c$ differs from $c$ on $V_x$ only by a permutation of the colors.

The next lemma, \Cref{lem:controlled}, informally claims that we ``can resolve'' -- in the sense of obtaining a CFON $4$-coloring (analogously to \Cref{obs:perfect}),
all matched vertices of $G$ and also all unmatched vertices of $G$, except those whose local configuration is one of the two difficult cases (Case~2B, Case~4C) or which are in a very specific local situation.
Importantly, the cumbersome formulation of \Cref{lem:controlled} gives us the power to resolve these resolvable unmatched vertices \emph{after} we (in yet undetermined way) resolve the difficult cases in the coming sections, and without conflicts or discrepancies between the cases.
The formulation also ensures that nothing else than the cases described in \Cref{lem:controlled}\,c) remains ``unresolved''.

\begin{lemma}\label{lem:controlled}
Let $G$ be a planar graph without isolated vertices. There exists a maximum matching $M$ of $G$ and a subset $U\subseteq V(G)\setminus V(M)$ of unmatched vertices such that the following hold.
Let $W=V(G)\setminus(U\cup V(M))$ be the (independent) set of unmatched vertices complementary to~$U$, and
$G_W$ denote (with respect to~$M$) the union of the subgraphs $G_x$ for $x\in W$.
\begin{enumerate}[a)]
 \item Every vertex of $V(G)\setminus V(M)$ is of degree at most~$5$ in~$G$ (cf.~\Cref{cor:lowdeg_free}).
 \item For any set of edges $F\subseteq E(G)$ such that $F\cap E(G_W)=M\cap E(G_W)$, let $H=H(G,F)$ be the corresponding contracted graph and $H'$ be the wheel augmented graph of $H$ at~$W$.
Then every proper $4$-coloring $c_1$ of $H'$ admits a $U$-protecting recoloring to a proper $4$-coloring $c$ of $H'$ such that,
denoting by $c'$ the $4$-coloring of $G$ into which the coloring $c$ lifts (cf.~\Cref{obs:perfect}),
every vertex $x\in W$ has a neighbor $u$ in~$G$ whose color $c'(u)$ is unique in the open neighborhood~$N(x)$.
\item For every vertex $x\in U$ and every $z\in S(x)$ (including $z=x$) such that the degree of $z$ in~$G$ is at most~$5$, let $M'$ be any matching obtained from $M$ in one of the following ways:
\begin{itemize}\vspace*{-1ex}%
\item $M' := M \triangle E(P)$ for an even $M$-alternating $x$--$z$ path in~$G$ (switching from $x$ to~$z$),
\item $M' := M \triangle E(Q)$ for $x=z$ and an $M$-alternating cycle $Q$ which is contained in the same connected component of $G[D(G)]$ as~$x$ (switching a cycle ``nearby~$x$'').
\end{itemize}
Then, with respect to~$M'$;
\begin{itemize}\vspace*{-1ex}%
\item the configuration at $z$ is one of the two difficult cases (Case~2B, Case~4C), or
\item the configuration at $z$ is one of the two Kempe-contractible cases (Case~4A, Case~5B) and there exist $4$ pairwise disjoint paths between the neighbors of $z$ and the set~$A(G)$.
\end{itemize}
\end{enumerate}
\end{lemma}

\begin{proof}
We fix a planar embedding of~$G$.
For a maximum matching $M'$ of $G$, we say that an unmatched vertex $v\in V(G)\setminus V(M')$ is \emph{bad} if the configuration at $v$ with respect to $M'$ is either one of the two difficult cases,
or one of the two Kempe-contractible cases with existing $4$ pairwise disjoint paths between $N(v)$ and~$A(G)$.
We pick $M$ as a maximum matching of $G$ such that all unmatched vertices are of degree at most~$5$ (\Cref{cor:lowdeg_free}) and that the set $U$ of bad vertices unmatched by $M$ is minimal over all such~$M$.
So, in particular, condition a) holds~true.

Verification of condition c) is straightforward;
if there were $x\in U$ and $z\in S(x)$ or $z=x$ of degree at most~$5$, and a switched matching $M'$ as in c), such that $z$ is not bad with respect to~$M'$, then we would consider $M'$ in place of~$M$.
By \Cref{clm:Gx_structure}, this switch to~$M'$ does not change the local matching configuration at any of the unmatched vertices other than $z$, and so it contradicts our minimality assumption on~$U$.
It thus remains to verify condition b) for all $x\in W$ (the non-bad unmatched vertices).

\medskip
Let $c_1:V(H')\to\{1,2,3,4\}$ be an arbitrary proper coloring of the graph~$H'$ which is a wheel augmented graph of planar $H=H(G,F)$, and so $H'$ is planar and \Cref{thm:FCT} applies.
Pick an unmatched vertex $x\in W$ such that the configuration at $x$ is one of the contractible cases in the matching $M$.
By the assumptions, $F$ coincides with $M$ in the neighborhood of~$x$.
Observe in Figures~\ref{fig:matching_one}--\ref{fig:neighbors_five} that in each of these cases at least one of the adjacent matching edges $f\in M$ (and so~$f\in F$ as well)
has exactly one of its ends adjacent to $x$ in $G$ (matched outwards in the picture), 
and the degree of $[x]$ is at most~$3$ in contracted $H$ except in Case~5A where it is~$5$.

Let $[f]$ denote the supervertex of $f$ in the contracted graph $H$. 
If $\deg_{H}(x)\leq3$, then $[f]$ is adjacent to every other neighbor of $x$ in $H'$, and so $[f]$ has a unique color among the neighbors of $x$ in any proper coloring $c_1$ of~$H'$.
And since exactly one end of $f$ is a neighbor of $x$ in $G$, the lifted coloring $c'$ of $G$ assigns this color $c_1([f])$ to a unique vertex in the neighborhood of $x$ in~$G$, too.
In Case~5A, the neighbors of $x$ in $H'$ span a $5$-cycle, and thus they receive $3$ distinct colors in any proper $4$-coloring $c_1$ of~$H'$ and one of these $3$ colors is unique.
Since all neighbors of $x$ are matched outwards in Case~5A, this color is also unique for $x$ in the lifted coloring $c'$ of $G$.
We have verified condition b) for such $x$ and any~$c=c_1$.

\medskip
Now, pick an unmatched vertex $x\in W$ such that the configuration at $x$ is one of the two Kempe-contractible cases.
Our aim is to show that, since $x$ is not bad $(x\not\in U$), condition~b) holds true for $x$.
Again, we will use that $F$ coincides with the matching $M$ in the neighborhood of~$x$.
So, in each of Case~4A and Case~5B the degree of $[x]$ in $H$ is $4$, and whenever a $4$-coloring $c_1$ assigns $3$ distinct colors to the neighbors of $[x]$ in~$H'$,
we get a unique color among them, which yields a unique color among the neighbors of $x$ in $G$ in the lifted coloring~$c'$ for any $c=c_1$.

We may hence further assume that $c_1$ gives only $2$ colors to the $4$ neighbors of $[x]$ in~$H'$ in a proper way.
More specifically, and without loss of generality, let $t_1,t_2,t_3,t_4$ be the neighbors of $x$ in $H'$ in this cyclic order, and $c_1(t_1)=c_1(t_3)=1$ and $c_1(t_2)=c_1(t_4)=2$.
Let $D_0$ be the connected component of $G[D(G)]$ containing $x$, and let $D_0'$ result from $D_0$ by contracting the edges of $M\cap E(D_0)$.
Furthermore, note that $M\cap E(D_0)=F\cap E(D_0)$ (so $D_0'\subseteq H'$), all vertices of $D_0$ except $x$ are matched by $M\cap E(D_0)$, 
and no edge of $F\setminus M$ is incident to~$D_0$ by the assumption $F\cap E(G_W)=M\cap E(G_W)$.

The next arguments are illustrated in \Cref{fig:4Kempechain}.

\begin{figure}[ht]
    \centering
    \begin{subfigure}[b]{0.47\textwidth}
        \centering
        \begin{tikzpicture}[every node/.style={circle, draw, fill=black, inner sep=1pt}, node distance=1cm]
            \node (t1) at (0,0) {};
            \node (l1) at ($(t1) + (-1.5,-1)$) {};
            \node (r1) at ($(t1) + (1.5,-1)$) {};
            \node (rm1) at ($(t1) + (0.5,-1)$) {};
            \node (lm1) at ($(t1) + (-0.5,-1)$) {};
            \node (r12) at ($(r1) + (0,-1)$) {};
            \node (rm12) at ($(rm1) + (0,-1)$) {};
            \node (l12) at ($(l1) + (0,-1)$) {};
            \node (lm12) at ($(lm1) + (0,-1)$) {};
            \draw (t1) -- (l1);
            \draw (t1) -- (r1);
            \draw (t1) -- (lm1);
            \draw (t1) -- (rm1);
            \draw (l1) -- (l12);
            \draw (r12) -- (r1);
            \draw (rm12) -- (rm1);
            \draw (lm12) -- (lm1);
            \draw[densely dashed] (l12) -- ++(0, -1) (lm12) -- ++(0, -1) (r12) -- ++(0, -1) (rm12) -- ++(0, -1);
	    \node (q12) at ($(l12)-(1,0)$) {};
	    \node (q13) at ($(q12)-(0,1)$) {};
	    \node (q1) at ($(l1)-(1,0)$) {};
	    \node (q2) at ($(q1)-(1,0)$) {};
	    \node (q3) at ($(q2)-(0,1)$) {};
	    \node (q4) at ($(q3)-(0,1)$) {};
            \draw[densely dashed] (l12)--(q12)--(q13)-- ++(0,-1) (q13)--(q4) (l1)--(q1)--(q2)--(q3)--(q4)-- ++(0,-1);
	    \draw (q12) -- (q13) (q1) -- (q2) (q3) -- (q4) ;

	\begin{scope}[on background layer]
            \node[draw=blue, ellipse, fit=(l1) (l12), inner sep=4pt, fill=green!30, transform shape, scale=0.80] {};
            \node[draw=blue, ellipse, fit=(rm1) (rm12), inner sep=4pt, fill=green!30, transform shape, scale=0.80] {};
            \node[draw=blue, ellipse, fit=(lm1) (lm12), inner sep=4pt, fill=red!30, transform shape, scale=0.80] {};
            \node[draw=blue, ellipse, fit=(r1) (r12), inner sep=4pt, fill=red!30, transform shape, scale=0.80] {};
	    \draw[blue] ($(l1)+(0.15,0)$)--($(lm1)-(0.15,0)$) ($(lm1)+(0.15,0)$)--($(rm1)-(0.15,0)$);
	    \draw[blue] ($(rm1)+(0.15,0)$)--($(r1)-(0.15,0)$) ($(l1)+(0,0.3)$) to[bend left=83] ($(r1)+(0,0.3)$);
            \node[draw=blue, ellipse, fit=(q12) (q13), inner sep=4pt, fill=orange!40, transform shape, scale=0.80] {};
            \node[draw=blue, ellipse, fit=(q1) (q2), inner sep=4pt, fill=orange!40, transform shape, scale=0.80] {};
            \node[draw=blue, ellipse, fit=(q3) (q4), inner sep=4pt, fill=green!30, transform shape, scale=0.80] {};
            \node[draw, inner sep=46, shape=rectangle, rounded corners, densely dotted, semithick, fill=none] at (-2.7,-2.3) {};
	\end{scope}

            \node[draw=none, fill=none, anchor=west] at ($(t1) + (0.15, 0)$) {$x$};
            \node[draw=none, fill=none, anchor=east] at ($(l1) - (0.15, 0.5)$) {$t_1$};
            \node[draw=none, fill=none, anchor=east] at ($(r1) - (0.15, 0.5)$) {$t_4$};
            \node[draw=none, fill=none, anchor=east] at ($(lm1) - (0.15, 0.5)$) {$t_2$};
            \node[draw=none, fill=none, anchor=east] at ($(rm1) - (0.15, 0.5)$) {$t_3$};
            \node[draw=none, fill=none] at ($(l12)-(0,1.5)$) {$L_1$};

        \end{tikzpicture}
	\bigskip
        \subcaption{\footnotesize The $(1,3)$-Kempe chain $L_1$ containing~$t_1$ in $H'-x$ does not include $t_3$ and does not leave the component~$D_0'$.
		Switching colors green and orange within $L_1$ yields a unique color in the neighborhood of~$x$.} 
    \end{subfigure}
    \hfill
    \begin{subfigure}[b]{0.47\textwidth}
        \centering
        \begin{tikzpicture}[every node/.style={circle, draw, fill=black, inner sep=1pt}, node distance=1cm]
            \node (t1) at (0,0) {};
            \node (l1) at ($(t1) + (-2.0,-1)$) {};
            \node (r1) at ($(t1) + (2.0,-1)$) {};
            \node (rm1) at ($(t1) + (1.0,-1)$) {};
            \node (m1) at ($(t1) + (0.0,-1)$) {};
            \node (lm1) at ($(t1) + (-1.0,-1)$) {};
            \node (r12) at ($(r1) + (0,-1)$) {};
            \node (rm12) at ($(rm1) + (0,-1)$) {};
            \node (m12) at ($(m1) + (0,-1)$) {};
            \draw (t1) -- (l1);
            \draw (t1) -- (r1);
            \draw (t1) -- (lm1);
            \draw (t1) -- (rm1);
            \draw (t1) -- (m1);
            \draw (l1) -- (lm1);
            \draw (r12) -- (r1);
            \draw (rm12) -- (rm1);
            \draw (m12) -- (m1);
            \draw[densely dashed] (l1) -- ++(0, -1) (lm1) -- ++(0, -1) (r12) -- ++(0, -1) (rm12) -- ++(0, -1) (m12) -- ++(0, -1);
	    \draw[densely dashed, decorate, decoration={zigzag}] ($(m12)+(0,-1)$) to[bend right=53] ($(r12)+(0,-1)$);
	    \node (l12) at ($(l1)-(0,1)$) {};
	    \node (l13) at ($(l12)-(0,1)$) {};
	    \node (q2) at ($(l1)-(1,0)$) {};
	    \node (q3) at ($(q2)-(1,0)$) {};
	    \node (q4) at ($(q2)-(0,1)$) {};
	    \node (q5) at ($(q4)-(0,1)$) {};
	    \draw[densely dashed] (l12)-- ++(0,-1) (lm1)-- ++(0,-1) (q2)--(q4) (q2)-- ++(0,1) (q4)-- ++(-1,0) (q5)-- ++(-1,0);
	    \draw (l12) -- (l13) (q2) -- (q3) (q4)--(q5) ;

	\begin{scope}[on background layer]
            \node[draw=blue, ellipse, fit=(l1) (lm1), inner sep=4pt, fill=red!30, transform shape, scale=0.80] {};
            \node[draw=blue, ellipse, fit=(rm1) (rm12), inner sep=4pt, fill=red!30, transform shape, scale=0.80] {};
            \node[draw=blue, ellipse, fit=(r1) (r12), inner sep=4pt, fill=green!30, transform shape, scale=0.80] {};
            \node[draw=blue, ellipse, fit=(m1) (m12), inner sep=4pt, fill=green!30, transform shape, scale=0.80] {};
	    \draw[blue] ($(m1)+(0.15,0)$)--($(rm1)-(0.15,0)$) ($(lm1)+(0.3,0)$)--($(m1)-(0.15,0)$);
	    \draw[blue] ($(rm1)+(0.15,0)$)--($(r1)-(0.15,0)$) ($(l1)+(0,0.16)$) to[bend left=63] ($(r1)+(0,0.3)$);
            \node[draw=blue, ellipse, fit=(l12) (l13), inner sep=4pt, fill=yellow!40, transform shape, scale=0.80] {};
            \node[draw=blue, ellipse, fit=(q2) (q3), inner sep=4pt, fill=yellow!40, transform shape, scale=0.80] {};
            \node[draw=blue, ellipse, fit=(q4) (q5), inner sep=4pt, fill=red!30, transform shape, scale=0.80] {};
            \node[draw, inner sep=52, shape=rectangle, rounded corners, densely dotted, semithick, fill=none] at (-2.5,-1.7) {};
	\end{scope}

            \node[draw=none, fill=none, anchor=west] at ($(t1) + (0.15, 0)$) {$x$};
            \node[draw=none, fill=none] at ($(l1) + (0.5,-0.5)$) {$t_4$};
            \node[draw=none, fill=none, anchor=east] at ($(r1) - (0.15, 0.5)$) {$t_3$};
            \node[draw=none, fill=none, anchor=east] at ($(m1) - (0.15, 0.5)$) {$t_1$};
            \node[draw=none, fill=none, anchor=east] at ($(rm1) - (0.15, 0.5)$) {$t_2$};
            \node[draw=none, fill=none] at ($(r12) - (0.15, 1.5)$) {$P_1$};
            \node[draw=none, fill=none] at ($(lm1)-(0,2)$) {$L_2$};
            
        \end{tikzpicture}
        \subcaption{\footnotesize The $\{1,3\}$-colored path $P_1$ forms with $x$ a cycle separating the $(2,4)$-Kempe chains of $H'-x$ into two
		groups, and in one of them (say, $L_2$ including the contracted vertex $t_4$) we switch colors red and yellow.
		This again yields a unique color (red of~$t_2$) in the neighborhood of~$x$ in~$G$.} 
    \end{subfigure}
    \caption{Illustrating the proof of \Cref{lem:controlled}: Kempe-chain exchanges in Case~4A and Case~5B -- 
	both cases are covered by the same collection of arguments since the local situation looks ``the same'' in the contracted graph~$H(G,M)$.
	The wheel augmentation at $x$ in $H'$ is depicted by blue edges between the vertices $t_1,t_2,t_3,t_4$ which result by contractions of the depicted matching edges of~$G$.
	Colors $1$ and $2$ of the neighbors $t_1,t_2,t_3,t_4$ of $x$ in $H'$ are in order green and red in the picture.}
    \label{fig:4Kempechain}
\end{figure}
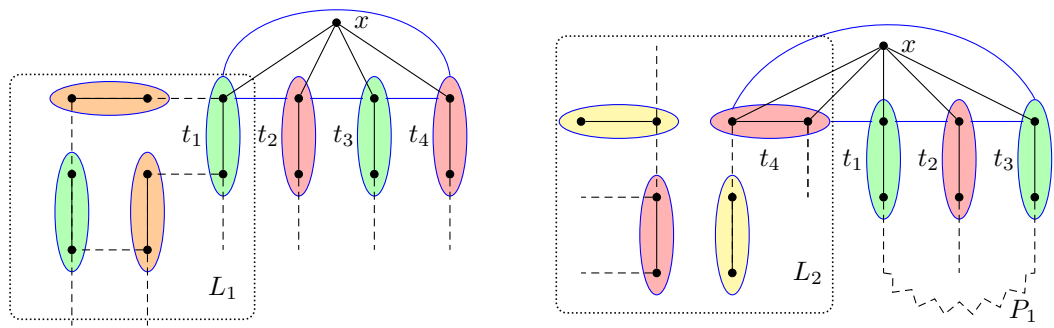

Let $L_1$ denote the $(1,3)$-Kempe chain in $H'-x$ containing~$t_1$.
If $L_1\subseteq V(D_0')$ and $t_3\not\in L_1$, then we exchange the colors $1$ and $3$ within $L_1$; more precisely, we define a new coloring $c_2$ of $H'$
by setting $c_2(v)=c_1(v)$ for $v\in V(H'-x)\setminus L_1$, $c_2(x)=4$ and $c_2(v)=4-c_1(v)$ for $v\in L_1$.
Then $c_2$ is again proper on $H'$ and $c_2$ assigns $3$ distinct colors to the neighbors of $x$, and hence condition b) holds true for $x$ in $c_2$, as argued above.
Importantly, the new coloring $c_2$ is a $U$-protecting recoloring of our (so far arbitrary) $c_1$ by \Cref{clm:Gx_structure}, since $V(D_0)\cap V(G_u)=\varnothing$ for the other unmatched vertices $u\not=x$.
In particular, the validity of b) is preserved (from $c=c_1$ to $c=c_2$) for the other unmatched vertices.

We can perform a similar $U$-protecting Kempe-chain exchange in the case that $L_1$ contains a path $P_1'\subseteq H'-[x]$ from $t_1$ to $t_3$ 
which lifts into a $t_1$--$t_3$ path $P_1$ in $G-x$ such that $|V(P_1)\cap A(G)|\leq1$.
In this case we let $L_2$ be the union of all $(2,4)$-Kempe chains in $H'-[x]$ which are drawn in the face of the cycle $P_1'+x$ which contains~$t_4$.
We now exchange the colors $2$ and $4$ within~$L_2$; we define a new coloring by $c_2(v)=c_1(v)$ for $v\in V(H'-[x])\setminus L_2$, $c_2(x)=3$ and $c_2(v)=6-c_1(v)$ for $v\in L_2$, which is proper on $H'$.
Again, we arrive at the case of $3$ colors in the neighborhood of $x$ solved above, thus verifying b) for~$x$.
Furthermore, since $|V(P_1)\cap A(G)|\leq1$, no other component of $G[D(G)]$ is ``split'' by the cycle $P_1+x$ drawn within~$G$,
and since the vertices of $P_1'$ do not participate in our color exchange, the other components of $G[D(G)]$ are only possibly affected by a permutation of colors.
Hence, $c_2$ indeed is a $U$-protecting recoloring of $c_1$ by definition and the validity of b) is not affected for the other unmatched vertices.

Since there are no edges between $V(D_0)$ and $C(G)$ by \Cref{clm:GallaiEdmonds}, the only remaining case is that the lift of $L_1$ to $G$ 
contains a path $Q_1$ from a neighbor $a_1$ of $x$ (where $a_1$ contracts into $t_1$ in $H'$) to the set $A(G)$.
We may further assume that $Q_1$ is contained in $D_0$ except its end in~$A(G)$.
By the previous arguments and symmetry, analogous paths $Q_2,Q_3,Q_4$ exist from other neighbors of $x$ which contract to $t_2,t_3,t_4$ in this order.
Now $Q_1\cup Q_3$ is disjoint from $Q_2\cup Q_4$ by their definition via Kempe chains.
If, say, $Q_1$ intersected $Q_3$, then there would be a $t_1$--$t_3$ path $P_1\subseteq Q_1\cap Q_3$ as considered in the previous paragraph.
Consequently, $Q_1,Q_2,Q_3,Q_4$ are pairwise disjoint and $x$ is thus bad, contradicting our assumption~$x\not\in U$.
\end{proof}

\section{Global planarity constraints}\label{sec:density}

With \Cref{lem:controlled} at hand, for the rest of the paper, we only need to focus on resolving the remaining difficult local matching configurations 
(cf.~\Cref{lem:controlled}\,c)) -- Case~2B and Case~4C, and on the subcase of having $4$ disjoint paths between $N(x)$ and~$A(G)$.

In their analysis, planarity of the graph $G$ will play an important role captured via the following lemma.
\onlyapx{Its proof and related comments and consequences are left for the Appendix.}

\begin{toappendix}
\label{apx:density}

\end{toappendix}

\begin{lemma2rep}\apxmark{apx:density}%
\label{lemma_gadgets}
Let $G_1$ be a simple graph with a vertex partition $V(G_1)=A\cup B$ where $|A|=k$.
Assume that the vertices of $B$ can be partitioned into exactly $k+1$ disjoint subsets, each of which forms, as a subgraph of~$G_1$, a \emph{gadget} with some vertices of $A$ (not necessarily distinct between the gadgets) that is of one of the following types:
\begin{enumerate}
    \item[\textbf{(I)}] a single vertex $b\in B$ of degree $4$, adjacent to four vertices of $A$;
    \item[\textbf{(II)}] a copy of $K_{2,3}$ with bipartition $(X,Y)$, where $X\subseteq B$, $|X|=2$, and $Y\subseteq A$, $|Y|=3$;
    \item[\textbf{(III)}] a triangle $T$ on three vertices of $B$, together with two vertices $a_1,a_2\in A$ such that each of $a_1,a_2$ is adjacent to all three vertices of $T$.
\end{enumerate}\smallskip
Then $G_1$ is not planar.
\end{lemma2rep}

\begin{proof}
Suppose for a contradiction that $G_1$ is planar, and fix a planar embedding.

We refer to the triangles arising from gadgets of type \textbf{(III)} as \emph{triangle gadgets}. For every triangle gadget $T$, the two vertices of $A$ adjacent to all three vertices of $T$ cannot lie in the same face of $T$. Indeed, once one such vertex is embedded inside $T$ and connected to all three vertices of $T$, the interior of $T$ is subdivided into three triangular faces, and no second vertex in the same face can be adjacent to all three vertices of $T$ without creating a crossing. Thus, in any planar embedding of a gadget of type \textbf{(III)}, one of these two vertices lies in the interior of $T$ and the other in the exterior. Each such triangle $T$ separates the plane into an interior and an exterior region. If a vertex lies strictly inside $T$, then any edge from it to a vertex outside $T$ would cross $T$, contradicting planarity. Hence any gadget lying strictly inside $T$ is adjacent only to vertices of $A$ that also lie inside $T$.

In particular, the interiors of triangle gadgets form a laminar family: for any two such triangles, their interiors are either disjoint or one is contained in the other.

\medskip

For a triangle gadget $T$, let $A(T)$ be the set of vertices of $A$ lying strictly inside $T$, and let $g(T)$ denote the number of gadgets whose $B$-vertices lie in the interior of $T$, including $T$ itself. We claim that
\begin{equation}
    g(T)\le |A(T)|.
    \label{eq:main}
\end{equation}

We prove \eqref{eq:main} by induction on the number of triangle gadgets contained in the interior~of~$T$.

\medskip
\noindent\textbf{Base case.}
Assume that $T$ contains no smaller triangle gadget. Let $x$ and $y$ be the numbers of gadgets of type \textbf{(I)} and \textbf{(II)} inside $T$, and set $a:=|A(T)|$. Then
\[
g(T)=1+x+y.
\]

All these $x+y$ gadgets are adjacent only to vertices of $A(T)$. Consider the bipartite subgraph $H$ induced by $A(T)$ and these gadgets. Each gadget of type \textbf{(I)} contributes one vertex of $B$, and each gadget of type \textbf{(II)} contributes two vertices of $B$. Hence
\[
|V(H)|=a+x+2y,
\qquad
|E(H)|=4x+6y.
\]

Since $H$ is planar and bipartite, we have $|E(H)|\le 2|V(H)|-4$ whenever $x+y>0$. Therefore
\[
4x+6y \le 2(a+x+2y)-4,
\]
which simplifies to
\[
x+y \le a-2.
\]
Thus
\[
g(T)=1+x+y \le a-1 < a.
\]

If $x+y=0$, then $g(T)=1$. By the observation above, one of the two vertices of $A$ adjacent to all three vertices of $T$ lies in the interior of $T$. Hence $a=|A(T)|\ge 1$, and therefore $g(T)\le a$.

\medskip
\noindent\textbf{Inductive step.}
Suppose $T$ contains smaller triangle gadgets, and assume \eqref{eq:main} holds for all of them. Let $T_1,\dots,T_m$ be the maximal triangle gadgets strictly inside $T$, whose interiors are pairwise disjoint. Define
\[
A_0 := A(T)\setminus \bigcup_{i=1}^m A(T_i).
\]

Let $x$ and $y$ be the numbers of gadgets of type \textbf{(I)} and \textbf{(II)} inside $T$ but outside all $T_i$. By separation, these gadgets are adjacent only to vertices of $A_0$. Moreover, by the observation above, one of the two vertices of $A$ adjacent to all three vertices of $T$ lies in the interior of $T$. Since the triangles $T_1,\dots,T_m$ are maximal inside $T$, this vertex belongs to $A_0$. In particular, $|A_0|\ge 1$.

As in the base case, considering the planar bipartite graph induced by $A_0$ and these gadgets, we obtain
\[
1+x+y \le |A_0|.
\]

Every gadget contained in $T$ is either $T$ itself, one of these $x+y$ gadgets, or lies inside some $T_i$. Hence
\[
g(T)=1+x+y+\sum_{i=1}^m g(T_i).
\]
By the induction hypothesis,
\[
g(T_i)\le |A(T_i)|,
\]
so
\[
g(T)\le 1+x+y+\sum_{i=1}^m |A(T_i)|
\le |A_0|+\sum_{i=1}^m |A(T_i)|
= |A(T)|.
\]

This proves \eqref{eq:main}.

\medskip

Finally, let $T_1,\dots,T_r$ be the maximal triangle gadgets in $G_1$, whose interiors are pairwise disjoint. Let $A_0$ be the set of vertices of $A$ outside all of them, and let $x,y$ be the numbers of gadgets of type \textbf{(I)} and \textbf{(II)} outside all $T_i$. As before, these satisfy
\[
x+y \le |A_0|.
\]

Every gadget is either one of these $x+y$ outer gadgets or lies inside some $T_i$, so
\[
k+1 = x+y+\sum_{i=1}^r g(T_i).
\]
By the inequalities above,
\[
k+1
\le x+y+\sum_{i=1}^r |A(T_i)|
\le |A_0|+\sum_{i=1}^r |A(T_i)|
= |A|
= k,
\]
a contradiction.

Therefore $G_1$ is not planar.
\end{proof}

We will use \Cref{lemma_gadgets} as follows. Note that the subcase of having $4$ disjoint paths between $N(x)$ and~$A(G)$ in the Kempe-contractible cases of \Cref{lem:controlled}\,c) readily yields a minor $G_1$ of $G$ containing one of the gadget types of \Cref{lemma_gadgets} -- type \textbf{(I)}.
In this specific application, the set $A$ is the Gallai--Edmonds set $A(G_x)$ of a particular unmatched vertex $x$, 
while the vertex $b\in B$ is the contraction of the component of $G[D(G_x)]$ containing~$z$.

Analogously, for any unmatched vertex $x$ and each of the difficult configurations of \Cref{lem:controlled}\,c) around~$x$, we will either be able to locally find a possibly another matching and a coloring of the contracted graph satisfying also the condition of \Cref{lem:controlled}\,b),
or we will show that \emph{each} of the factor-critical components of $G[D(G_x)]$ yields (as a minor on the $B$-side) one of the three gadgets of \Cref{lemma_gadgets}.
Since, by \Cref{clm:Gx_structure}, there are $|A(G_x)|+1=k+1$ such components, and hence gadgets, \Cref{lemma_gadgets} would then contradict planarity of~$G_x$.

\begin{toappendix}

A handful of upcoming claims will help us with finding the gadgets of \Cref{lemma_gadgets}.

\begin{lemma}
\label{lemma_6deg_neighbors}
Let $G$ be a simple planar graph and let $H \subseteq G$ be a nonempty induced subgraph. Put
\[
X := N_G(V(H)) \setminus V(H).
\]
If
\[
\frac{1}{|V(H)|} \sum_{v \in V(H)} d_G(v) \ge 6,
\]
then $|X| \ge 4$.
\end{lemma}

\begin{proof}
Let $n := |V(H)|$, $k := |X|$, and let $m := |E(H)|$. Let $c$ be the number of edges of $G$ joining a vertex of $H$ to a vertex of $X$.

Since $H$ is induced, every edge of $G$ with at least one endpoint in $V(H)$ is either an edge of $H$ or an edge joining $H$ to $X$. Hence
\[
\sum_{v \in V(H)} d_G(v) = 2m + c.
\]
By assumption,
\[
2m + c \ge 6n. \tag{1}
\]

We first dispose of the cases $n \le 2$.

If $n=1$, then $m=0$, so (1) gives $c\ge 6$. Since every edge counted by $c$ joins the unique vertex of $H$ to a distinct vertex of $X$, we have $|X| \ge 6$.

If $n=2$, then $m\le 1$, so (1) gives
\[
c \ge 6n - 2m \ge 12 - 2 = 10.
\]
Since $G$ is simple, each vertex of $X$ is adjacent to at most two vertices of $H$, and therefore
\[
c \le 2|X|.
\]
Hence $|X| \ge 5$.

Thus we may assume from now on that $n \ge 3$.

Now consider the planar graph
\[
F := G[V(H) \cup X].
\]
It has $n+k$ vertices, and its edges consist of the $m$ edges of $H$, the $c$ edges between $H$ and $X$, and possibly some edges inside $X$. Therefore
\[
|E(F)| \ge m+c.
\]
Since $F$ is simple and planar,
\[
|E(F)| \le 3(n+k)-6,
\]
and hence
\[
m+c \le 3(n+k)-6. \tag{2}
\]

Suppose for contradiction that $k \le 3$. Then (2) yields
\[
m+c \le 3n+3. \tag{3}
\]
Combining (1) and (3), we obtain
\[
6n \le 2m+c \le m+(m+c) \le m+(3n+3),
\]
so
\[
m \ge 3n-3.
\]
But $H$ is itself a simple planar graph on at least three vertices, so
\[
m \le 3n-6,
\]
a contradiction.

Therefore $k \ge 4$, that is, $|X| \ge 4$.
\end{proof}

\begin{corollary}
\label{lem:avgdeg_to_four_neighbors}
Let $G$ be a planar graph, let $M$ be a maximum matching of $G$, let $x$ be an unmatched vertex, and consider $G_x$. 
Let $D_0$ be a connected component of $G[D(G_x)]$ with $n:=|V(D_0)|\ge 3$, and let
\[
t := |N_{G_x}(D_0)\cap A(G_x)|.
\]
If
\[
\sum_{v\in V(D_0)} \deg_{G_x}(v) \;\ge\; 6n,
\]
then $t \ge 4$.
\end{corollary}

\begin{proof}
Apply \Cref{lemma_6deg_neighbors} to the graph $G_x$ and the induced subgraph $H := G_x[D_0]$.

Since $G_x$ is planar and $D_0$ is an induced subgraph of $G_x$, the assumptions of \Cref{lemma_6deg_neighbors} are satisfied. The set
\[
X := N_{G_x}(V(H)) \setminus V(H)
\]
is exactly $N_{G_x}(D_0)\cap A(G_x)$, because by \Cref{lem:DGx_equals_Splus} all vertices outside $D(G_x)$ belong to $A(G_x)$.

Thus $|X| = t$, and the conclusion $t \ge 4$ follows.
\end{proof}

If we show that a factor-critical component has average degree at least $6$, or that it contains a minor with average degree at least $6$, then by \Cref{lemma_6deg_neighbors} this component has at least four distinct neighbors in $A(G_x)$. In that case, after contracting the component to a single vertex, we obtain a gadget of type \textbf{(I)} from \Cref{lemma_gadgets}.

The following observation will also be useful.

\begin{observation}
\label{obs:neighbors_of_x}
Let $G$ be a planar graph, let $M$ be a maximum matching of $G$, and let $x\in V(G)$ be unmatched. Let $u\in N(x)$, and let $v$ be the vertex matched with $u$ in $M$.

\begin{enumerate}
    \item If $v\in N(x)$, then both $u$ and $v$ belong to $S(x)$, and lie in the same connected component of $G[D(G_x)]$ as $x$.
    
    \item If $v\notin N(x)$, then $u$ is either in $S(x)$ (and hence in $D(G_x)$), or $u\in A(G)$ (and therefore $u\in A(G_x)$).
\end{enumerate}
\end{observation}

\begin{proof}
Let $u\in N(x)$, and let $v$ be its partner in $M$.

\smallskip
\noindent\emph{(1)} Suppose $v\in N(x)$. Then $uv\in M$, and the paths
\[
x,u,v
\quad\text{and}\quad
x,v,u
\]
are even $M$-alternating paths whose last edges lie in $M$. Hence both $u$ and $v$ belong to $S(x)$.

Since $u$ and $v$ are adjacent to $x$, all three vertices lie in $S(x)=D(G_x)$ and are connected in $G[D(G_x)]$. Thus $u$ and $v$ lie in the same connected component of $G[D(G_x)]$ as $x$.

\smallskip
\noindent\emph{(2)} Suppose $v\notin N(x)$. If $u\in S(x)$, there is nothing to prove. Otherwise, $u\notin S(x)$.

Since $u\in V(G_x)$, \Cref{lem:V_Gx_cap_DG} implies that $u\notin D(G)$. Hence $u$ is either in $A(G)$ or $C(G)$. But $u$ has a neighbor $x\in D(G)$, and there are no edges between $C(G)$ and $D(G)$. Therefore $u\in A(G)$.

Finally, since $u\in N(S(x))\setminus S(x)$, Lemma~\ref{lem:DGx_equals_Splus} gives $u\in A(G_x)$.
\end{proof}

We will also use the following simple reduction.

\begin{observation}
\label{obs:handles}
Let $G$ be a graph and let $a_1,a_2 \in V(G)$ be vertices of degree $2$ such that
\[
N(a_1)=\{v,a_2\}
\quad\text{and}\quad
N(a_2)=\{v,a_1\}.
\]
that is, $a_1$ and $a_2$ are adjacent and share the same second neighbor $v$. Suppose moreover that both $a_1$ and $a_2$ are inessential.

Let $G' := G - \{a_1,a_2\}$. If $G'$ admits a conflict-free coloring with at most $4$ colors, then so does $G$.
\end{observation}

\begin{proof}
Let $c'$ be a conflict-free coloring of $G'$ using at most $4$ colors. We extend $c'$ to a coloring $c$ of $G$.

The vertices $a_1$ and $a_2$ are adjacent only to each other and to $v$. Thus their neighborhoods are identical, and consist of $\{v,a_2\}$ and $\{v,a_1\}$, respectively.

Consider the colors appearing in $N(v)$ under $c'$. Let $\alpha$ be a color that appears exactly once in $N(v)$ (which exists since $c'$ is conflict-free). The vertices $a_1$ and $a_2$ must avoid two colors:
\begin{itemize}
    \item the color $c'(v)$, to preserve properness,
    \item the unique color $\alpha$, so as not to destroy the conflict-free witness of $v$.
\end{itemize}
Thus at most two colors are forbidden for $a_1$ and $a_2$, and since we use at most four colors, there exist two distinct available colors for $a_1$ and $a_2$. Assign these two colors to $a_1$ and $a_2$.

It remains to check the conflict-free condition. The vertices $a_1$ and $a_2$ each see two neighbors with different colors, so one of them appears exactly once in their neighborhood. The vertex $v$ retains its original unique color witness, since neither $a_1$ nor $a_2$ was assigned that color. All other vertices are unaffected.

Thus $c$ is a conflict-free coloring of $G$ using at most $4$ colors.
\end{proof}

We refer to such a configuration as described above as a \emph{handle}. 
By \Cref{obs:handles}, handles can be removed without affecting the existence of a $4$-conflict-free coloring: we may iteratively delete all handles from the graph, solve the problem on the resulting graph, and then reinsert the handles and extend the coloring as described above.

Therefore, from this point on, we may assume without loss of generality that the graph contains no handles.

The following lemma is motivated by our treatment of the two remaining difficult cases. The local analysis will show the following. Whenever the current component contains a vertex of degree $2$ or $4$ which is not immediately resolved by a special lift, and cannot be reduced to a contractible case by switching the matching, then either one of the desired gadgets has already appeared, or the closed neighborhood of this vertex can be contracted safely. Here safely means that the contraction preserves planarity, does not decrease the degree of any vertex represented by the component (including the ones being contracted), and creates a new vertex of degree at least $6$.

Once this local statement is available, the rest of the argument is iterative. We repeatedly contract such safe closed neighborhoods. Since each contraction removes at least one low-degree vertex and creates only vertices of degree at least $6$, the process eventually produces a connected component of minimum degree at least $6$. The planar density lemma then forces this component to have at least four distinct neighbors in $A$, giving the desired contracted gadget.

\begin{lemma}
\label{lem:reduce_low_degree_component}
Let $G$ be a plane graph with vertex partition $V(G)=A\cup B$, where
$G[B]$ is connected. Suppose that every vertex of $B$ has degree $2$, $4$,
or at least $6$, and that every vertex $v\in B$ of degree $2$ or $4$ satisfies $N_G(v)\subseteq B$.

Assume that $G$ is handle-free,

Assume moreover that every vertex $v\in B$ of degree $2$ or $4$ is reducible
in the following sense:
\begin{enumerate}
\item there is no path in $G[B]-v$ between two distinct vertices of
$N_G(v)$;
\item identifying $N_G[v]$ to a single vertex preserves planarity and does
not decrease the degree of any vertex of $B$, including the vertices being
contracted; that is, the resulting vertex has degree at least the maximum
degree of the vertices in $N_G[v]$.
\item The resulting vertex has degree at least $6$.
\end{enumerate}

Finally, assume the following local overlap condition. Whenever two adjacent
vertices $u,v\in B$ both have degree $2$ or $4$, they have an odd number of
common neighbors in $B$; equivalently, $|N_G(u)\cap N_G(v)\cap B|\in \{1,3\}$.

Then one of the following holds:
\begin{enumerate}
\item $B$ can be contracted to a single vertex adjacent to at least four
distinct vertices of $A$ (gadgets~\textbf{(I)}); or
\item $G$ contains a contraction to a copy of $K_{2,3}$ with bipartition
$(X,Y)$, where $X\subseteq B$, $|X|=2$, and $Y\subseteq A$, $|Y|=3$. (gadgets~\textbf{(II)})
\end{enumerate}
\end{lemma}

\begin{proof}
We contract configurations centered at degree-$2$ and degree-$4$ vertices of $B$. Whenever a set of vertices of $B$ is contracted, the resulting vertex is again regarded as a vertex of $B$. Thus vertices of $A$ are never contracted.

We proceed inductively. Suppose that after $k$ contractions no desired gadget has been found, and assume that no contraction performed so far has decreased the degree of any vertex in the current image of $G[B]$. We prove that either a desired gadget appears, or there is another contraction centered at a degree-$2$ or degree-$4$ vertex which again does not decrease the degree of any vertex.

For $k=0$, the assumption is vacuous, and the existence of the first contraction follows from the reducibility assumption.

Now assume that $k\ge 0$ and let $C$ be the current image of $G[B]$. By the induction assumption, all previous contractions preserved the degrees of vertices of $C$.

We also note that every vertex of $C$ of degree $2$ or $4$ has all its neighbors in $C$. Indeed, every vertex created by a previous contraction has degree at least $6$, since by the induction assumption no degree has decreased. Hence a current vertex of degree $2$ or $4$ was not created by a contraction. It is therefore an original vertex of $B$ of degree $2$ or $4$. By assumption, all neighbors of such a vertex originally belonged to $B$, and since all contractions are performed inside $B$, their current images all belong to $C$.

Let $v\in C$ be the next vertex of degree $2$ or $4$ to be reduced, and put
$$
N:=N_C(v).
$$
By the observation above, all vertices of $N$ lie in the current image of $B$. We consider the contraction of
$$
S:={v}\cup N.
$$

Suppose first that, in the current graph, two vertices of $N$ are adjacent although this adjacency was not present in the previous contractions (and original graph). Then this new adjacency lifts (undoing the previous contractions), in the original graph, to a path between two distinct vertices of $N_G(v)$ whose internal vertices avoid $N_G[v]$. This contradicts the reducibility assumption for $v$.

Similarly, suppose that some vertex
$$
w\in C\setminus S
$$
is adjacent in the current graph to two distinct vertices of $N$. Lifting the two corresponding adjacencies through the previous contractions gives a path in the original graph between two distinct vertices of $N_G(v)$ whose internal vertices avoid $N_G[v]$, again contradicting reducibility.

Thus no obstruction inside the current image of $B$ can make the contraction of $S$ lose degree: no new edge has appeared inside $N$, and no vertex of $C\setminus S$ sees two vertices of $N$.

Therefore, the only possible reason why contracting $S$ could fail to produce a vertex of degree at least $6$ without decreasing the degree of any vertex is that some vertices of $A$ are adjacent to several vertices of $N$. By the observation above, such vertices of $A$ can be adjacent only to vertices of $N$ whose degree is at least $6$.

Let
$$
H:={u\in N : \deg_C(u)\ge 6}.
$$
If the contracted image of $S$ has at least four distinct neighbors in $A$, then contracting $S$ gives a single vertex adjacent to at least four distinct vertices of $A$, which is the desired gadget of type~\textbf{(I)}.

We may therefore assume that the contracted image of $S$ has at most three distinct neighbors in $A$. If two vertices $p,q\in H$ have three common neighbors in $A$, say
$$
Y={y_1,y_2,y_3}\subseteq A,
$$
then $p$ and $q$, together with $Y$, give a $K_{2,3}$ contraction: we contract the branch set represented by $p$ to one vertex and the branch set represented by $q$ to another vertex. Thus we obtain the desired gadget of type~\textbf{(II)}.

It remains to consider the case where no two vertices of $H$ have three common neighbors in $A$. Equivalently, every pair of vertices of $H$ has at most two common neighbors in $A$. We show that in this case the contraction of $S$ is still safe.

First suppose that $\deg_C(v)=2$, and write
$$
N={p,q}.
$$
The only possible loss in the contraction of $S$ can come from common neighbors of $p$ and $q$ in $A$. By the assumption above, $p$ and $q$ have at most two such common neighbors. Hence, when $S$ is contracted, each of $p$ and $q$ contributes at least two neighbors outside $S$ which were not already neighbors of the other one. Thus the resulting vertex has degree at least $6$, and the degree of no vertex is decreased.

It remains to consider the case where the chosen vertex has degree $4$. Let $w$ be this vertex. We distinguish cases according to the number of neighbors of $w$ which have degree at least $6$.

If $w$ has at most one neighbor of degree at least $6$, then the contraction is safe. Indeed, only vertices of degree at least $6$ can have neighbors in $A$, and we have already ruled out any loss through vertices of $B$.

Suppose next that $w$ has exactly two neighbors of degree at least $6$. If one of the remaining neighbors of $w$ has degree $2$, then we contract the closed neighborhood of that degree-$2$ vertex instead, and the desired outcome follows from the degree-$2$ case above.

Hence we may assume that the two remaining neighbors of $w$ both have degree $4$. By the local overlap condition, each of them has either one or three neighbors in common with $w$. If both have only one common neighbor with $w$, then each of them contributes two neighbors outside $N_C[w]$. In this case, contracting $N_C[w]$ does not decrease the degree of any vertex.

It remains to consider the case where both of these degree-$4$ neighbors have three neighbors in common with $w$. Then the two neighbors of $w$ of degree at least $6$ necessarily lie on different faces; otherwise they could be joined by an edge in one face, giving us the $K_5$, which contradicts the planarity assumption. Hence each of them has at most three neighbors in $N_C[w]$, and therefore each has at least three neighbors outside $N_C[w]$. Since they lie on different faces, these outside neighbors cannot be shared. Consequently, contracting $N_C[w]$ creates a vertex of degree at least $6$ and does not decrease the degree of any vertex.

Suppose next that $w$ has exactly three neighbors of degree at least $6$. If the remaining neighbor has degree $2$, then we contract the closed neighborhood of that degree-$2$ vertex instead, and the desired outcome follows from the degree-$2$ case. Hence we may assume that the remaining neighbor has degree $4$; call it $b$.

First assume that $b$ has three neighbors in common with $w$. Then, as in the previous case, two of the degree-at-least-$6$ neighbors of $w$, say $a_1$ and $a_2$, must lie on different faces of the local configuration. In particular, they contribute distinct neighbors outside $N_C[w]$.

Let $a_3$ be the third neighbor of $w$ of degree at least $6$. The vertex $a_3$ can have at most four neighbors in $N_C[w]$. Moreover, since no two vertices of $H$ have three common neighbors in $A$, the vertex $a_3$ shares at most two neighbors in $A$ with each of $a_1$ and $a_2$.

Suppose that $a_3$ is adjacent to three vertices of $A$, all of which are also adjacent to $a_1$ or to $a_2$. Then we obtain the gadget of type~\textbf{(II)}: the vertices $a_1$ and $a_2$ can be contracted to one vertex through the path going via $w$, and together with $a_3$ and these three vertices of $A$ we obtain a contraction to $K_{2,3}$. Thus we may assume that this does not happen.

Consequently, after contracting $N_C[w]$, the possible loss in the degree of $a_3$ is compensated by the distinct neighbors contributed by $a_1$ and $a_2$. Hence the contraction of $N_C[w]$ does not decrease the degree of $a_3$. The vertices $a_1$ and $a_2$ also do not lose degree: each has at most three neighbors in $N_C[w]$, and, since they lie on different faces, they contribute distinct neighbors outside $N_C[w]$. Therefore contracting $N_C[w]$ is safe and the resulting vertex has degree at least $6$.

It remains to consider the case where $b$ has only one neighbor in common with $w$. Then $b$ contributes two neighbors outside $N_C[w]$. Each of the three degree-at-least-$6$ neighbors of $w$ has at most three neighbors in $N_C[w]$, and hence has at least three neighbors outside $N_C[w]$.

These outside neighbors cannot all be the same. Indeed, if they were all the same, then these common neighbors would have to lie in $A$, and we would obtain the gadget of type~\textbf{(II)}. Hence, after contracting $N_C[w]$, each degree-at-least-$6$ neighbor loses at most three neighbors inside $N_C[w]$, but gains the two neighbors contributed by $b$ and at least one further distinct neighbor contributed by one of the other degree-at-least-$6$ neighbors. Since any two of them share at most two neighbors in $A$, no degree is decreased. Thus the contraction of $N_C[w]$ is safe and the resulting vertex has degree at least $6$.

Finally, suppose that all neighbors of $w$ have degree at least $6$. If some of these neighbors have four neighbors in $N_C[w]$, then two of the degree-at-least-$6$ neighbors necessarily lie on different faces of the local configuration, and the same argument as above applies.

Thus we may assume that every neighbor of $w$ has at most three neighbors in $N_C[w]$. Then, when $N_C[w]$ is contracted, the degree of any neighbor of $w$ can decrease by at most $3$. On the other hand, since any two vertices of $H$ have at most two common neighbors in $A$, the other three neighbors of $w$ each contribute at least one neighbor which was not already adjacent to the given vertex. Hence no degree is decreased, and the contracted vertex has degree at least $6$.

We have shown that, unless one of the desired gadgets already appears, no loss can occur through vertices of $B$ or through vertices of $A$. Therefore the contraction of $S$ preserves the degree of every vertex in the current component and produces a vertex of degree at least $6$. This proves the induction step.

If a gadget appears during the process, we are done. Otherwise, after all degree-$2$ and degree-$4$ vertices have been contracted, the resulting component has minimum degree at least $6$: no contraction decreased any degree, and every created vertex has degree at least $6$.

By \Cref{lemma_6deg_neighbors}, this final component has at least four distinct neighbors in $A$. Contracting the whole final component to a single vertex gives a vertex adjacent to at least four vertices of $A$.

Thus, in every case, one of the two desired outcomes holds.
\end{proof}

\end{toappendix}

\section{Difficult cases: 2B and 4C}\label{sec:difficult2}

Having finished the preparatory work in \Cref{lem:controlled} and \Cref{lemma_gadgets}, we turn to a local in-depth analysis of each of the difficult cases; Case~2B and Case~4C.
The rather lengthy and technical details will be left entirely for the Appendix, and here we only briefly illustrate the proof techniques used there on the (in our view more interesting) Case~4C.

\subparagraph{Case 4C sketched.}\apxmark{apx:difficult4C}
Consider a maximum matching~$M$ and a vertex $x\in U$ as in \Cref{lem:controlled}, such that $x$ is Case~4C.
Referring to \Cref{fig:neighbors_four}, we assume that $x$ has neighbors $a_1,a_2,a_3,a_4$ in this cyclic order, such that $a_1a_2\in M$ and $a_3a_4\in M$ (observe that $a_1a_3, a_2a_4\in M$ would violate planarity).
The rough ideas of our solution are as follows.

Let $D_0$ be the connected component of $G[D(G_x)]$ containing $x$.
For start we show that if any two consecutive vertices of $a_1,a_2,a_3,a_4$ are connected by a path $P\subseteq D_0$ (which is not a single edge and which avoids $x$ and $a_1,a_2,a_3,a_4$ except at the ends), 
then we can resolve this case locally, or we contradict planarity of $G$ or the conclusions of \Cref{lem:controlled}:
\begin{itemize}
\item Since $P\subseteq D_0$, every internal vertex of $P$ is reachable from $x$ by some $M$-alternating path inside $D_0$ via one of $a_1,a_2,a_3,a_4$, and so using an edge $a_1a_2$ or $a_3a_4$.
	With help of an easy case analysis, one can use this fact to construct an $M$-alternating path $P'\subseteq D_0$ between two consecutive vertices of $a_1,a_2,a_3,a_4$ (possibly not the same two as the ends~of~$P$).
\item If the two ends of $P'$ are $a_1,a_2$ up to symmetry, see \Cref{fig:4C_path_matched}, then we can switch $M$ by the $M$-alternating cycle $P'+a_1a_2$, obtaining contractible Case~4B.
	This is in contrary to \Cref{lem:controlled}\,c) (in other words, we would directly resolve Case~4B instead).
\item If, up to symmetry, the two ends of $P'$ are $a_2,a_3$ and $a_1a_4\in E(G)$, then the same solution is obtained
	by switching the $M$-alternating cycle $P'+a_1a_2+a_1a_4+a_3a_4$.
\item If the two ends of $P'$ are $a_2,a_3$ (but $a_1a_4\not\in E(G)$), see \Cref{fig:4C_path_non_matched}, then we choose a set $F\subseteq E(G)$ from switched $M\triangle E(P')$ by adding the two edges $xa_1$ and $xa_4$,
	and define $H'$ as the contracted graph $H(G,F)$ with added edge $a_2a_3$ (which preserves planarity since $a_2$,$a_3$ are consecutive around~$x$).
	As a result, any proper $4$-coloring of $H'$ assigns the supervertex $[a_1xa_4]$ one color, say $1$, and two other colors, say $2$ and $3$, to the supervertices of $a_2$ and $a_3$.
	Then, in the lifted coloring of $G$, $x$ sees a unique color $2$ in $a_2$, and both $a_1$ and $a_4$ see a unique color $1$ in~$x$ (since $a_1a_4\not\in E(G)$).
	The rest follows from \Cref{lem:controlled}.
\end{itemize}

Otherwise, we examine possible vertex degrees in~$D_0$.
Since each vertex of $V(D_0)\subseteq S(x)$ (which in this case include $a_1,a_2,a_3,a_4$) is switchable from $x$, 
\Cref{lem:controlled} implies that every vertex of $D_0$ has degree $2$, $4$ or at least~$6$, or it is in Case~5B with $4$ disjoint paths to $A(G)$.
We then apply an inductive edge-density argument within $D_0$ which shows that there is a path $P''\subseteq D_0$ between some two of $a_1,a_2,a_3,a_4$,
or it eventually yields a gadget of \Cref{lemma_gadgets},

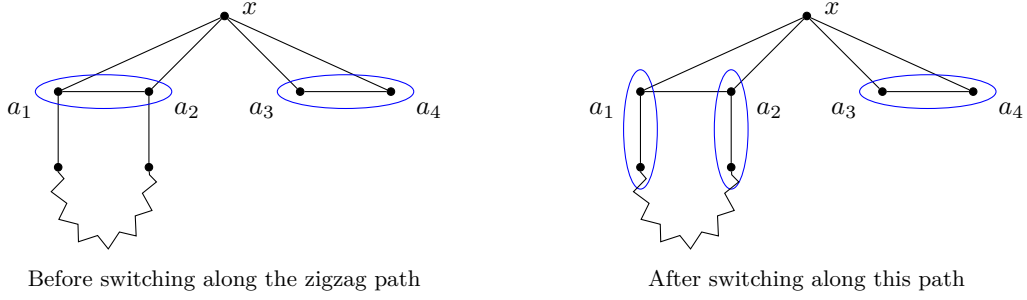
\begin{figure}[t]
        \centering
        \begin{subfigure}[b]{0.45\textwidth}
            \centering
        \begin{tikzpicture}[every node/.style={circle, draw, fill=black, inner sep=1pt}, node distance=1cm]
            \node (t1) at (0,0) {};
            \node (l1) at ($(t1) + (-2.2,-1)$) {};
            \node (r1) at ($(t1) + (2.2,-1)$) {};
            \node (rm1) at ($(t1) + (1.0,-1)$) {};
            \node (lm1) at ($(t1) + (-1.0,-1)$) {};
            \node (l2) at ($(l1) + (-0,-1)$) {};
            \node (lm2) at ($(lm1) + (0,-1)$) {};
            \draw (t1) -- (l1);
            \draw (t1) -- (r1);
            \draw (t1) -- (lm1);
            \draw (t1) -- (rm1);
            \draw (l2) -- (l1) -- (lm1) -- (lm2);
            \draw (rm1) -- (r1);

            \node[draw=blue, ellipse, fit=(l1) (lm1), inner sep=4pt, fill=none, transform shape, scale=0.80] {};
            \node[draw=blue, ellipse, fit=(r1) (rm1), inner sep=4pt, fill=none, transform shape, scale=0.80] {};

            \node[draw=none, fill=none, anchor=west] at ($(t1) + (0.15, 0.1)$) {$x$};
            \node[draw=none, fill=none, anchor=east] at ($(l1) + (-0.25, -0.25)$) {$a_1$};
            \node[draw=none, fill=none, anchor=west] at ($(r1) + (0.25, -0.25)$) {$a_4$};
            \node[draw=none, fill=none, anchor=west] at ($(lm1) + (0.25, -0.25)$) {$a_2$};
            \node[draw=none, fill=none, anchor=east] at ($(rm1) + (-0.25, -0.25)$) {$a_3$};

            \draw[decorate, decoration={zigzag}]
            (l2) .. controls (-2.3,-3.3) and (-0.9,-3.3) .. (lm2);

            \node[draw=none, fill=none, shape=rectangle] at (0,-3.5) {\footnotesize Before switching along the zigzag path};
        \end{tikzpicture}
        \end{subfigure}
        \hfill
        \begin{subfigure}[b]{0.45\textwidth}
            \centering
        \begin{tikzpicture}[every node/.style={circle, draw, fill=black, inner sep=1pt}, node distance=1cm]
            \node (t1) at (0,0) {};
            \node (l1) at ($(t1) + (-2.2,-1)$) {};
            \node (r1) at ($(t1) + (2.2,-1)$) {};
            \node (rm1) at ($(t1) + (1.0,-1)$) {};
            \node (lm1) at ($(t1) + (-1.0,-1)$) {};
            \node (l2) at ($(l1) + (-0,-1)$) {};
            \node (lm2) at ($(lm1) + (0,-1)$) {};
            \draw (t1) -- (l1);
            \draw (t1) -- (r1);
            \draw (t1) -- (lm1);
            \draw (t1) -- (rm1);
            \draw (l1) -- (lm1);
            \draw (rm1) -- (r1);
            \draw (l1) -- (l2);
            \draw (lm1) -- (lm2);

            \node[draw=blue, ellipse, fit=(l1) (l2), inner sep=4pt, fill=none, transform shape, scale=0.80] {};
            \node[draw=blue, ellipse, fit=(lm1) (lm2), inner sep=4pt, fill=none, transform shape, scale=0.80] {};
            \node[draw=blue, ellipse, fit=(r1) (rm1), inner sep=4pt, fill=none, transform shape, scale=0.80] {};

            \node[draw=none, fill=none, anchor=west] at ($(t1) + (0.15, 0.1)$) {$x$};
            \node[draw=none, fill=none, anchor=east] at ($(l1) + (-0.25, -0.25)$) {$a_1$};
            \node[draw=none, fill=none, anchor=west] at ($(r1) + (0.25, -0.25)$) {$a_4$};
            \node[draw=none, fill=none, anchor=west] at ($(lm1) + (0.25, -0.25)$) {$a_2$};
            \node[draw=none, fill=none, anchor=east] at ($(rm1) + (-0.25, -0.25)$) {$a_3$};
            \draw[decorate, decoration={zigzag}]
            (l2) .. controls (-2.3,-3.3) and (-0.9,-3.3) .. (lm2);

            \node[draw=none, fill=none, shape=rectangle] at (0,-3.5) {\footnotesize After switching along this path};
        \end{tikzpicture}
        \end{subfigure}
        \caption{An $M$-alternating path in Case 4C between two consecutive matched neighbors $a_1,a_2$.}
	\label{fig:4C_path_matched}
\end{figure}

\begin{figure}[tb]
        \centering
        \begin{subfigure}[b]{0.45\textwidth}
            \centering
        \begin{tikzpicture}[every node/.style={circle, draw, fill=black, inner sep=1pt}, node distance=1cm]
            \node (t1) at (0,0) {};
            \node (l1) at ($(t1) + (-2.2,-1)$) {};
            \node (r1) at ($(t1) + (2.2,-1)$) {};
            \node (rm1) at ($(t1) + (1.0,-1)$) {};
            \node (lm1) at ($(t1) + (-1.0,-1)$) {};
            \node (rm2) at ($(rm1) + (0,-1)$) {};
            \node (lm2) at ($(lm1) + (0,-1)$) {};
            \draw (t1) -- (l1);
            \draw (t1) -- (r1);
	    \draw[blue] ($(lm1)+(0.3,0)$)--($(rm1)-(0.3,0)$) ;
            \draw (t1) -- (lm1) -- (lm2);
            \draw (t1) -- (rm1) -- (rm2);
            \draw (l1) -- (lm1);
            \draw (rm1) -- (r1);

            \node[draw=blue, ellipse, fit=(l1) (lm1), inner sep=4pt, fill=none, transform shape, scale=0.80] {};
            \node[draw=blue, ellipse, fit=(r1) (rm1), inner sep=4pt, fill=none, transform shape, scale=0.80] {};

            \node[draw=none, fill=none, anchor=west] at ($(t1) + (0.15, 0.1)$) {$x$};
            \node[draw=none, fill=none, anchor=east] at ($(l1) + (-0.25, -0.25)$) {$a_1$};
            \node[draw=none, fill=none, anchor=west] at ($(r1) + (0.25, -0.25)$) {$a_4$};
            \node[draw=none, fill=none, anchor=west] at ($(lm1) + (0.25, -0.25)$) {$a_2$};
            \node[draw=none, fill=none, anchor=east] at ($(rm1) + (-0.25, -0.25)$) {$a_3$};
            \draw[decorate, decoration={zigzag,}]
            (lm2) .. controls (-0.9,-3.3) and (0.9,-3.3) .. (rm2);

            \node[draw=none, fill=none, shape=rectangle] at (0,-3.5) {\footnotesize Before switching along the zigzag path};
        \end{tikzpicture}
        \end{subfigure}
        \hfill
        \begin{subfigure}[b]{0.45\textwidth}
            \centering
        \begin{tikzpicture}[every node/.style={circle, draw, fill=black, inner sep=1pt}, node distance=1cm]
            \node (t1) at (0,0) {};
            \node (l1) at ($(t1) + (-2.2,-1)$) {};
            \node (r1) at ($(t1) + (2.2,-1)$) {};
            \node (rm1) at ($(t1) + (1.0,-1)$) {};
            \node (lm1) at ($(t1) + (-1.0,-1)$) {};
            \node (rm2) at ($(rm1) + (0,-1)$) {};
            \node (lm2) at ($(lm1) + (0,-1)$) {};
            \draw (t1) -- (l1);
            \draw (t1) -- (r1);
            \draw (lm1) -- (lm2);
            \draw (rm1) -- (rm2);
	    \draw[blue] ($(lm1)+(0.15,0)$)--($(rm1)-(0.15,0)$) ;
            \draw (t1) -- (lm1);
            \draw (t1) -- (rm1);
            \draw (l1) -- (lm1);
            \draw (rm1) -- (r1);

            \node[draw=blue, ellipse, fit=(lm1) (lm2), inner sep=4pt, fill=none, transform shape, scale=0.80] {};
            \node[draw=blue, ellipse, fit=(rm1) (rm2), inner sep=4pt, fill=none, transform shape, scale=0.80] {};

            \node[draw=none, fill=none, anchor=west] at ($(t1) + (0.15, 0.1)$) {$x$};
            \node[draw=none, fill=none, anchor=east] at ($(l1) + (-0.25, -0.25)$) {$a_1$};
            \node[draw=none, fill=none, anchor=west] at ($(r1) + (0.25, -0.25)$) {$a_4$};
            \node[draw=none, fill=none, anchor=west] at ($(lm1) + (0.25, -0.25)$) {$a_2$};
            \node[draw=none, fill=none, anchor=east] at ($(rm1) + (-0.25, -0.25)$) {$a_3$};
            \draw[decorate, decoration={zigzag,}]
            (lm2) .. controls (-0.9,-3.3) and (0.9,-3.3) .. (rm2);

            \node[draw=none, fill=none, shape=rectangle] at (0,-3.5) {\footnotesize After switching along this path};
        \end{tikzpicture}
        \end{subfigure}
        \caption{An $M$-alternating path $P'$ in Case 4C between two consecutive non-matched neighbors~$a_2,a_3$.
		We switch by $P'$ and contract all three vertices $a_1,x,a_4$ into one before coloring.}
	\label{fig:4C_path_non_matched}
\end{figure}
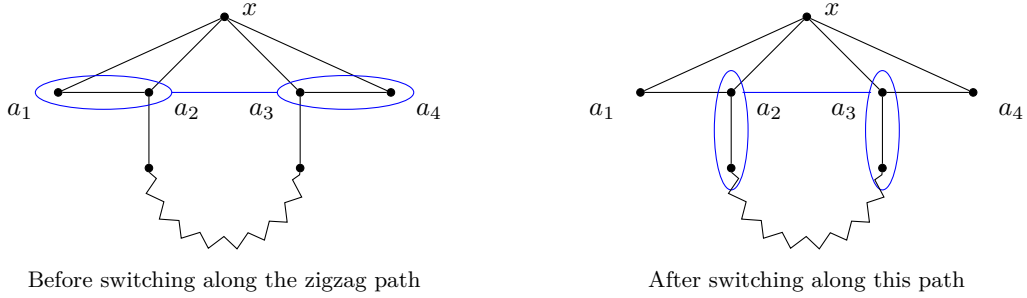

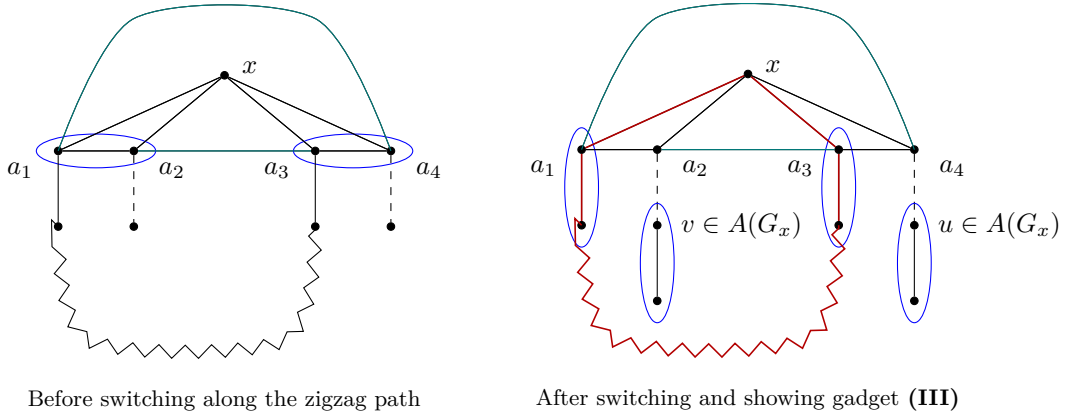
\begin{figure}[ht]
        \begin{subfigure}[b]{0.45\textwidth}
            \centering
        \begin{tikzpicture}[every node/.style={circle, draw, fill=black, inner sep=1pt}, node distance=1cm]
            \node (t1) at (0,0) {};
            \node (l1) at ($(t1) + (-2.2,-1)$) {};
            \node (r1) at ($(t1) + (2.2,-1)$) {};
            \node (rm1) at ($(t1) + (1.2,-1)$) {};
            \node (lm1) at ($(t1) + (-1.2,-1)$) {};
            \node (lm2) at ($(lm1) + (0,-1)$) {};
            \node (r2) at ($(r1) + (0,-1)$) {};
            \node (l2) at ($(l1) + (0,-1)$) {};
            \node (rm2) at ($(rm1) + (0,-1)$) {};
            \draw (t1) -- (l1);
            \draw (t1) -- (r1);
            \draw (rm1) -- (lm1);
            \draw (t1) -- (lm1);
            \draw (t1) -- (rm1);
            \draw (l1) -- (lm1);
            \draw (rm1) -- (r1);
            \draw (rm1) -- (rm1);
            \draw[dashed] (lm1) -- (lm2);
            \draw[dashed] (r1) -- (r2);

            \node[draw=blue, ellipse, fit=(l1) (lm1), inner sep=4pt, fill=none, transform shape, scale=0.80] {};
            \node[draw=blue, ellipse, fit=(r1) (rm1), inner sep=4pt, fill=none, transform shape, scale=0.80] {};

            \node[draw=none, fill=none, anchor=west] at ($(t1) + (0.15, 0.1)$) {$x$};
            \node[draw=none, fill=none, anchor=east] at ($(l1) + (-0.25, -0.25)$) {$a_1$};
            \node[draw=none, fill=none, anchor=west] at ($(r1) + (0.25, -0.25)$) {$a_4$};
            \node[draw=none, fill=none, anchor=west] at ($(lm1) + (0.25, -0.25)$) {$a_2$};
            \node[draw=none, fill=none, anchor=east] at ($(rm1) + (-0.25, -0.25)$) {$a_3$};

            \draw
              plot[smooth] coordinates {
                ($(l1)+(0.00,0.00)$)
                (-1.2, 0.75)
                (1.2,0.75)
                ($(r1)+(0.00,0.00)$)
              };

            \draw[black] (t1) -- (l1);
            \draw[black] (t1) -- (rm1);
            
            \draw[black, decorate,decoration=zigzag]
              plot[smooth] coordinates {
                ($(l2)+(0.00,0.00)$)
                (-1.8, -3.5)
                (0.8,-3.5)
                ($(rm2)+(0.00,0.00)$)
              };
            \draw[black] (rm1) -- (rm2) (l1) -- (l2);

            \draw[black] (t1) -- (r1) (t1) -- (lm1);
            \draw[teal] (rm1) -- (lm1);
            \draw[black] (lm1) -- (l1) (r1) -- (rm1);
            
            \draw[teal]
              plot[smooth] coordinates {
                ($(l1)+(0.00,0.00)$)
                (-1.2, 0.75)
                (1.2,0.75)
                ($(r1)+(0.00,0.00)$)
              };

            \node[draw=none, fill=none, shape=rectangle] at (0,-4.3) {\footnotesize Before switching along the zigzag path};
        \end{tikzpicture}
        \end{subfigure}
        \qquad
        \begin{subfigure}[b]{0.45\textwidth}
            \centering
        \begin{tikzpicture}[every node/.style={circle, draw, fill=black, inner sep=1pt}, node distance=1cm]
            \node (t1) at (0,0) {};
            \node (l1) at ($(t1) + (-2.2,-1)$) {};
            \node (l2) at ($(l1) + (0,-1)$) {};
            \node (r1) at ($(t1) + (2.2,-1)$) {};
            \node (rm1) at ($(t1) + (1.2,-1)$) {};
            \node (rm2) at ($(rm1) + (0,-1)$) {};
            \node (lm1) at ($(t1) + (-1.2,-1)$) {};
            \node (lm2) at ($(lm1) + (0,-1)$) {};
            \node (r2) at ($(r1) + (0,-1)$) {};
            \node (lm3) at ($(lm2) + (0,-1)$) {};
            \node (r3) at ($(r2) + (0,-1)$) {};
            \draw (t1) -- (l1);
            \draw (t1) -- (r1);
            \draw (rm1) -- (lm1);
            \draw (t1) -- (lm1);
            \draw (t1) -- (rm1);
            \draw (l1) -- (lm1);
            \draw (rm1) -- (r1);
            \draw (rm1) -- (rm2);
            \draw (l1) -- (l2);
            \draw (rm1) -- (rm1);
            \draw (lm2) -- (lm3);
            \draw (r2) -- (r3);
            \draw[dashed] (lm1) -- (lm2);
            \draw[dashed] (r1) -- (r2);

            \node[draw=blue, ellipse, fit=(l1) (l2), inner sep=4pt, fill=none, transform shape, scale=0.80] {};
            \node[draw=blue, ellipse, fit=(rm1) (rm2), inner sep=4pt, fill=none, transform shape, scale=0.80] {};
            \node[draw=blue, ellipse, fit=(lm2) (lm3), inner sep=4pt, fill=none, transform shape, scale=0.80] {};
            \node[draw=blue, ellipse, fit=(r2) (r3), inner sep=4pt, fill=none, transform shape, scale=0.80] {};

            \node[draw=none, fill=none, anchor=west] at ($(t1) + (0.15, 0.1)$) {$x$};
            \node[draw=none, fill=none, anchor=east] at ($(l1) + (-0.25, -0.25)$) {$a_1$};
            \node[draw=none, fill=none, anchor=west] at ($(r1) + (0.25, -0.25)$) {$a_4$};
            \node[draw=none, fill=none, anchor=west] at ($(lm1) + (0.25, -0.25)$) {$a_2$};
            \node[draw=none, fill=none, anchor=east] at ($(rm1) + (-0.25, -0.25)$) {$a_3$};
            \node[draw=none, fill=none, anchor=west] at ($(lm2) + (0.25, 0)$) {$v \in A(G_x)$};
            \node[draw=none, fill=none, anchor=west] at ($(r2) + (0.25, 0)$) {$u \in A(G_x)$};

            \draw
              plot[smooth] coordinates {
                ($(l1)+(0.00,0.00)$)
                (-1.2, 0.75)
                (1.2,0.75)
                ($(r1)+(0.00,0.00)$)
              };

            \draw[red!70!black, semithick] (t1) -- (l1);
            \draw[red!70!black, semithick] (t1) -- (rm1);
            \draw[red!70!black, semithick] (l1) -- (l2);
            \draw[red!70!black, semithick] (rm1) -- (rm2);
            
            \draw[red!70!black, semithick, decorate,decoration=zigzag]
              plot[smooth] coordinates {
                ($(l2)+(0.00,0.00)$)
                (-1.8, -3.5)
                (0.8,-3.5)
                ($(rm2)+(0.00,0.00)$)
              };

            \draw[black] (t1) -- (r1) (t1) -- (lm1);
            \draw[teal] (rm1) -- (lm1);
            \draw[black] (lm1) -- (l1) (r1) -- (rm1);
            
            \draw[teal]
              plot[smooth] coordinates {
                ($(l1)+(0.00,0.00)$)
                (-1.2, 0.75)
                (1.2,0.75)
                ($(r1)+(0.00,0.00)$)
              };
            
            \node[draw=none, fill=none, shape=rectangle] at (0,-4.3) {\footnotesize After switching and showing gadget~\textbf{(III)}};
        \end{tikzpicture}
        \end{subfigure}
        \caption{An $M$-alternating path between two non-consecutive neighbors $a_1,a_3$ in Case~4C, and the switched situation in which we attempt to create an edge between $a_1$ and $a_3$
        (along the zigzag path) or find a path connecting $a_2$ or $a_4$ to either $a_1$ or $a_3$ entirely within the factor-critical component, or we obtain gadget~\textbf{(III)} with $A = A(G_x)$ and the gadget triangle drawn in red.}
	\label{fig:xmatched_four_all}
\end{figure}

With respect to the starting argument, we may hence assume that $P''$ is an $M$-alternating path and the ends of $P''$ are not consecutive neighbors of~$x$. See \Cref{fig:xmatched_four_all}. 
Moreover, by a straightforward planarity argument, the vertices $a_1$ and $a_3$ can be joined by an edge not crossing the given drawing of $G$ (informally, drawn on either side along $P''$),
unless there exist, for $j=2,4$, a path $Q_j$ from $a_j$ to $V(P'')\setminus\{a_1,a_3\}$. Since there is no path between consecutive neighbors of $x$ in $D_0$, both paths $Q_2,Q_4$ intersect~$A(G)$.
Consequently, after wheel augmentation at $x$ in~$G$, the cycle $P''+x$ together with $Q_2$, $Q_4$ and the augmenting edges $a_1a_4$, $a_2a_3$ yield a gadget of type~\textbf{(III)}.

Otherwise, we can (analogously as above) choose a set $F$ from switched $M\triangle E(P'')$ and the two edges $xa_2$ and $xa_4$,
and define $H'$ as the contracted graph $H(G,F)$ with added edge $a_1a_3$ which now preserves planarity.
As a result, again, any proper $4$-coloring of $H'$ assigns the supervertex $[a_2xa_4]$ one color, say $1$, and two other colors, say $2$ and $3$, to the supervertices of $a_1$ and $a_3$.
So, in the lifted coloring of $G$, $x$ sees a unique color $2$ in $a_1$, and both $a_2$ and $a_4$ see a unique color $1$ in~$x$.
The rest again follows from \Cref{lem:controlled}.

\begin{toappendix}

\subsection{Case 2B}

The remaining difficult configurations require a different approach. We no longer try to resolve every local structure directly from the coloring of the wheel augmented graph. Instead, we use the minimal choice of the set $U$: if a configuration can be handled by the adjusted lift, or if after switching inside the relevant factor-critical component it reduces to one of the already treated configurations, then the corresponding unmatched vertex should not belong to $U$.

Consequently, any difficult configuration that remains around a vertex of $U$ must have a rigid structure. For Case~2B this rigidity is expressed first by the following separation lemma: unless the configuration reduces to a contractible case, the two matched neighbors of the unmatched vertex cannot be joined inside the same factor-critical component by an additional path avoiding the unmatched vertex. This separation is what later allows us either to contract the configuration to a high-degree vertex or to find one of the forbidden gadgets.

\begin{lemma}
\label{lem:case2B_no_long_path}
Let $G$ be a planar graph, let $M$ be a maximum matching of $G$, and let $x \in V(G)$ be an unmatched vertex of degree $2$ with neighbors $a_1,a_2$, where $a_1a_2 \in M$ (i.e., $x$ represents Case~2B). Let $D_0$ be the connected component of $G[D(G_x)]$ containing $x$.

If there is a path in $G[D_0]$ from $a_1$ to $a_2$ of length at least $2$ that avoids $x$, then there is an $M$-alternating path from $a_1$ to $a_2$ of length at least $2$ that avoids $x$, hence $x$ can be reduced to a contractible case.
\end{lemma}

\begin{proof}
Suppose for contradiction that there exists a path
\[
P=(a_1=v_0,v_1,\dots,v_t=a_2),
\qquad t\ge 2,
\]
contained entirely in $D_0$ whose internal vertices avoid $N(x)\cup\{x\}$ and the configuration cannot be reduced to a contractible case.

Since $D_0$ is factor-critical and $M|_{D_0}$ leaves exactly $x$ unmatched, every vertex of $D_0$ is joined to $x$ by an even $M$-alternating path contained in $D_0$ (\Cref{lem:alternating_components}).

\medskip
\noindent
\emph{Step 1: existence of an $M$-alternating path between $a_1$ and $a_2$.}

For each vertex $v_i\in V(P)\setminus\{a_1,a_2\}$, fix a shortest even $M$-alternating path from $x$ to $v_i$. Since $a_1a_2\in M$, such a path starts with either
\[
x,a_1,a_2
\qquad\text{or}\qquad
x,a_2,a_1.
\]
Accordingly, we say that $v_i$ is reachable from $a_2$ or reachable from $a_1$, respectively.

Now traverse $P$ from $a_1$ toward $a_2$. If $v_1$ is reachable from $a_2$, then concatenating an even $M$-alternating path from $a_2$ to $v_1$ with the edge $a_1v_1$ yields an $M$-alternating path from $a_2$ to $a_1$ whose internal vertices avoid $N(x)\cup\{x\}$, and we are done. Note that since the path is even its last edge is in $M$, thus we can indeed add the edge $a_1v_1$ to the path and keep it alternating.

So assume that $v_1$ is reachable from $a_1$. Continue along $P$ until the first vertex $v_i$ that is reachable from $a_2$. Then $v_{i-1}$ is reachable from $a_1$.

Fix an even $M$-alternating path
\[
R:a_1\leadsto v_{i-1}
\]
and an even $M$-alternating path
\[
S:a_2\leadsto v_i,
\]
both with internal vertices avoiding $N(x)\cup\{x\}$.

We claim that these paths yield an $M$-alternating path between $a_1$ and $a_2$.

First suppose that $R$ and $S$ are internally disjoint. Then the edge $v_{i-1}v_i$ joins their endpoints. Since both $R$ and $S$ are even $M$-alternating paths starting with non-matching edges, they end with matching edges. 
Hence, if \(v_{i-1}v_i\notin M\), then
\[
R\cup v_{i-1}v_i\cup S
\]
is an $M$-alternating path from $a_1$ to $a_2$;
- if \(v_{i-1}v_i\in M\), then the alternating structure forces \(R\) and \(S\) to approach this edge from opposite directions, and the same union again contains an $M$-alternating path from $a_1$ to $a_2$.

Now suppose that $R$ and $S$ are not internally disjoint. Let $e$ be the first edge that belongs to both $R$ and $S$ when $R$ is traversed from $v_{i-1}$ toward $a_1$ and $S$ from $v_i$ toward $a_2$. Since $R$ and $S$ are alternating, they cannot first meet in a single vertex without also sharing an incident matching edge; thus such a first common edge exists.

If $R$ and $S$ traverse $e$ in opposite directions, then concatenating the initial segment of $R$, the edge $e$, and the initial segment of $S$ yields an $M$-alternating path from $a_1$ to $a_2$.

If, on the other hand, $R$ and $S$ traverse $e$ in the same direction, then from the moment $S$ reaches $e$ it may follow the remainder of $R$. Consequently, $v_{i-1}$ becomes reachable from $a_2$ by an even $M$-alternating path whose internal vertices avoid $N(x)\cup\{x\}$, contradicting the choice of $v_i$ as the first vertex along $P$ that is reachable from $a_2$.

Therefore, in all cases, there exists an $M$-alternating path from $a_1$ to $a_2$ whose internal vertices avoid $N(x)\cup\{x\}$.

Finally, if no vertex of $P\setminus\{a_1,a_2\}$ is reachable from $a_2$, then in particular $v_{t-1}$ is reachable from $a_1$, and concatenating the corresponding even $M$-alternating path from $a_1$ to $v_{t-1}$ with the edge $v_{t-1}a_2$ yields an $M$-alternating path from $a_1$ to $a_2$ whose internal vertices avoid $N(x)\cup\{x\}$.

Thus in every case there exists an $M$-alternating path between $a_1$ and $a_2$.

\medskip
\noindent
\emph{Step 2: reduction to a contractible case.}

Since $a_1a_2 \in M$, the $M$-alternating path between $a_1$ and $a_2$ starts and ends with edges not in $M$. Together with the edge $a_1a_2$, it forms an $M$-alternating cycle $C$ contained entirely in $D_0$.

Define
\[
M' := M \triangle C.
\]
This switching preserves maximality and is confined to $D_0$.

After the switch, the edge $a_1a_2$ is no longer in the matching, and both $a_1$ and $a_2$ are matched to vertices outside $\{a_1,a_2\}$. Hence the neighborhood of $x$ is no longer a matched pair, and $x$ is reduced to a contractible configuration (Case~2A).

\medskip

This contradicts the assumption that the configuration at $x$ cannot be reduced to a contractible case.
\end{proof}

The previous lemma shows that in Case~2B the only genuinely difficult situation is when the matched neighbors of $x$ are not connected inside the factor-critical component except through $N(x) \cup x$. Thus, in the proof of the full Case~2B lemma, we may assume that no such path exists and focus only on the remaining local degree configurations.

\begin{lemma}\label{lemma:case2B_full}
Let $G$ be a planar graph. There exists a maximum matching $M$ of $G$ and a subset $U\subseteq V(G)\setminus V(M)$ of unmatched vertices such that the following hold.
Let $W=V(G)\setminus(U\cup V(M))$ be the (independent) set of unmatched vertices complementary to~$U$.
With respect to $M$, let $G_W$ denote the union of the subgraphs $G_x$ for $x\in W$ and $W^* \subseteq W$ be the set of the vertices of~$W$ in a difficult configuration.
Let $E^*=E[N(W^*)]$ denote the set of edges induced by the neighborhood of $W^*$.
\begin{enumerate}[a)]
 \item Every vertex of $V(G)\setminus V(M)$ is of degree at most~$5$ in~$G$ (cf.~\Cref{cor:lowdeg_free}).
 \item There is a set of edges $F\subseteq E(G)$ such that $F\cap(E(G_W)\setminus E^*)=M\cap(E(G_W)\setminus E^*)$, 
let $H=H(G,F)$ be the corresponding contracted graph and $H'$ be the wheel augmented graph of $H$ at~$W$.
Then every proper $4$-coloring $c_1$ of $H'$ admits a $U$-protecting recoloring to a proper $4$-coloring $c$ of $H'$ such that,
denoting by $c'$ the $4$-coloring of $G$ into which the coloring $c$ lifts (cf.~\Cref{obs:perfect}),
every vertex $x\in W\cup N(W^*)$ has a neighbor $u$ in~$G$ whose color $c'(u)$ is unique in the open neighborhood~$N(x)$.
\item For every vertex $x\in U$ and every $z\in S(x)$ (including $z=x$) such that the degree of $z$ in~$G$ is at most~$5$, let $M'$ be any matching obtained from $M$ in one of the following ways:
\begin{itemize}\vspace*{-1ex}%
\item $M' := M \triangle E(P)$ for an even $M$-alternating $x$--$z$ path in~$G$ (switching from $x$ to~$z$),
\item $M' := M \triangle E(Q)$ for $x=z$ and an $M$-alternating cycle $Q$ which is contained in the same connected component of $G[D(G)]$ as~$x$ (switching a cycle ``nearby~$x$'').
\end{itemize}
Then, with respect to~$M'$;
\begin{itemize}\vspace*{-1ex}
        \item the configuration at $z$ is Case~2B, and if $D_z$ denotes the
        connected component of $G[D(G_z)]$ containing $z$, then $z$ and $D_z$ satisfy
        one of the following structural outcomes:
        \begin{enumerate}
            \item[(i)] $N[z]$ can be contracted to a single vertex of degree at
            least $6$ in such a way that the contraction does not decrease the
            degree of any vertex of the relevant factor-critical
            component (including the ones being contracted);

            \item[(ii)] $D_z$ yields a minor isomorphic to $K_{2,3}$ with
            bipartition $(A,B)$, where the two branch sets corresponding to
            $B$ are contained in $D_z$, and
            \[
                A\subseteq A(G_z), \qquad |A|=3 .
            \]
            This is gadget~\textbf{(II)};

            \item[(iii)] $D_z$ yields a minor obtained by contracting a
            connected subgraph of $D_z$ to a single vertex adjacent to four
            distinct vertices of $A(G_z)$. This is gadget~\textbf{(I)}.
        \end{enumerate}

        \item the configuration at $z$ is Case~4C;

        \item the configuration at $z$ is one of the two Kempe-contractible cases (Case~4A, Case~5B) and there exist $4$ pairwise disjoint paths between the neighbors of $z$ and the set~$A(G)$. This is gadget~\textbf{(I)}.
    \end{itemize}
\end{enumerate}
\end{lemma}

\begin{proof}
We begin as in the proof of \Cref{lem:controlled} by fixing a planar embedding of~$G$.
For a maximum matching $M'$ of $G$, we say that an unmatched vertex $v\in V(G)\setminus V(M')$ is \emph{bad} if the configuration at $v$ with respect to $M'$ is either one of the two difficult cases,
or one of the two Kempe-contractible cases with existing $4$ pairwise disjoint paths between $N(v)$ and~$A(G)$.
We pick $M$ as a maximum matching of $G$ such that all unmatched vertices are of degree at most~$5$ (\Cref{cor:lowdeg_free}) and that the set $U$ of bad vertices unmatched by $M$ is minimal over all such~$M$.
So, in particular, condition a) holds~true.

Let us further assume that $N_G(x)=\{a_1,a_2\}$, and that $x$ represents
Case~2B. Thus $a_1a_2\in M$. We will later distinguish cases according to the degrees of
$a_1$ and $a_2$. We will be verifying conditions b) and c) at the same time.

First observe that if there exists a vertex $z\in S(x)$ such that switching from $x$ to $z$ produces a configuration which is either contractible, Kempe-contractible without four disjoint paths to $A(G_z)$, or a Case~2B configuration not satisfying all the assumptions in item~c), then we may perform this switch and resolve the resulting unmatched vertex. In the first two cases this follows from \Cref{lem:controlled}, and in Case~2B it follows from the arguments below. Hence, if $x\in U$, this would contradict the minimality of $U$; otherwise, the switched matching satisfies condition~b).

It remains to justify that such a switch does not destroy the already resolved configurations of the other unmatched vertices. The local configurations of unmatched vertices are preserved by \Cref{lem:switch_preserves_other_local_configs}. Moreover, Kempe-contractible configurations which did not contain four disjoint paths to the corresponding set $A(G_u)$ retain this property after the switch. The same applies to Case~2B. Indeed, we show below that whenever Case~2B cannot be contracted to a single vertex of degree at least $6$ without decreasing the degree of some vertex in that component and its factor-critical component contains no gadget from condition~c) as a minor, then it is resolvable either by not including the matched edge $a_1,a_2$ in the set $F$ or by reducing it to a contractible case through switching inside its factor-critical component. These properties are preserved when switching from $x$ to $z$, because they depend only on the degrees of the neighboring vertices and on the existence of a path within the relevant factor-critical component between the two neighbors; both features remain unchanged after switching by \Cref{lem:Dx_disjoint_Gu}.

All switching operations used below are performed along alternating paths contained in the factor-critical component under consideration. Hence the matching is changed only inside this component. By \Cref{lem:Dx_disjoint_Gu}, this does not affect the auxiliary graph $G_u$ of any other unmatched vertex $u\neq x$. Consequently, all witnesses and local configurations already fixed outside the present component remain unchanged. In particular, vertices outside this component keep their status with respect to $U$, and the wheel-augmentation and Kempe-contractible-case arguments for them remain valid.

We take $F$ to coincide with $M$ on $G_x$, except in one case described below, in which we omit one edge from~$F$. Whenever we change the matching, we modify $F$ accordingly so that it again corresponds to the new matching.

We now begin the verification of conditions b) and c):

\medskip
\noindent\textbf{Case $I$.} $\deg_G(a_1)=\deg_G(a_2)=2$.

In this case, the vertices induce an isolated copy of $C_3$, which can be colored trivially. Hence condition b) holds.

\medskip
\noindent\textbf{Case $II$.} Exactly one of $a_1,a_2$ has degree $2$.

Then the subgraph induced by $x$, $a_1$, and $a_2$ contains a handle. This
contradicts our standing assumption, made after \Cref{obs:handles}, that all
handles have already been removed from the graph.

\medskip
\noindent\textbf{Case $III$.} One of $a_1,a_2$ has degree $3$.

Then, after switching to that vertex, its local neighborhood falls into one of
the contractible configurations. Hence the present Case~2B configuration reduces
to a previously treated case. If $x\notin U$, this means that the adjusted lift
can handle the configuration after the switch as prescribed in \Cref{lem:controlled}. If $x\in U$, replacing $M$ by the
switched maximum matching gives an admissible choice with a smaller set $U$,
contradicting the minimality of $|U|$.

\medskip
\noindent\textbf{Case $IV$.} $\deg_G(a_1)=\deg_G(a_2)=4$.

We claim that this case is handled by the fact that we remove $a_1,a_2$ from $F$. Since both
$a_1$ and $a_2$ have degree $4$, each of them must be adjacent to a
matched pair; otherwise, after switching to that vertex, we would obtain a
contractible configuration, contradicting the minimality of $|U|$.

Moreover, these two matched pairs must be distinct. Indeed, if both $a_1$ and
$a_2$ were adjacent to the same matched pair, then by
\Cref{obs:neighbors_of_x} those vertices would lie in the same factor-critical
component, forcing a reduction to a previously treated case. Again, this would
either be absorbed by the adjusted lift or would contradict the minimality of
$|U|$.

After lifting the proper coloring of the contracted graph $H$, the vertex $x$ sees a unique color on either $a_1$ or~$a_2$. We next consider $a_1$; the argument for $a_2$ is symmetric. The vertices $x$ and $a_2$ are forced to receive distinct colors in $H'$. Moreover, the two matched neighbors of $a_1$ receive the same color after the lift, since they correspond to a single vertex of~$H'$. Therefore, this common color can coincide with the color of at most one of $x$ and $a_2$, and hence either $x$ or $a_2$ has a color that is unique in the neighborhood of~$a_1$.

Notice that the matched neighbors of $a_1$ and $a_2$ belong to the factor-critical component containing $x$. They therefore remain represented by edges of $F$, since $F$ coincides with $M$ except for the single omitted edge described above. Furthermore, a factor-critical component contains at most one unmatched vertex. Consequently, these matching edges remain in $F$ even after other unmatched vertices in Case~2B are resolved by the same procedure.

The modification of $F$ is confined to the present factor-critical component and therefore does not interfere with any other already resolved unmatched vertex. Hence all such vertices keep their unique-colored witnesses. Moreover, all other matched vertices keep their matched partners as unique-colored neighbors. Note that the recolorings presented in \Cref{lem:controlled} are consistent with this argument and $x, a_1, a_2$ all keep their uniquely colored neighbors even after it. Hence we have shown that the condition b) holds and if $x \in U$ then we get a contradiction with minimality of $U$.

\medskip
\noindent\textbf{Case $V$.} One of $a_1,a_2$ has degree $5$.

Then, after switching to that vertex, we reduce either to a contractible case or
to Case~5B. If we reduce to a contractible case, then the configuration is already
handled after the switch: for $x\notin U$ the condition b) holds, while for $x\in U$ it contradicts the minimality of $|U|$.

It remains to consider the reduction to Case~5B. By
\Cref{lem:controlled}, this yields either a coloring modification so condition b) holds, or a gadget of type \textbf{(I)} with $A=A(G_x)$. In the former subcase, $x$ is handled and hence
cannot remain in the minimal set $U$. In the latter subcase, one of the
structural outcomes listed in the statement holds.

\medskip
\noindent\textbf{Case $VI$.} One of $a_1,a_2$ has degree $4$ and the other has
degree at least $6$.

Without loss of generality, let $\deg_G(a_1)\ge 6$ and $\deg_G(a_2)=4$.
Then $a_2$ must, after switching, represent Case~4C; otherwise we would reduce
to a contractible case or to Case~4A, both already treated, and we would again be
in a case where either condition b) holds or excluded by the minimality of $|U|$.
In particular, $a_2$ is adjacent to a matched pair distinct from the neighbors
of $a_1$.

By \Cref{lem:case2B_no_long_path}, if the present configuration is not
reducible to a contracted case, which would again either contradict the minimality of $|U|$ or mean $x$ is not in $U$ and can be resolved by changing this matching, then there is no path in the factor-critical
component $D_0$ between $a_1$ and $a_2$ whose internal vertices lie outside the
configuration. We will say that Case~2B configuration is internally separated,
and may be contracted to a single vertex.

The resulting contracted vertex has degree at least $6$: it inherits the
neighbors of $a_1$, and the contribution of $a_2$ is sufficient by the internal
separation assumption. Since the configuration is internally separated, this
contraction does not decrease the degree of any other vertex of $D_0$. Hence
structural outcome~\textbf{(i)} holds.

\medskip
\noindent\textbf{Case $VII$.} $\deg_G(a_1)\ge 6$ and $\deg_G(a_2)\ge 6$.

First suppose that $a_1$ and $a_2$ have three common neighbors in $A(G_x)$.
Then these two vertices together with those three common neighbors yield a
minor isomorphic to $K_{2,3}$, so the condition c) holds.

Hence we may assume that $a_1$ and $a_2$ have at most two common neighbors in
$A(G_x)$. As in the previous case, if the configuration is not reducible to a
contractible case, then by \Cref{lem:case2B_no_long_path} there is no path in
$D_0$ from $a_1$ to $a_2$ whose internal vertices lie outside the configuration.
Therefore the configuration is internally separated and may be contracted.

Because both $a_1$ and $a_2$ have degree at least $6$ and share at most two
external neighbors, the contracted vertex still has degree at least $6$.
Again, internal separation guarantees that no other degree decreases. The degrees of the vertices being contracted are not decreased either. Indeed, the two degree-at-least-$6$ vertices share at most two neighbors in $A$, and their remaining neighbors outside $N[x]$ are distinct. Thus, when $N[x]$ is contracted, each of them may lose at most two neighbors, but this loss is compensated by at least two neighbors contributed by the other degree-at-least-$6$ vertex.
Thus
structural outcome~\textbf{(i)} holds.

\medskip

This exhausts all possibilities for a Case~2B vertex. Furthermore, we may choose $F$ to coincide with the matching everywhere except in the neighborhoods of those unmatched vertices in Case~2B that require the modification described above. The resolutions of the contractible and Kempe-contractible cases therefore remain the same as in \Cref{lem:controlled}, since all our operations are consistent with the techniques used in its proof. Together with \Cref{lem:controlled}, this shows that every vertex in $U$ satisfies condition~c), while every Case~2B vertex that does not satisfy condition~c) can be resolved. Hence conditions~b) and~c) are both verified.
\end{proof}

\subsection{Case 4C}
\label{apx:difficult4C}

We begin by analyzing $M$-alternating paths between vertices in the neighborhood of an unmatched vertex $x$ representing Case~4C. In particular, we study how the existence of such paths, contained entirely within the factor-critical component of $x$, constrains the structure and enforces the existence of an $M$-alternating path between vertices of the neighborhood, which will later help us solve this case.

\begin{lemma}
\label{clm:adjacent_path_implies_adjacent_alternating}
Let $D_0$ be the connected component of $G[D(G_x)]$ containing $x$.
Suppose there exists a path
\[
P=(p_0,p_1,\dots,p_t)
\]
contained entirely in $D_0$, whose endpoints $p_0,p_t$ are consecutive vertices in
the cyclic order $(a_1,a_2,a_3,a_4)$ of the neighbors of $x$ and whose internal vertices avoid $N(x)\cup\{x\}$.

Then there exists an $M$-alternating path in $D_0$ joining two adjacent neighbors
of $x$ whose internal vertices avoid $N(x)\cup\{x\}$.
\end{lemma}

\begin{proof}
As we show first, without loss of generality, we may assume that
\[
a_1a_2,\ a_3a_4\in M
\]
and that
\[
p_0=a_1,\qquad p_t=a_2.
\]
The other possible matching pattern $a_2a_3, a_4a_1 \in M$ is solved analogously.

Note that the matching cannot consist of two pairs of non-consecutive vertices in the cyclic order, that is, we cannot have $a_1a_3$ and $a_2a_4$ in $M$.

Indeed, consider the cycle (triangle) formed by $x,a_2,a_4$. In a planar embedding, this cycle separates the plane into an interior and an exterior region. Since $a_1$ and $a_3$ are non-consecutive neighbors of $x$ in the cyclic order, one of them must lie in the interior of this cycle and the other in the exterior. Consequently, any edge between $a_1$ and $a_3$ would have to cross the cycle, contradicting planarity.

Therefore, up to symmetry, we may assume that the matching pairs correspond to consecutive vertices in the cyclic order, namely
\[
a_1a_2,\ a_3a_4 \in M.
\]

By \Cref{lem:alternating_components}, every vertex of $D_0$ is joined to $x$ by
an even $M$-alternating path contained in $D_0$. For a vertex
$v\in V(P)\setminus \{a_1,a_2\}$, fix such a path of minimum length. Since
\[
a_1a_2,\ a_3a_4\in M,
\]
the first two edges of this path follow one of
\[
x,a_1,a_2;\qquad x,a_2,a_1;\qquad x,a_3,a_4;\qquad x,a_4,a_3.
\]
Accordingly, we say that $v$ is reachable from $a_2$, from $a_1$,
from $a_4$, or from $a_3$, respectively.

Because we chose the path shortest, once a vertex is reached from some $a_i$, the
path does not use any further edge in $G[N(x)\cup\{x\}]$; otherwise one could
shorten the initial segment and obtain a shorter alternating path from $x$.

Now consider the vertex $p_{t-1}$. If $p_{t-1}$ is reachable from some
$a_i\neq a_2$, then concatenating an even $M$-alternating path from $a_i$ to
$p_{t-1}$ with the edge $p_{t-1}p_t=p_{t-1}a_2$ yields an $M$-alternating path
from $a_i$ to $a_2$ whose internal vertices avoid $N(x)\cup\{x\}$. Indeed, the
path from $a_i$ to $p_{t-1}$ has even length and starts with a non-matching edge,
hence it ends at $p_{t-1}$ with a matching edge, so the final edge
$p_{t-1}a_2\notin M$ extends it alternatively.

Thus the only case left is that $p_{t-1}$ is reachable from $a_2$. In that case,
traverse $P$ backwards from $p_{t-1}$ toward $p_1$, and let $p_{t-i}$ be the
first vertex that is reachable from some $a_j\neq a_2$. Then
$p_{t-i+1}$ is reachable from $a_2$.

Fix an even $M$-alternating path $R$ from $a_j$ to $p_{t-i}$ and an even
$M$-alternating path $S$ from $a_2$ to $p_{t-i+1}$, both with internal vertices
avoiding $N(x)\cup\{x\}$.

We claim that these paths yield an $M$-alternating path from $a_j$ to $a_2$.

First suppose that $R$ and $S$ are internally disjoint. Then the edge
$p_{t-i}p_{t-i+1}$ joins their endpoints. Since both $R$ and $S$ are even
$M$-alternating paths starting with a non-matching edge, they end with matching
edges. Consequently, the last edge of $R$ at $p_{t-i}$ and the last edge of $S$
at $p_{t-i+1}$ are both in $M$.

If $p_{t-i}p_{t-i+1}\notin M$, then concatenating $R$, the edge
$p_{t-i}p_{t-i+1}$, and $S$ yields an $M$-alternating path from $a_j$ to $a_2$.

If $p_{t-i}p_{t-i+1}\in M$, then the alternating structure forces $R$ and $S$ to
approach this edge from opposite directions (since each endpoint is already
incident with a matching edge in its respective path). Hence, again,
\[
R\cup p_{t-i}p_{t-i+1}\cup S
\]
contains an $M$-alternating path from $a_j$ to $a_2$.

Now suppose that $R$ and $S$ are not internally disjoint. Let $e$ be the first
edge that belongs to both $R$ and $S$ when $R$ is traversed from $p_{t-i}$ toward
$a_j$ and $S$ is traversed from $p_{t-i+1}$ toward $a_2$. Since $R$ and $S$ are
alternating, they cannot first meet in a single vertex without also sharing an
incident matching edge; thus such a first common edge exists.

If $R$ and $S$ traverse $e$ in opposite directions, then concatenating the
initial segment of $R$, the edge $e$, and the initial segment of $S$ yields an
$M$-alternating path from $a_j$ to $a_2$.

If, on the other hand, $R$ and $S$ traverse $e$ in the same direction, then from
the moment $R$ reaches $e$ it may follow the remaining part of $S$ instead.
Consequently, $p_{t-i+1}$ is reachable from $a_j$ by an even $M$-alternating path
whose internal vertices avoid $N(x)\cup\{x\}$, contradicting the choice of
$p_{t-i}$ as the first vertex encountered when moving backwards along $P$ that is
reachable from a vertex other than $a_2$.

Therefore, in all cases, there exists an $M$-alternating path from $a_2$ to some
$a_j\neq a_2$ whose internal vertices avoid $N(x)\cup\{x\}$.

If $a_j\in\{a_1,a_3\}$, then the endpoints are consecutive and we are done. Hence
the only unresolved case is $a_j=a_4$, that is, there exists an $M$-alternating
path from $a_2$ to $a_4$ whose internal vertices avoid $N(x)\cup\{x\}$.

If no vertex of $P\setminus\{a_1,a_2\}$ is reachable from a vertex different from
$a_2$, then in particular $p_1$ is reachable from $a_2$, and concatenating the
corresponding even $M$-alternating path from $a_2$ to $p_1$ with the edge
$a_1p_1$ yields an $M$-alternating path from $a_2$ to $a_1$ whose internal
vertices avoid $N(x)\cup\{x\}$. So this case is also settled.

We now repeat the same argument from the other end. Namely, working from $a_1$
along $P$, we obtain an $M$-alternating path from $a_1$ to some
$a_\ell\neq a_1$ whose internal vertices avoid $N(x)\cup\{x\}$.

If $a_\ell\in\{a_2,a_4\}$, then the endpoints are consecutive and we are again done.
Hence the only remaining possibility is $a_\ell=a_3$, so that we have obtained
an $M$-alternating path from $a_2$ to $a_4$ and an $M$-alternating path from
$a_1$ to $a_3$, both with internal vertices avoiding $N(x)\cup\{x\}$.

Since these two paths join crossing pairs among
$a_1,a_2,a_3,a_4$ in their cyclic order around $x$, planarity implies that they
intersect. Let $e$ be their first common edge. As above, this common edge is an
edge of $M$. Starting from $a_4$ and following the path to $a_2$ until $e$, and
then continuing from $e$ along the path between $a_1$ and $a_3$, we obtain an
$M$-alternating path from $a_4$ to either $a_1$ or $a_3$. In both cases the
endpoints are consecutive neighbors of $x$, and the internal vertices avoid
$N(x)\cup\{x\}$.

Finally, observe that in the proof we only used the structure of alternating
paths starting at $x$, and not that the endpoints $p_0,p_t$ themselves form a
matching edge. In particular, the same argument applies when $P$ joins any pair
of consecutive vertices among $a_1,a_2,a_3,a_4$, even if this pair is not matched
by $M$ (for example, $a_2$ and $a_3$). Thus the conclusion holds for every such
choice of endpoints.

All constructed paths lie entirely in the factor-critical component $D_0$.

This proves the claim.
\end{proof}

We can now return to the full analysis of Case~4C. By \Cref{clm:adjacent_path_implies_adjacent_alternating}, the existence of a path in the factor-critical component joining two consecutive neighbors of $x$ (with internal vertices avoiding $N(x)\cup\{x\}$) already yields an $M$-alternating path that will help us reduce it to a contractible case by switching along this path, this will be very useful to us in the upcoming case analysis.

\begin{lemma}\label{lemma:case4C_full}
Let $G$ be a planar graph. There exists a maximum matching $M$ of $G$ and a subset $U\subseteq V(G)\setminus V(M)$ of unmatched vertices such that the following hold.
Let $W=V(G)\setminus(U\cup V(M))$ be the (independent) set of unmatched vertices complementary to~$U$.
With respect to $M$, let $G_W$ denote the union of the subgraphs $G_x$ for $x\in W$ and $W^* \subseteq W$ be the set of the vertices of~$W$ in a difficult configuration.
Let $E^*=E[N(W^*)]$ denote the set of edges induced by the neighborhood of $W^*$, and let $D^\#$ be the union of the factor-critical components of vertices in $W^*$ representing Case~4C.

\begin{enumerate}[a)]
 \item Every vertex of $V(G)\setminus V(M)$ is of degree at most~$5$ in~$G$ (cf.~\Cref{cor:lowdeg_free}).
 \item There is a maximum matching $M'$ that differs from $M$ only locally within $D^\#$, and a set of edges $F\subseteq E(G)$ such that $F\cap(E(G_W)\setminus E^*)=M'\cap(E(G_W)\setminus E^*)$, let $H=H(G,F)$ be the corresponding contracted graph and $H'$ be the wheel augmented graph of $H$ at~$W$.
Then every proper $4$-coloring $c_1$ of $H'$ admits a $U$-protecting recoloring to a proper $4$-coloring $c$ of $H'$ such that,
denoting by $c'$ the $4$-coloring of $G$ into which the coloring $c$ lifts (cf.~\Cref{obs:perfect}),
every vertex $x\in W\cup N(W^*)$ has a neighbor $u$ in~$G$ whose color $c'(u)$ is unique in the open neighborhood~$N(x)$.
\item For every vertex $x\in U$ and every $z\in S(x)$ (including $z=x$) such that the degree of $z$ in~$G$ is at most~$5$, let $M'$ be any matching obtained from $M$ in one of the following ways:
\begin{itemize}\vspace*{-1ex}%
\item $M' := M \triangle E(P)$ for an even $M$-alternating $x$--$z$ path in~$G$ (switching from $x$ to~$z$),
\item $M' := M \triangle E(Q)$ for $x=z$ and an $M$-alternating cycle $Q$ which is contained in the same connected component of $G[D(G)]$ as~$x$ (switching a cycle ``nearby~$x$'').
\end{itemize}
Then, with respect to~$M'$;
\begin{itemize}\vspace*{-1ex}
        \item the configuration at $z$ is Case~2B, and if $D_z$ denotes the
        connected component of $G[D(G_z)]$ containing $z$, then $z$ and $D_z$ satisfy
        one of the following structural outcomes:
        \begin{enumerate}
            \item[(i)] $N[z]$ can be contracted to a single vertex of degree at
            least $6$ in such a way that the contraction does not decrease the
            degree of any vertex of the relevant factor-critical
            component (including the ones being contracted);

            \item[(ii)] $D_z$ yields a minor isomorphic to $K_{2,3}$ with
            bipartition $(A,B)$, where the two branch sets corresponding to
            $B$ are contained in $D_z$, and
            \[
                A\subseteq A(G_z),\qquad |A|=3 .
            \]
            This is gadget~\textbf{(II)};

            \item[(iii)] $D_z$ yields a minor obtained by contracting a
            connected subgraph of $D_z$ to a single vertex adjacent to four
            distinct vertices of $A(G_z)$. This is gadget~\textbf{(I)}.
        \end{enumerate}

        \item the configuration at $z$ is Case~4C, and if $D_z$ denotes the
        connected component of $G[D(G_z)]$ containing $z$, then $z$ and $D_z$ satisfy
        one of the following structural outcomes:
        \begin{enumerate}
            \item[(i)] $N[z]$ can be contracted to a single vertex of degree at
            least $6$ in such a way that the contraction does not decrease the
            degree of any vertex of the relevant factor-critical
            component (including the ones being contracted);

            \item[(ii)] $D_z$ yields a minor isomorphic to $K_{2,3}$ with
            bipartition $(A,B)$, where the two branch sets corresponding to
            $B$ are contained in $D_z$, and
            \[
                A\subseteq A(G_z),\qquad |A|=3 .
            \]
            This is gadget~\textbf{(II)};

            \item[(iii)] $D_z$ yields a minor consisting of a triangle whose
            three branch sets are contained in $D_z$, together with two vertices
            $a_1,a_2\in A(G_z)$ adjacent to all three branch sets of the
            triangle. This is gadget~\textbf{(III)};

            \item[(iv)] $D_z$ yields a minor obtained by contracting a connected
            subgraph of $D_z$ to a single vertex adjacent to four distinct
            vertices of $A(G_z)$. This is gadget~\textbf{(I)}.
        \end{enumerate}

        \item the configuration at $z$ is one of the two Kempe-contractible cases (Case~4A, Case~5B) and there exist $4$ pairwise disjoint paths between the neighbors of $z$ and the set~$A(G)$. This is gadget~\textbf{(I)}.
    \end{itemize}
\end{enumerate}
\end{lemma}

\begin{proof}
Again as in proof of \Cref{lemma:case2B_full} and \Cref{lem:controlled} we fix a planar embedding of~$G$.
For a maximum matching $M'$ of $G$, we say that an unmatched vertex $v\in V(G)\setminus V(M')$ is \emph{bad} if the configuration at $v$ with respect to $M'$ is either one of the two difficult cases,
or one of the two Kempe-contractible cases with existing $4$ pairwise disjoint paths between $N(v)$ and~$A(G)$.
We pick $M$ as a maximum matching of $G$ such that all unmatched vertices are of degree at most~$5$ (\Cref{cor:lowdeg_free}) and that the set $U$ of bad vertices unmatched by $M$ is minimal over all such~$M$.
So, in particular, condition a) holds~true.

We will proceed by a case analysis to show, as before, that either we can color the Case~4C configuration or reduce it to a contractible or Kempe-contractible case, otherwise the configuration forces one of the outcomes defined in condition~c). The other bad configurations will remain resolvable exactly as described in \Cref{lemma:case2B_full} and \Cref{lem:controlled} as all our techniques for Case~4C will be consistent with the techniques used in those proofs.

As in \Cref{lemma:case2B_full}, if there exists a vertex $z \in S(x)$ such that switching from $x$ to $z$ creates a case of degree less than $6$ which does not satisfy the conditions in item~c), then we may simply switch from $x$ to $z$. All consistency arguments are the same as in \Cref{lemma:case2B_full}.

Again as in \Cref{lemma:case2B_full}, all switching operations used below are performed along alternating paths contained in the factor-critical component under consideration. Hence the matching is changed only inside this component. By \Cref{lem:Dx_disjoint_Gu}, this does not affect the auxiliary graph $G_u$ of any other unmatched vertex $u\neq x$.

Consequently, all witnesses and local configurations already fixed outside the present component remain unchanged. In particular, vertices outside this component keep their status with respect to $U$, and the wheel-augmentation, Kempe-contractible-case, and Case~2B arguments for them remain valid.

We take $F$ to coincide with $M$ on $G_x$, except for in the neighborhoods of some difficult cases. In case of the Case~2B it will differ from $M$ exactly as described in \Cref{lemma:case2B_full}, in case of Case~4C we will, based on the situation, be changing the matching locally within its factor-critical component, then change the set $F$ according to the new matching and finally in the neighborhood of the Case~4C we might do additional changes to $F$. Finally, note that whenever we change the matching, we modify $F$ accordingly so that it again corresponds to the new matching, aside from the neighborhoods of some difficult cases that require a special set of edges in $F$ as described further in the proof.

Let $a_1,a_2,a_3,a_4$ be the consecutive neighbors of $x$ in cyclic order, where $a_1a_2 \in M$ is the matched pair witnessing Case~4C. By the same argument as in \Cref{clm:adjacent_path_implies_adjacent_alternating}, matching edges between non-adjacent vertices in this cyclic ordering are excluded by planarity.

First observe that no vertex $a_i$ can have degree $1$. If some $a_i$ has degree $3$ or $5$, then by switching along a suitable $M$-alternating path we may expose $a_i$, reducing the configuration either to a contractible case or to Case~5B, which is already handled and can either be colored or produce gadget~\textbf{(I)}, satisfying condition c). Hence we may assume that each $a_i$ has degree $2$, $4$, or at least $6$.

We distinguish two main situations. Either there is no path contained entirely in $D_0$ joining two vertices of $N(x)$ whose internal vertices avoid $N[x]$, or such a path exists. By \Cref{clm:adjacent_path_implies_adjacent_alternating}, the latter also implies an $M$-alternating path.

We begin with the first situation. Throughout the following cases, we assume that there is no path of length greater than $1$ in $D_0$ joining two vertices of $\{a_1,a_2,a_3,a_4\}$ whose internal vertices avoid $N[x]$, until we say otherwise.

\medskip
\noindent
\textbf{Case $I$.} At least two vertices among $a_1,a_2,a_3,a_4$ have degree $2$.

If two such vertices are consecutive neighbors in the cyclic order around $x$, they form a handle, contrary to our assumptions. Hence all degree-$2$ vertices are non-adjacent.

Each of them can therefore be treated independently as an instance of Case~2B. By \Cref{lemma:case2B_full}, each is resolvable by the lemma or contracted without decreasing any degree in $D_0$. Since they are non-adjacent and no path in $D_0$ joins vertices of $N(x)$ outside $N(x)\cup\{x\}$, these operations do not interfere, otherwise they produce a gadget satisfying condition c).

    \medskip
\noindent
\textbf{Case $II$.} At least two vertices among $\{a_1,a_2,a_3,a_4\}$ have degree $4$.

We distinguish several subcases.

\smallskip
\noindent
\emph{Case $II$a.} A vertex of degree $4$ has exactly one neighbor outside $N(x)\cup\{x\}$.

Without loss of generality, let this vertex be $a_1$. Then we may switch along the path $x,a_2,a_1$, exposing $a_1$ and reducing the configuration to a contractible case.

\smallskip
\noindent
\emph{Case $II$b.} Two vertices of degree $4$ have all their neighbors inside $N(x)\cup\{x\}$.

Without loss of generality, assume these vertices are $a_1$ and $a_3$. Then they are adjacent to all other vertices in $\{a_1,a_2,a_3,a_4,x\}$, forming a $K_4$ on $\{a_1,a_2,a_3,x\}$ and $\{a_1,a_4,a_3,x\}$

Observe that two consecutive vertices in the cyclic ordering around $x$ (or the matched pair $a_1,a_2$) cannot both have all their neighbors inside $N(x)\cup\{x\}$, as this would create a $K_5$, contradicting planarity. Hence $a_2$ and $a_4$ must have neighbors outside this set.

If either $a_2$ or $a_4$ has degree at most $5$, then we reduce to a previously handled case: degree $3$ or $5$ yields a contractible case or Case~5B, and degree $4$ with an external neighbor allows a switching reduction (along $x,a_1,a_2$ or $x,a_3,a_4$) to Case~4B. Thus, we may assume that both $a_2$ and $a_4$ have degree at least $6$.

Since two of the degree $4$ vertices have all their neighbors in $N[x]$, the remaining two vertices cannot be on the same face, as that would imply we could connect them, while preserving planarity, which would create a $K_5$. Therefore, the degree $6$ vertices cannot share a single neighbor outside of $N[x]$. Hence, after contracting the configuration on $\{x,a_1,a_2,a_3,a_4\}$ into a single vertex, each of $a_2$ and $a_4$ loses at most three incident edges but gains at least three new ones, so their degree remains at least $6$. The vertices of degree $4$ become adjacent to at least six vertices after contraction. Hence, condition c) holds.

\smallskip
\noindent
\emph{Case $II$c.} Two vertices of degree $4$ each have exactly two neighbors outside $N(x)\cup\{x\}$.

Note that these neighbors of the degree $4$ vertices must be matched together, otherwise we could switch $x$ to this vertex and reduce to a contractible case~4B. However, if they are matched together, then they are in $D_0$, as they are in $S(x)$ and also adjacent to the degree $4$ vertex which is in $D_0$.

The remaining vertices in $N(x)$ cannot both have degree $2$ (by Case~$I$). If one of the remaining vertices has degree $4$, then in this case there are at least six distinct neighbors of the configuration outside $N(x)\cup\{x\}$. After contracting the configuration into a single vertex, each vertex loses at most four incident edges but gains at least four new ones, so no degree decreases. Since there are no internal paths between the vertices of $N(x)$ outside the configuration, the degrees of all vertices in $D(G_x)$ are preserved. Hence, the condition c) holds.

If the remaining vertices have degree at least $6$, then, if each such vertex has at least three neighbors outside $N(x)\cup\{x\}$, then the same contraction argument as above applies and yields condition c). Note that the vertices of degree at least $6$ can share at most $2$ neighbors, as they are not connected by a path in $D_0$, as such, the only vertices they can share are in $A(G_x)$, however if they shared $3$ or more vertices we could immediately get a gadget~\textbf{(II)} with $A = A(G_x)$.

This completes Case~$II$.

    \medskip
\noindent
\textbf{Case $III$.} At least two vertices among $\{a_1,a_2,a_3,a_4\}$ have degree at least $6$.

\smallskip
\noindent
\emph{Case $III$a.} Two vertices have degree at least $6$, and at least one of the remaining vertices has degree $2$ or $4$.

Since we assume that no switching reduces the configuration to a contractible case, the vertex of degree $4$ cannot be safely switched, and hence both vertices of degree at least $6$ have at most three neighbors in $N(x)\cup\{x\}$. Therefore each of them has at least three neighbors outside the configuration.

If the two degree-$6$ vertices have the same three or more neighbors outside the configuration, then, together with these vertices in $A(G_x)$, they form a $K_{2,3}$, yielding a gadget of type~\textbf{(II)} with~$A=A(G_x)$.

Otherwise, they do not share all their external neighbors. If they share at most two such neighbors, we contract the configuration on $N(x)\cup\{x\}$ into a single vertex. Each degree-$6$ vertex loses at most three incident edges (to the contracted set) but gains at least three new ones (coming from the other vertex and from distinct external neighbors), so its degree does not decrease.

Moreover, no vertex in the configuration loses degree, and since there are no paths in $D_0$ between vertices of $N(x)$ outside the configuration, the degrees of other vertices in $D(G_x)$ are also preserved. Hence condition c) holds.

Note that the degree $4$ vertex cannot have all neighbors in $N[x]$, as that would imply that we can reduce to a contractible Case~2A, since we could switch to the degree $2$ vertex, and then there would be a path fully contained in the factor-critical component between its immediate neighbors (\Cref{lem:case2B_no_long_path}).

\smallskip
\noindent
\emph{Case $III$b.} All four vertices have degree at least $6$.

If every vertex has at most three neighbors in $N(x)\cup\{x\}$, then a similar contraction argument as above applies (also shown in \cref{lem:reduce_low_degree_component}): each vertex has at least three external neighbors, and after contracting the configuration, no degree decreases. Thus, condition c) holds.

\smallskip
\noindent
It remains to consider the case where some vertices have four neighbors in $N[x]$. Without loss of generality, assume that $a_2$ and $a_4$ each have four neighbors in $N(x)\cup\{x\}$.

We now analyze their neighbors outside the configuration. Since we assume that there are no paths inside $D_0$ between non-adjacent vertices of $N(x)$, any common neighbor of $a_2$ and $a_4$ outside $N(x)\cup\{x\}$ must lie in $A(G_x)$.

We claim that $a_2$ and $a_4$ can share at most two such neighbors with $a_3$ or $a_1$. Indeed, suppose for contradiction that they share three distinct neighbors in $A(G_x)$. Then these three vertices, together with $a_2$ and the contracted $a_1,x,a_2$, form a $K_{2,3}$ subgraph, yielding a gadget of type~\textbf{(II)} with $A = A(G_x)$.

Hence $a_2$ and $a_4$ have at most two common neighbors outside $N(x)\cup\{x\}$ with $a_1$ or $a_2$.

\smallskip
Now consider the effect of contracting the configuration induced by $N(x)\cup\{x\}$ into a single vertex.

The vertices $a_1$ and $a_3$ each have at most three neighbors in $N(x)\cup\{x\}$, and since they have degree at least $6$, each of them has at least three neighbors outside the configuration. Moreover, these external neighbors are distinct: indeed, the edge $a_2a_4$ separates the regions containing the external neighbors of $a_1$ and $a_3$, so $a_1$ and $a_3$ cannot share an external neighbor without creating a forbidden crossing. Hence, after contraction, each of $a_1$ and $a_3$ loses at most three edges inside the configuration but gains at least three new neighbors through the external neighbors of the other vertex. Therefore, their degrees do not decrease.

Now consider $a_2$ and $a_4$. Each of them has four neighbors in $N(x)\cup\{x\}$, so after contraction each loses at most four edges inside the configuration. On the other hand, the vertices $a_1$ and $a_3$ together contribute six distinct external neighbors. Since $a_2$ and $a_4$ can share at most two of these six neighbors, each of $a_2$ and $a_4$ gains at least four new neighbors after the contraction. Hence, their degrees also do not decrease.

Therefore, after contracting the whole configuration, no vertex in $D_0$ loses degree. Hence, condition c) holds.

\medskip
This completes Case~$III$.

    \medskip
\noindent

\textbf{Case $IV$.} There exists a path $P$ in $D_0$ joining neighbors of $x$.

Note that we always mean a path $P$ avoiding vertices of $N[x]$.

We further divide this case according to whether there exists a path contained entirely in $D_0$ joining two consecutive vertices in the cyclic ordering around $x$ and one that has only paths connecting non-consecutive vertices in the cyclic ordering around $x$.

\textbf{Case $IVa$.} $P$ joins two consecutive neighbors of $x$.

Without loss of generality assume that
\[
a_1a_2,\ a_3a_4 \in M.
\]
By \Cref{clm:adjacent_path_implies_adjacent_alternating}, the existence of the path implies that there exists an $M$-alternating path between two adjacent vertices among $a_1,a_2,a_3,a_4$ whose internal vertices avoid $N[x]$.

We distinguish two cases.

\smallskip
\noindent
\textbf{Case 1.} There exists an $M$-alternating path between two consecutive vertices that are matched together, say between $a_1$ and $a_2$.

Since $a_1a_2\in M$, this alternating path starts and ends with edges not in $M$. Together with the edge $a_1a_2$, it forms an $M$-alternating cycle $C$ contained entirely in the factor-critical component of $x$.

Define a new matching
\[
M' := M \triangle C.
\]
This switching is performed entirely within the factor-critical component of $x$ and preserves maximality of the matching. After the switch, the edge $a_1a_2$ is no longer in the matching, and $x$ becomes adjacent to two vertices that are not matched together. Thus, the configuration at $x$ reduces to a contractible case (specifically Case~4B).

\smallskip
\noindent
\textbf{Case 2.} There exists an $M$-alternating path between two consecutive vertices that are not matched together, say between $a_2$ and $a_3$.

Let $P$ be such a path. Since $a_2$ and $a_3$ are matched to $a_1$ and $a_4$, respectively, the path $P$ starts and ends with edges not in $M$. Consider the path
\[
P' := a_1a_2  P  a_3a_4,
\]
which is an $M$-alternating path whose endpoints are $a_1$ and $a_4$.

Define
\[
M' := M \triangle P'.
\]
After this switching, the vertices $a_1$ and $a_4$ become unmatched.

If $a_1a_4 \in E(G)$, then we add this edge to $M'$ and obtain a matching in which $x$ is reduced to a contractible case.

If $a_1a_4 \notin E(G)$, we add either $a_1x$ or $a_4x$ to $M'$ just to preserve the maximum matching, furthermore we also add both edges into our set $F$, which always copies the matching after any change to it. In the contracted wheel augmented graph $H(G,F)$, any proper coloring yields a conflict-free coloring of $G_x$: the vertices $a_1$ and $a_4$ have a uniquely colored neighbor in $x$, and $x$ has a uniquely colored neighbor among $\{a_2,a_3\}$, since these vertices are adjacent (by the wheel) and hence receive distinct colors.

All these modifications are illustrated in \Cref{fig:4C_path_matched} and \Cref{fig:4C_path_non_matched}.
\end{toappendix}

\begin{toappendix}

\smallskip

In both cases, the configuration at $x$ is either resolved by the contractions and satisfies condition~b) or reduced to a contractible case again satisfying condition~b). If $x$ was in $U$, this would contradict the minimality of $U$, hence existence of path $P$ joining two consecutive vertices in the cyclic ordering around $x$ implies that $x \notin U$.

\textbf{Case $IVb$.} $P$ joins two non-consecutive neighbors of $x$.

Without loss of generality, suppose that there is a path in $D_0$ joining $a_1$ and $a_3$.

The remaining pair $a_2,a_4$ is not joined by such a path. Indeed, the vertices $a_2$ and $a_4$ lie in different faces separated by the cycle formed by $a_1,x,a_3$ and the path from $a_3$ to $a_1$. Hence any path between $a_2$ and $a_4$ would have to meet this cycle. This would give a path between two consecutive neighbors of $x$ in the cyclic ordering, contrary to the present case.

By the same argument as in \Cref{clm:adjacent_path_implies_adjacent_alternating}, the path between $a_1$ and $a_3$ inside $D_0$ implies either that there is an even $M$-alternating path between $a_1$ and $a_3$, or that two consecutive neighbors of $x$ in the cyclic ordering are joined by an $M$-alternating path avoiding $N[x]$. The latter is excluded in the present case. Hence there is an even $M$-alternating path inside $D_0$ from $a_1$ to $a_3$. Switching along this path produces a matching $M'$ in which the vertices $x,a_2,a_4$ are unmatched, while $a_1$ and $a_3$ are matched to vertices outside $N(x)$.

We first note that there is no edge $a_2a_4$. Indeed, if such an edge existed, the edge $a_2a_4$ would cross the cycle formed by $a_1,x,a_3$ and the path from $a_3$ to $a_1$, contradicting planarity.

Now, if we contract $a_2,x,a_4$ any proper coloring assigning $a_1$ and $a_3$ a distinct colors lifts to a valid conflict-free coloring of the $G_x$. Indeed, $x$ has a unique color in either $a_1$ or $a_3$, $a_2$ and $a_4$ have a unique color in $x$, and all the matched vertices retain their unique color.

Thus, it remains to consider the case where $a_1$ and $a_3$ receive the same color. If $a_1$ and $a_3$ lie on a common face of the contracted graph, then we can add the edge $a_1a_3$ in that face and thereby force them to receive distinct colors in a proper coloring. Hence, we may assume that no such edge can be added.

We use the following elementary planar observation. Let
\[
C:=P\cup xa_1\cup xa_3 .
\]
Since the neighbors of $x$ occur in the cyclic order $a_1,a_2,a_3,a_4$, the cycle $C$ separates $a_2$ from $a_4$. Let $R_2$ and $R_4$ denote the two closed regions bounded by $C$, chosen so that $a_j\in R_j$ for $j=2,4$.

Consider one of these regions, say $R_j$. If there is no path in $R_j$ from $a_j$ to an internal vertex of $P$, then the side of $P$ facing $R_j$ is free for the augmentation: the edge $a_1a_3$ can be drawn in $R_j$, sufficiently close to $P$, without crossing the drawing of $G$. Indeed, if such an arc could not be drawn, then some connected part of the drawing in $R_j$ would separate the side of $P$ from the corner of $R_j$ containing the edge $xa_j$. Thus some connected component of this part of the drawing meets both $a_j$ and $V(P)\setminus{a_1,a_3}$. This component contains a path from $a_j$ to an internal vertex of $P$.

Consequently, if the edge $a_1a_3$ cannot be added on either side of $P$, then for each $j\in{2,4}$ there exists a path $Q_j$ from $a_j$ to $V(P)\setminus{a_1,a_3}$. We choose $Q_j$ minimal, so that it meets $P$ only in its final vertex.

Moreover, $Q_j$ cannot be contained entirely in $D_0$. For if, say, $Q_2\subseteq D_0$, then $Q_2$, together with one of the two sub-paths of $P$ from its endpoint to $a_1$ or to $a_3$, would give a path in $D_0$ from $a_2$ to a consecutive neighbor of $x$, with internal vertices avoiding $N[x]$. This is excluded by the present case. Hence, each $Q_j$ must leave $D_0$, and therefore it meets $A(G_x)\subseteq A(G)$.

Thus, after the wheel augmentation at $x$, the cycle $P\cup xa_1\cup xa_3$, together with the two paths $Q_2,Q_4$, gives the required gadget of type~\textbf{(III)}.

This construction is illustrated in \Cref{fig:xmatched_four_all}.

Notice first that all changes to $F$ and changes of matching used above are performed inside the factor-critical component under consideration. Hence they do not affect the auxiliary graphs or configurations associated with other unmatched vertices. In particular, all already resolved unmatched vertices outside the present component retain their unique-colored witnesses, and the resolutions obtained from \Cref{lem:controlled} and \Cref{lemma:case2B_full} remain valid.

Conversely, the recolorings used in the Kempe-contractible cases and the resolutions described in \Cref{lemma:case2B_full} do not interfere with the arguments above, since they are either confined to their own factor-critical components or only perform a permutation of the colors. Therefore, once a resolvable Case~4C configuration is fixed, it remains compatible with the other resolution methods used so far, and it stays resolvable after recolorings or matching switches performed in the other fixed cases. 

Thus, we can either change the matching and find a desirable set of edges $F$ (same as current matching on all edges except the neighborhood of some vertices representing either Case~2B or Case~4C) to obtain a valid conflict-free coloring satisfying condition b) or obtain a property satisfying condition c). Note that all other unmatched vertices not representing Case~4C can be argued the same way as in \Cref{lem:controlled} and \Cref{lemma:case2B_full}.

\end{proof}

\subsection{Density of factor-critical components and consistency of switching}

In this section we take a closer look at the consequences of our analysis of the difficult cases. Suppose that $x$ is an unmatched vertex representing one of the difficult configurations. By the lemmas above, one of the following happens: either the configuration can be resolved using one of the techniques we have shown, or it produces one of the desired gadgets, or the corresponding closed neighborhood can be contracted without decreasing any degree, while producing a vertex of degree at least $6$.

For the moment, ignore the last possibility, namely the safe-contraction outcome. Then every difficult configuration which is not resolvable produces a gadget. In particular, if every switchable vertex in $G_x$ is either Kempe-contractible with four disjoint paths to distinct vertices of $A(G_x)$, or represents a difficult case which cannot be resolved by the techniques developed above, then every relevant factor-critical component contains one of the desired gadgets. By \Cref{lemma_gadgets}, this cannot happen in a planar graph. Hence at least one switchable vertex must be contractible, resolvable by the techniques introduced above, or belong to the remaining safe-contraction case.

It remains to deal with this last possibility where the difficult cases admit the special contractions. We want to show that it does not give a new obstruction: if all unresolved difficult vertices in a factor-critical component only lead to safe contractions, then the whole component contains one of the desired gadgets. Equivalently, such a component either contains a gadget, or it contains a vertex which can be resolved by the techniques already developed. This is the content of the following lemma.

\begin{lemma}
\label{lem:usecase_of_density}
Let $G$ be a planar graph. There exists a maximum matching $M$ of $G$ and a subset $U\subseteq V(G)\setminus V(M)$ of unmatched vertices such that the following hold.
Let $W=V(G)\setminus(U\cup V(M))$ be the (independent) set of unmatched vertices complementary to~$U$.
With respect to $M$, let $G_W$ denote the union of the subgraphs $G_x$ for $x\in W$ and $W^* \subseteq W$ be the set of the vertices of~$W$ in a difficult configuration.
Let $E^*=E[N(W^*)]$ denote the set of edges induced by the neighborhood of $W^*$, and let $D^\#$ be the union of the factor-critical components of vertices in $W^*$ representing Case~4C.
\begin{enumerate}[a)]
     \item Every vertex of $V(G)\setminus V(M)$ is of degree at most~$5$ in~$G$ (cf.~\Cref{cor:lowdeg_free}).
     \item There is a maximum matching $M'$ that differs from $M$ only locally within $D^\#$, and a set of edges $F\subseteq E(G)$ such that $F\cap(E(G_W)\setminus E^*)=M'\cap(E(G_W)\setminus E^*)$, let $H=H(G,F)$ be the corresponding contracted graph and $H'$ be the wheel augmented graph of $H$ at~$W$.
    Then every proper $4$-coloring $c_1$ of $H'$ admits a $U$-protecting recoloring to a proper $4$-coloring $c$ of $H'$ such that,
    denoting by $c'$ the $4$-coloring of $G$ into which the coloring $c$ lifts (cf.~\Cref{obs:perfect}),
    every vertex $x\in W\cup N(W^*)$ has a neighbor $u$ in~$G$ whose color $c'(u)$ is unique in the open neighborhood~$N(x)$.
    
    \item For any $x \in U$, let $D_0$ be the factor-critical component containing $x$. Then either $D_0$ contains a gadget as a subgraph or $x \notin U$.
\end{enumerate}
\end{lemma}
\begin{proof}
    Assume that $x \in U$, then $x$ cannot be contractible. Indeed, in that case we could simply remove $x$ from $U$, contradicting the minimality of $|U|$.

    If $x$ is Kempe-contractible, then either it is colorable as in \Cref{lem:controlled}, in which case we again contradict the minimality of $|U|$, or there exist four disjoint paths from the neighbors of $x$ to four distinct vertices of $A(G_x)$. In the latter case we obtain gadget~\textbf{(I)}.
    
    Suppose next that $x$ represents one of the difficult cases. Then the techniques developed in the lemmas prior either resolve it directly, or reduce it to one of the previous cases. In the first case we again contradict the minimality of $|U|$. In the second case we either obtain the same contradiction, or obtain gadget~\textbf{(I)}. If neither of these possibilities occurs, then by \Cref{lemma:case2B_full} and \Cref{lemma:case4C_full} we can obtain a minor of one of the desired gadgets, and hence the required conclusion follows.
    
    It remains to consider the last possible case. Namely, every vertex of $D_0$ either has degree at least $6$, or represents one of the difficult cases whose closed neighborhood can be contracted without decreasing the degree of any vertex of $D_0$ and such that the resulting vertex has degree at least $6$. In this case, $D_0$ together with $A(G_x)$ satisfies all assumptions of \Cref{lem:reduce_low_degree_component}. Indeed, there are no paths inside $D_0$ between the relevant neighbors of our degree-$2$ or degree-$4$ vertices, vertices of degree $2$ and $4$ are not adjacent to vertices of $A(G_x)$, and any two adjacent vertices of degree $2$ or $4$ have exactly $1$ or $3$ common neighbors. Hence, by applying \Cref{lem:reduce_low_degree_component}, we obtain that $D_0$ contains either gadget~\textbf{(I)} or gadget~\textbf{(II)}.

\end{proof}

We also prove a consistency property for switching operations. Namely, let $U$ be the set of unmatched vertices chosen as in the lemma above and in the preceding arguments. If we switch an unmatched vertex $x\in U$ to some vertex $z\in S(x)$, then the resulting maximum matching does not force any already resolved, fixed, unmatched vertex to enter $U$.

\begin{lemma}
\label{lem:consistency_full}
Let $G$ be a planar graph. There exists a maximum matching $M$ of $G$ and a subset $U\subseteq V(G)\setminus V(M)$ of unmatched vertices such that the following hold.
Let $W=V(G)\setminus(U\cup V(M))$ be the (independent) set of unmatched vertices complementary to~$U$.
With respect to $M$, let $G_W$ denote the union of the subgraphs $G_x$ for $x\in W$ and $W^* \subseteq W$ be the set of the vertices of~$W$ in a difficult configuration.
Let $E^*=E[N(W^*)]$ denote the set of edges induced by the neighborhood of $W^*$, and let $D^\#$ be the union of the factor-critical components of vertices in $W^*$ representing Case~4C.
\begin{enumerate}[a)]
     \item Every vertex of $V(G)\setminus V(M)$ is of degree at most~$5$ in~$G$ (cf.~\Cref{cor:lowdeg_free}).
     \item There is a maximum matching $M'$ that differs from $M$ only locally within $D^\#$, and a set of edges $F\subseteq E(G)$ such that $F\cap(E(G_W)\setminus E^*)=M'\cap(E(G_W)\setminus E^*)$, let $H=H(G,F)$ be the corresponding contracted graph and $H'$ be the wheel augmented graph of $H$ at~$W$.
    Then every proper $4$-coloring $c_1$ of $H'$ admits a $U$-protecting recoloring to a proper $4$-coloring $c$ of $H'$ such that,
    denoting by $c'$ the $4$-coloring of $G$ into which the coloring $c$ lifts (cf.~\Cref{obs:perfect}),
    every vertex $x\in W\cup N(W^*)$ has a neighbor $u$ in~$G$ whose color $c'(u)$ is unique in the open neighborhood~$N(x)$.

    \item Switching $x$ to $z$ does not increase $|U|$.

\end{enumerate}
\end{lemma}
\begin{proof}
    Since $U$ contains only unmatched vertices, we consider why each possible type of unmatched vertex cannot be introduced into $U$ by switching from $x$ to $z$ and hence changing the matching. Suppose for a contradiction that $u$ is introduced into $U$ as a result of the switch. By \Cref{lem:switch_preserves_other_local_configs}, the configurations around unmatched vertices are preserved. Hence, if $u$ represents a contractible configuration, then it remains contractible with respect to $M'$, and therefore remains outside $U$.
    
    Suppose next that $u$ represents a Kempe-contractible configuration. Since $u$ was not originally in $U$, by \Cref{lem:controlled} it does not contain four disjoint paths to vertices of $A(G_u)$. Since the switch does not change the factor-critical component of $u$ and does not create or delete edges, it cannot create such disjoint paths. Hence $u$ cannot be introduced into $U$, because otherwise we could still apply the Kempe-chain recoloring argument from the proof of \Cref{lem:controlled}, contradicting the minimality of $|U|$.
    
    Suppose that $u$ represents Case~2B, then from the case analysis in \Cref{lemma:case2B_full}, we know that this can happen only if we can device such $F$, that it omits the edge between its neighbors. However, by switching the matching from $x$ to $z$, we did not modify the factor-critical component containing $u$. Hence, the degree and matching of the neighborhood of the neighboring vertices of $u$ stays the same and as such the same resolution of not including the edge into $F$ still applies.

    Finally, suppose that $u$ corresponds to Case~4C. From the case analysis in \Cref{lemma:case4C_full} we know that this can only happen if there is $M'$ differing from $M$ only within the factor-critical of $u$ and $F$ such that the special resolution from \Cref{lemma:case4C_full}, where we added adjacent edges to $F$ applies. However since all these operations are contained within the factor-critical component of $u$, they can be performed even after the switch from $x$ to $z$, hence $u$ remains outside of $U$.
    
    Since we have considered all possible configurations of $u$, and none of them can belong to $U$ after the switch, we have shown that the size of $U$ does not increase.
\end{proof}

Intuitively, we have shown that once a case is resolvable, switching in another part of the graph does not destroy this resolution. Thus already fixed resolvable cases remain fixed under the switching operations used for other unmatched vertices.

\end{toappendix}

\section{Completing the main proof}\label{sec:completing}

We now assemble all previous ingredients into a proof of the main result.

\begin{theorem}[Combinatorial side of \Cref{thm:main}]
\label{thm:main1}
Every planar graph $G$ admits an open-neighborhood conflict-free coloring with at most $4$ colors.
\end{theorem}

\begin{proof}
For start, we remove all isolated vertices of~$G$ (since \Cref{def:CFON} does not impose restriction on them).
Let $F_1\subseteq E(G)$ be an edge set defining the contracted graph $H_1=H(G,F_1)$, let $M_1\subseteq F_1$ be a matching (not necessarily maximum)
such that no edge of $M_1$ shares a vertex with an edge of $F_1\setminus M_1$, and~$U_1\subseteq V(G)\setminus V(M_1)$.
Assume that $c_1$ is a proper $4$-coloring of $H_1$ which lifts into a $4$-coloring $c_1'$ of $G$,
satisfying that every vertex of $V(G)\setminus(U_1\cup V(M_1))$ has a neighbor in~$G$ of a unique color in~$c_1'$.

First observe that if $U_1=\emptyset$, then $c_1'$ is a CFON $4$-coloring of~$G$.
Indeed, if $v\in V(M_1)$ and $w$ is the $M_1$-neighbor of~$v$, then the color $c_1'(w)=c_1'(v)$ is unique in the neighborhood of~$v$ in~$G$ since $c_1$ is proper in~$H_1$.
If, otherwise, $v\in V(G)\setminus V(M_1)$, then (since $U_1=\emptyset$) $v$ has a neighbor in~$G$ with a unique color in $c_1'$ by our assumption.

In order to find such $M_1$ and $U_1=\emptyset$ in given~$G$, we proceed as follows.
We start with fixing a maximum matching $M$ of the graph $G$ and a subset $U$ of unmatched vertices as in \Cref{lem:controlled},
which satisfy the previous initial conditions when setting~$M_1=F_1=M$ and~$U_1=U$ by \Cref{lem:controlled}\,b).
Let $W=V(G)\setminus(U\cup V(M))$ and $G_{W}$ denote (with respect to~$M$) the union of the subgraphs $G_x$ for $x\in W$.
Among all choices of a matching $M_1$ obtained from $M$ by a sequence of switchings avoiding~$E(G_W)$, 
a set $F_1\subseteq E(G)$, $F_1\supseteq M_1$ such that no edge of $M_1$ shares a vertex with an edge of $F_1\setminus M_1$, 
and $U_1\subseteq U$, we assume one such that:
$U_1$ is inclusion-minimal for which there is a proper $4$-coloring $c_1$ of~$H_1=H(G,F_1)$, 
obtained from an arbitrary proper $4$-coloring of $H_1$ by a $U$-protecting recoloring as in \Cref{lem:controlled}\,b),
such that $c_1$ satisfies the initial assumptions stated above.
Note that, in particular, $M_1\cap E(G_W)=M\cap E(G_W)=F_1\cap E(G_W)$. 

As sketched in \Cref{sec:difficult2} and detailed in the Appendix (\Cref{lemma:case2B_full} and \Cref{lemma:case4C_full}), for any vertex $x\in U_1\not=\emptyset$, one of the following three possibilities must hold.
\begin{itemize}
\item In every connected component $D_0$ of $G[D(G_x)]$ we obtain, by a possible independent wheel augmentation and suitable contractions restricted to $D_0$
(which both preserve planarity), one of the gadgets of \Cref{lemma_gadgets} (Appendix: \Cref{lem:usecase_of_density}).
Since there are $|A(G_x)|+1$ such gadgets by \Cref{clm:Gx_structure}, \Cref{lemma_gadgets} then contradicts planarity of~$G$.
\item Within one of the connected components of $G[D(G_x)]$ we can switch the matching $M$ to obtain one of the contractible or Kempe-contractible cases.
This contradicts the conclusions of \Cref{lem:controlled}\,c) or yields the gadget~\textbf{(I)} as covered by the previous point.
\item There is $z\in S(x)$ such that we switch the matching $M_1$ to $M_2$ along an $M_1$-alternating path from $x$ to $z$
(this does not interfere with already resolved/fixed vertices of $U$, in other words it does not enlarge $U_1$, cf.~\Cref{lem:consistency_full}),
and locally modify the contract set $F_1$ to $F_2$ within the component of $G[D(G_x)]$ containing~$z$.
The local configuration at $z$ is Case~2B or Case~4C, and as shown in the detailed analysis of these two cases,
in each case $M_2$ and $H_2=H(G,F_2)$ can be chosen to satisfy the initial assumptions of this argument for $U_2=U_1\setminus\{x\}$.
This contradicts minimality of~$U_1$.
\end{itemize}

Hence, indeed, $U_1=\emptyset$ and the proof is finished.
\end{proof}

\section{Conflict-free $4$-coloring algorithmically}\label{sec:algorithm}

The proof of \Cref{thm:main1}, as presented above, is itself almost entirely efficiently algorithmic. There are two points in this procedure that deserve comment in advance.

First, we start in \Cref{lem:controlled} with a specific maximum matching of the input graph~$G$, namely the one of \Cref{cor:lowdeg_free}, and this is not simply provided by the algorithm of \Cref{thm:blossom}. Instead, we will use a two-step self-reduction approach.

Second, our proof relies on excluding the possible gadgets of \Cref{lemma_gadgets} before trying to color the neighborhood of an unmatched vertex.
In the algorithm, we instead blindly try the modifications and recolorings anticipated by our proof at the unmatched vertices, and either find an admissible resolution, or mark this case as yielding a gadget (which we do not need to know explicitly).
This paradigm shift leads to necessity to re-iterate some of the proof steps in the algorithm, such as the step of proper $4$-coloring of a contracted graph.

In view of the previous remarks, we assert:

\begin{theorem}[Algorithmic side of \Cref{thm:main}]
\label{thm:main2}
An open-neighborhood conflict-free $4$-coloring of an input planar graph $G$ can be found in polynomial time.
\end{theorem}

\begin{proof}
The algorithm works, on a high level, in three phases: (i) initialization of a matching, (ii) resolution of all unmatched vertices, and (iii) coloring and postprocessing.
As noted above, in some cases the algorithm can return from phase (iii) to phase (ii) with a new unmatched vertex.
\begin{enumerate}[(i)]\parskip3pt
\item In the initial phase, we first compute a matching $M$ as in \Cref{cor:lowdeg_free}. This is done by two calls to the algorithm of \Cref{thm:blossom}; first one computes any maximum matching and the number $t$ of unmatched vertices of $G$.
	Then, denoting by $G^+$ the graph obtained from $G$ by adding $t$ new vertices adjacent to all vertices of degree $\leq5$ in $G$, the second call computes a perfect matching $M^+$ of $G^+$ and sets $M$ to be the restriction of $M^+$ to~$G$.
	Subsequent computation of the Gallai--Edmonds decomposition of $G$ with respect to $M$ and of the components of $G[D(G)]$ can be done by local search in linear time.
\item In this phase we process, in a greedy order, all currently unmatched vertices $x$ (with respect to $M$) that are not yet marked as unresolvable.
	In each case of $x$, we determine the local configuration at $x$ and at the vertices switchable from~$x$ by a local search within~$G_x$. Then:
    \begin{itemize}\parskip0pt
	\item If the local configuration at $x=z$ is one of the contractible cases, or $x$ can be switched to $z$ being in a contractible case,
	then we annotate it as needing a wheel augmentation at $z$ and mark the vertex $z$ as resolved.
	\item Otherwise, if $x$ is one of the two Kempe-contractible cases, then we annotate this case with a wheel augmentation at $x$ and local modification steps anticipated by the proof of \Cref{lem:controlled}, and mark $x$ as potentially resolved.
	\item If $x$ is one of the two difficult cases (note that $x$ already cannot be switched to a contractible case),
	and the specific local modification steps anticipated by the proofs of these cases can be carried out there,
	then we annotate the case with these modification steps and again mark $x$ as potentially resolved.
	If not, we can conclude that the component of $G[D(G)]$ containing $x$ yields a gadget of \Cref{lemma_gadgets}, mark $x$ as unresolvable and continue the main cycle with switching to another unmatched vertex of degree at most~$5$.
    \end{itemize}
\item In the last phase, we first complete the previously annotated wheel augmentations and local modifications, and make the contracted graph $H$.
	We find a proper $4$-coloring of $H$ by the algorithm of \Cref{thm:FCT}, apply previously annotated recolorings at the vertices marked as potentially resolved, and lift the resulting coloring back to~$G$. 
	If we succeed, we stop the algorithm with this lifted coloring as the output.

	If the lifted coloring is not conflict-free at some `potentially resolved' vertex~$x$, then we again conclude that the component of $G[D(G)]$ containing $x$ yields a gadget of \Cref{lemma_gadgets} and mark $x$ as unresolvable.
	We switch the matching from $x$ to another vertex $y\in S(x)$ of degree at most~$5$ such that $y$ is not yet marked as unresolvable, and we return to phase (ii) with $y$.

	Since $G$ is planar, not every connected component of $G[D(G_x)]$ may yield a gadget of \Cref{lemma_gadgets}, and so we must eventually succeed in a resolved case above.
\end{enumerate}

The runtime of our algorithm is dominated by three parts; computing a maximum matching in a planar graph which is $\mathcal O(n\sqrt n)$ by \cite{MicaliV80};
the overall time for processing and postprocessing components of all inessential vertices of $G$, upper-bounded by $\mathcal O(n^2)$;
and the time needed to compute a $4$-coloring of a planar graph which is $\mathcal O(n\log n)$ by \cite{inoue2026colortheoremlinearlyreducible}, multiplied by up to $\mathcal O(n)$ iterations.
These sum to overall $\mathcal O(n^2\log n)$ runtime.

A complete pseudocode can be, due to space restriction, found in the Appendix.
\end{proof}

\begin{toappendix}

We add a complete pseudocode of our algorithm right below. 

In the algorithm below we distinguish the current matching $M$ from the contraction set~$F$, which defines the contracted graph $H(G,F)$. Initially $F=M$. During the algorithm, $M$ may be changed by switching along alternating paths, while $F$ is updated accordingly, except for the local modifications explicitly prescribed by the difficult cases.

We shall also use the term \emph{local resolution} for Cases~2B and~4C. This means one of the local transformations described in the proofs of \Cref{lemma:case2B_full} and \Cref{lemma:case4C_full}. More precisely, a local resolution is a modification supported inside the relevant factor-critical component of $G[D(G)]$, and it is of one of the following two types. Either we switch the matching locally so that the resulting unmatched vertex falls into one of the already contractible configurations, or we modify the contraction set $F$ locally by adding or deleting the prescribed contraction edges. In both cases the resulting contracted graph remains planar, previously fixed components are not affected, and the configuration is either resolved immediately or recorded for the final protected recoloring step.

\smallskip
\noindent\textbf{Algorithm \refstepcounter{algorithm}\thealgorithm. 
Constructing a CFON $4$-coloring of a planar graph}
\label{alg:CFON_planar}
\begin{algorithmic}[1]
\Require A planar graph $G$ without isolated vertices
\Ensure An open-neighborhood conflict-free coloring of $G$ with at most $4$ colors

\State Compute a planar embedding of $G$
\State Remove all handles from $G$ and store them for later reinsertion
\State Compute the size of a maximum matching of $G$ and the matching defect $t$
\State Compute a maximum matching $M$ of $G$ conforming to \Cref{cor:lowdeg_free}, using self-reduction based on the parameter~$t$
\State Compute the Gallai--Edmonds decomposition of $G$ with respect to $M$
\State Let $F:=M$. This set will define the final contracted graph.
\State Initialize $\mathcal U$ as the set of unmatched vertices still under consideration.
\State Initialize $\mathcal P:=\emptyset$ as the set of already processed unmatched representatives.
\State Initialize $\mathcal R:=\emptyset$ as the set of vertices scheduled for the final recoloring step.
\State Initialize the set of annotations describing local changes of $F$.

\While{$\mathcal U \neq \emptyset$}
\label{it:while1}
    \State Choose an unmatched vertex $x$ from $\mathcal U$

    \If{$x$ is in or can be switched into a contractible configuration}
        \State Switch $M$ along the $M$-alternating path.
        \State Update $F$ accordingly
        \State Remove $x$ from $\mathcal U$

    \ElsIf{$x$ is represented by Kempe-contractible case}
        \State Add $x$ to $\mathcal R$ and remove $x$ from $\mathcal U$

    \ElsIf{$x$ is represented by Case~2B or Case~4C and local resolution can be applied}
        \State Make appropriate local resolutions and annotate any contractions needed not implied by matching (in Case~4C) or any contractions that should not be performed (in Case~2B), update $F$ accordingly
        \State Remove $x$ from $\mathcal U$

    \Else{ ($x$ is represented by Case~2B or Case~4C and no local resolution can be applied)}
        \State Perform switch to another non processed vertex $v$ of degree less than $6$
        \State Update $F$ accordingly
        \State Remove $x$ from $\mathcal U$ and add $v$ to $\mathcal U$
    \EndIf

   \State Add $x$ to $\mathcal P$ (mark it as processed).
\EndWhile

\State Contract the components of $(V(G),F)$, respecting all annotations, and let $H:=H(G,F)$.
\State Make a wheel in $H$ around each vertex representing a contractible configuration
\State Compute a proper $4$-coloring of $H$
\State Lift this coloring of $H$ to the original graph

\For{each vertex $x \in \mathcal R$}

    \If{the lifted coloring is already conflict-free on $G_x$}
        \State Continue
    \EndIf
    
    \State Perform the Kempe-chain recoloring on $H$ as described in the proof of \Cref{lem:controlled}

    \If{the recoloring does not succeed}
        \State Conclude that the factor-critical component containing $x$ yields a forbidden gadget minor (\Cref{lem:controlled}, \Cref{lem:usecase_of_density})
        \State Discard the current contracted graph $H$
        \State Perform switch to another non processed vertex $v$ of degree less than $6$
        \State Update \(F\) accordingly
        \State add $v$ to $\mathcal U$ and remove $x$ from $\mathcal R$
        \State Go back to line \ref{it:while1}
    \EndIf
\EndFor

\State Reinsert the deleted handles and extend the coloring
\State \Return the resulting $4$-CFON coloring of $G$
\end{algorithmic}

Notice that if we try to switch to a non-processed vertex and are unable to do so, then we may conclude that the graph is not planar. Indeed, in that case all relevant vertices have already been processed but not resolved. Hence they satisfy condition c) of \Cref{lemma:case4C_full}, and therefore, by \Cref{lem:usecase_of_density}, all factor-critical components of $G_x$ contain a gadget.

\end{toappendix}

\section{Concluding remarks and questions}

We have shown that planar graphs are CFON $4$-colorable by lengthy technical case analysis, resolving an open question, and have given an efficient algorithm alongside.
The first natural question is how to simplify our analysis and, especially, streamline the flow of arguments.

Secondly it would be nice to improve the algorithm runtime a bit; perhaps by searching for the gadget obstructions first, which would eliminate the need for repeated coloring phase.

Another direction is to explore generalizations of our result.
More broadly, our method relies heavily on the interaction between planarity and the structure of factor-critical components. 
It is therefore natural to ask how far this approach extends beyond planar graphs. In particular, can \Cref{cor:lowdeg_free} be generalized to wider interesting graph classes?

From the complexity point of view, one can ask how difficult it is to CFON color a planar graph with $2$ or $3$ colors.
In \cite{AbelCF} the authors proved that \emph{partial} CFON coloring of planar graphs with $1$, $2$ or $3$ colors is \NP-hard, but it seems complicated to translate their result to the \emph{full} coloring case.
In the Appendix, we give a simple proof that CFON $3$-coloring of planar graphs is \NP-hard.

\begin{toappendix}

\begin{proposition}
The following problem is \NP-complete: Given a planar graph $G$, is there a CFON coloring of $G$ with $3$ colors?
\end{proposition}
\begin{proof}
Given a planar graph $H$, we construct $H_1$ by subdividing every edge of $H$ with a new vertex, and $H_2$ by adding a new leaf to every vertex of $V(H)$ in~$H_1$.
Then $H$ is properly $3$-colorable, if and only if $H_2$ is CFON $3$-colorable, which we prove now.

$\Rightarrow$: We take a proper $3$-coloring of $H$ and assign the same colors to the vertices of $V(H)$ in $H_2$.
Then we give color $1$ to each vertex of $V(H_1)\setminus V(H)$ in $H_2$, and color $2$ to every vertex of $V(H_2)\setminus V(H_1)$ in $H_2$.
This is clearly a CFON coloring.

$\Leftarrow$: Each vertex $x\in V(H_1)\setminus V(H)$ in $H_2$ can see a unique color in a coloring of $H_2$ if and only if the two neighbors of $x$ in $H_2$ have distinct colors.
Hence, if and only if the edge of $H$ which is subdivided by $x$ in $H_1$ receives two distinct colors.
Therefore, a CFON coloring of $H_2$ restricts to a proper coloring of~$H$.

The conclusion follows since proper $3$-colorability of planar graphs is \NP-complete.
\end{proof}

\end{toappendix}

\bibliography{refferences}

\end{document}